\documentclass[a4paper,reqno,12pt]{amsart}
\usepackage[T1]{fontenc}
\usepackage[utf8]{inputenc}
\usepackage[english]{babel}

\usepackage{amsmath}
\usepackage{amsthm}
\usepackage{amssymb}
\usepackage{mathtools}
\usepackage{amsmath}
\usepackage{amsfonts}
\usepackage{mathrsfs}
\usepackage{esint}
\usepackage{yfonts}
\usepackage{quoting}
\usepackage{graphics}
\usepackage{booktabs}
\usepackage{bm}
\usepackage{bbm}
\usepackage{lmodern}
\usepackage{stmaryrd}
\usepackage{caption}
\usepackage{enumitem}			
\usepackage{tikz}
\usepackage{bookmark}
\usepackage{pgfplots}
\usepgfplotslibrary{fillbetween}
\usepackage{orcidlink}

\usepackage{hyperref}
\hypersetup{
	colorlinks=true,
	linkcolor=blue,
	citecolor=blue,
	urlcolor=black,
	linktoc=all
}

\usepackage[margin=1.1in]{geometry}

\quotingsetup{font=small}

\theoremstyle{plain}

\theoremstyle{plain}
\newtheorem{theorem}{Theorem}[section]
\numberwithin{equation}{section}	

\theoremstyle{plain}
\newtheorem{lemma}[theorem]{Lemma}

\theoremstyle{plain}

\theoremstyle{definition}
\newtheorem{remark}{Remark}[section]

\DeclarePairedDelimiter{\abs}{\lvert}{\rvert}
\DeclarePairedDelimiter{\norma}{\lVert}{\rVert}

\renewcommand{\epsilon}{\varepsilon}

\DeclareMathOperator{\defi}{def} 

\newcommand{\numberset}{\mathbb}
\newcommand{\N}{\numberset{N}}			
\newcommand{\R}{\numberset{R}}			
\newcommand{\sfera}{\numberset{S}}		
\newcommand{\Haus}{\mathcal{H}}			
\newcommand{\loc}{{\rm loc}}
\newcommand{\df}{\delta_{f}(u,\kappa)}

\def\Xint#1{\mathchoice
	{\XXint\displaystyle\textstyle{#1}}%
	{\XXint\textstyle\scriptstyle{#1}}%
	{\XXint\scriptstyle\scriptscriptstyle{#1}}%
	{\XXint\scriptscriptstyle\scriptscriptstyle{#1}}%
	\!\int}
\def\XXint#1#2#3{{\setbox0=\hbox{$#1{#2#3}{\int}$ }
		\vcenter{\hbox{$#2#3$ }}\kern-.6\wd0}}

\def\dashint{\Xint-}

\usepackage{enumitem}

\begin{document}
	
		\title[Quantitative symmetry for semilinear equations in~$\R^n$]{A quantitative study of radial symmetry for solutions to semilinear equations in~$\R^n$}
		
		\author[Giulio Ciraolo]{Giulio Ciraolo \orcidlink{0000-0002-9308-0147}}
		\address[]{Giulio Ciraolo. Dipartimento di Matematica ‘Federigo Enriques’, Università degli Studi di Milano, Via Cesare Saldini 50, 20133, Milan, Italy}
		\email{giulio.ciraolo@unimi.it}
 		
 		\author[Matteo Cozzi]{Matteo Cozzi \orcidlink{0000-0001-6105-692X}}
 		\address[]{Matteo Cozzi. Dipartimento di Matematica ‘Federigo Enriques’, Università degli Studi di Milano, Via Cesare Saldini 50, 20133, Milan, Italy}
 		\email{matteo.cozzi@unimi.it}
		
		\author[Michele Gatti]{Michele Gatti \orcidlink{0009-0002-6686-9684}} 
		\address[]{Michele Gatti. Dipartimento di Matematica ‘Federigo Enriques’, Università degli Studi di Milano, Via Cesare Saldini 50, 20133, Milan, Italy}
		\email{michele.gatti1@unimi.it}
		
		\date{\today}
		
		\begin{abstract}
			A celebrated result by Gidas, Ni \& Nirenberg asserts that positive classical solutions, decaying at infinity, to semilinear equations~$\Delta u +f(u)=0$ in~$\R^n$ must be radial and radially  decreasing. In this paper, we consider both energy solutions in~$\mathcal{D}^{1,2}(\R^n)$ and non-energy local weak solutions to small perturbations of these equations, and study its quantitative stability counterpart.
			
			To the best of our knowledge, the present work provides the first quantitative stability result for non-energy solutions to semilinear equations involving the Laplacian, even for the critical nonlinearity.
		\end{abstract}
	
		\maketitle
		
	
	\section{Introduction}
	\label{sec:intro}

	\noindent
	In the seminal paper~\cite{gnn}, Gidas, Ni \& Nirenberg exploited the method of moving planes to prove the radial symmetry and monotonicity of positive solutions to semilinear equations such as 
	\begin{equation}
	\label{eq:eq_semil_kx}
		\Delta u + \kappa(x) f(u) = 0 \quad \mbox{in } \R^n ,
	\end{equation}
	where~$n \ge 3$, the nonlinearity~$f$ satisfies some regularity and growth assumptions, the solution~$u$ decays at infinity with a rate connected with the behavior of~$f$, and~$\kappa$ is either a positive rotationally symmetric and strictly decreasing function or a positive constant -- in the latter case, the symmetry of~$u$ naturally has to be understood up to translations. See also~\cite{li-ni-asymp-I,li-ni-asymp-II,li-ni-mak} for further related results.

	The main goal of this paper is to address a quantitative stability result for this problem. Roughly speaking, we assume that~$\kappa$ is close to be a constant and show that, under some suitable assumptions, the solution is almost radial. We will also provide quantitative estimates for this result, where the proximity to a radial solution is quantified in terms of some deficit describing the closeness of~$\kappa$ to a constant. 

	In order to properly state the results and clarify the motivations, we divide the introduction into three different subsections. In the first subsection we describe the state of the art for the critical equation, which serves as a motivation for the present manuscript. Then, in the remaining two subsections, we state our main contributions.

	\subsection{Quantitative results for the critical Laplace equation}
	As far as we know, quantitative studies for symmetry of solutions to semilinear equations~\eqref{eq:eq_semil_kx} have been carried  out only in the case of the critical Laplace equation with~$f(u)=u^{2^*-1}$, i.e.\ when~\eqref{eq:eq_semil_kx} takes the form
	\begin{equation}
	\label{eq:intro_critical}
		\Delta u + \kappa_0 u^{2^*-1} = 0 \quad \mbox{in } \R^n,
	\end{equation}
	for some fixed~$\kappa_0>0$, in the recent works~\cite{cfm,dsw,fg}. It is well-known that~\eqref{eq:intro_critical} arises as the Euler-Lagrange equation associated to the problem of finding the optimal constant in the Sobolev inequality, namely
	\begin{equation}
	\label{eq:Sob_quot}
		S \coloneqq \min_{u \in \mathcal{D}^{1,2}(\R^n) \setminus \{ 0 \}} \frac{ \norma*{\nabla u}_{L^2(\R^n)}}{\norma*{u}_{L^{2^\ast}\!(\R^n)}},
	\end{equation} 
	as well as in the study of Yamabe problem, in particular for the prescribed scalar curvature problem on the sphere. Here
	\begin{equation*}
		\mathcal{D}^{1,2}(\R^n) \coloneqq \left\{u \in L^{2^\ast}\!(\R^n) \,\big\lvert\, \nabla u \in L^2(\R^n) \right\}.
	\end{equation*}
	We notice that~\eqref{eq:intro_critical} is invariant under scaling and translations and, assuming~$\kappa_0=1$, the only positive solutions are given by the so-called \emph{Talenti bubbles}
	\begin{equation*}
		\label{eq:talentiane}
		u(x) = U[z,\lambda] (x) \coloneqq \left( \frac{\lambda \sqrt{n(n-2)}}{1+\lambda^2 \,\abs*{x-z}^2} \right)^{\!\!\frac{n-2}{2}} \!,
	\end{equation*}
	with~$n \geq 3$, where~$\lambda>0$ and~$z \in \R^n$ are the scaling and translation parameters, respectively. This classification result has been obtained in~\cite{au,gnn,oba,tal} for minimizers of~\eqref{eq:Sob_quot} or by assuming a suitable decay at infinity, namely
	\begin{equation*}
		u(x) = O \left(\abs{x}^{2-n}\right) \quad \mbox{as } \abs{x} \to \infty \,.
	\end{equation*}
	Remarkably, these variational and decay assumptions can be removed as shown by Caffarelli, Gidas \& Spruck~\cite{cgs} and Chen \& Li~\cite{cl}, and hence Talenti bubbles are the only positive solutions to~\eqref{eq:intro_critical} without any further requirement on the solution.
	
	Under the assumption~$u \in \mathcal{D}^{1,2}(\R^n)$, a stability version of this result was proved by Struwe~\cite{struwe}, who showed that if~$u$ \emph{almost} solves~\eqref{eq:intro_critical} then~$u$ is \emph{close} to the sum of Talenti bubbles. This phenomenon is called \emph{bubbling}.

	A quantitative version of Struwe's theorem~\cite{struwe} was later studied by Ciraolo, Figalli \& Maggi~\cite{cfm} under the a priori assumption 
	\begin{equation}
	\label{eq:bound-energia}
		\frac{1}{2} S^n \leq \int_{\R^n} \,\abs*{\nabla u}^2 \, dx \leq \frac{3}{2} S^n,
	\end{equation}
	which entails that the energy of the solution is approximately that of one bubble, thus preventing bubbling. A sharp quantitative study of bubbling was subsequently obtained by Figalli \& Glaudo~\cite{fg} and Deng, Sun \& Wei~\cite{dsw}, still by assuming that~$u\in\mathcal{D}^{1,2}(\R^n)$ but dropping~\eqref{eq:bound-energia} and considering, instead, its counterpart for~$m\in \N$ bubbles.

	In these papers, the stability result is understood in the sense of~$\mathcal{D}^{1,2}(\R^n)$, thus providing quantitative estimates on the $L^2$-norm of the gradient of the difference between the solution and a suitable sum of bubbles in terms of some deficit~$\delta(u)$, which measures how close~$\kappa$ is to being constant.

	It is interesting to understand how~$\delta(u)$ is defined. Let us consider~\eqref{eq:eq_semil_kx} in case of the critical nonlinearity, i.e.
	\begin{equation}
	\label{eq:eq_semil_crti}
		\Delta u + \kappa(x) u^{2^\ast-1} = 0 \quad \mbox{in } \R^n,
	\end{equation}
	or, in the weak formulation,
	\begin{equation}
	\label{eq:criti1}
	- \int_{\R^n} \nabla u \cdot \nabla \phi \, dx + \int_{\R^n}\kappa(x) u^{2^*-1} \phi \, dx = 0 \quad \mbox{for all } \phi \in \mathcal{D}^{1,2}(\R^n).
	\end{equation}
	If~$\kappa \equiv \kappa_0$ is constant, then~$u\in \mathcal{D}^{1,2}(\R^n)$ must be a Talenti bubble and it holds
	\begin{equation}
	\label{eq:ref-const}
		\kappa_0=\kappa_0(u) \coloneqq \frac{\int_{\R^n} \kappa u^{2^\ast} dx}{\int_{\R^n} u^{2^\ast} dx} = \frac{\int_{\R^n} \,\abs{\nabla u}^2 \, dx}{\int_{\R^n} u^{2^\ast} dx} \,.
	\end{equation}
	Hence, it is convenient to see~\eqref{eq:criti1} as
	\begin{equation*}
	\label{eq:criti2}
		- \int_{\R^n} \nabla u \cdot \nabla \phi \, dx + \kappa_0 \int_{\R^n} u^{2^*-1} \phi \, dx = - \int_{\R^n} (\kappa(x)-\kappa_0) \, u^{2^*-1} \phi \, dx \quad \mbox{for all } \phi \in \mathcal{D}^{1,2}(\R^n)
	\end{equation*}
	and, since~$u,\phi  \in \mathcal{D}^{1,2}(\R^n)$, it is natural to measure how much~$\kappa$ is close to~$\kappa_0$ in terms of the deficit
	\begin{equation}
	\label{eq:def_cfm}
		\delta(u) \coloneqq \norma*{\left(\kappa-\kappa_0\right)u^{2^*-1}}_{L^{(2^\ast)'}(\R^n)}.
	\end{equation}
	See also \cite{cfm} for a more precise explanation. Hence, the deficit used in~\cite{cfm,dsw,fg} for the critical Laplace equation directly follows from the energy assumption~$u \in \mathcal{D}^{1,2}(\R^n)$. More generally, the quantitative analyses performed in~\cite{cfm,dsw,fg} not only rely on the assumption that the solution belongs to the energy space~$\mathcal{D}^{1,2}(\R^n)$, but also strongly exploit the structure of the equation~\eqref{eq:eq_semil_crti} and the knowledge of the explicit solutions~$U[z,\lambda]$. 

	Thus, the results in~\cite{cfm,dsw,fg} can be seen as a quantitative version of the Gidas-Ni-Nirenberg theorem for energy solutions of the critical Laplace equation. However, this approach does not seem to be applicable neither when one drops the energy assumption~$u \in \mathcal{D}^{1,2}(\R^n)$ -- as done in~\cite{cgs} and~\cite{cl} -- nor to equations with a more general nonlinearity as the one in~\eqref{eq:eq_semil_kx}.

	The goal of this paper is to tackle these two problems and provide quantitative stability estimates. This will be done by exploiting a careful quantitative adaptation of the method of moving planes -- which seems the most natural approach to be adopted under such a generality, also for the proof of the exact symmetry result.

	The study of quantitative stability of symmetry via the method of moving planes is nowadays well-established. Starting from the paper of Aftalion, Busca \& Reichel~\cite{abr}, quantitative analyses of this method have been successfully applied to many different situations as in~\cite{cicopepo,cdppv,cfmnov,cirmagsak,cirvez}. In particular, we mention the recent work of Ciraolo, Cozzi, Perugini \& Pollastro~\cite{cicopepo} which contains a quantitative analysis of the Gidas-Ni-Nirenberg theorem for semilinear equations on a ball with Dirichlet boundary conditions. Nevertheless, the approach in the present paper and the one in~\cite{cicopepo} differ in many aspects, as we will explain in the remainder of the introduction.\newline
	
	We now describe our main results.
	
	
	\subsection{Quantitative results for energy solutions to semilinear equations} Our first main result deals with general semilinear equations. Let~$u \in \mathcal{D}^{1,2}(\R^n)$ be a non-negative, non-trivial weak solution to
	\begin{equation}
		\label{eq:mainprob-f}
		\Delta u + \kappa(x) f(u) = 0 \quad \textmd{in } \R^n,
	\end{equation}
	with~$f:[0,+\infty) \to \R$,~$f \in C^{0,1}_{\loc}\left([0,+\infty)\right)$, and~$\kappa \in L^\infty(\R^n)$ a non-negative function. In partial analogy to~\eqref{eq:def_cfm}, we measure the proximity of~$\kappa$ to being constant through the deficit
	\begin{equation} \label{eq:deficit2}
		\df \coloneqq \inf_{t \in \R} \, \norma*{\left(\kappa-t\right)f(u)}_{L^{(2^\ast)'}(\R^n)}.
	\end{equation}
	Unlike in~\cite{cfm}, where an energy estimate is provided for the distance between the solution and a reference solution of the symmetric configuration  -- Talenti bubbles in that case --, we will instead measure directly how far the solution deviates from the radial symmetry. For this reason, an a priori upper bound on the solution is needed in order to have uniform estimates in terms of the deficit~\eqref{eq:deficit2}, hence we assume that
	\begin{equation}
	\label{eq:C0_def}
		\norma*{u}_{L^\infty(\R^n)} \leq C_0 ,
	\end{equation}
	for some~$C_0 \geq 1$.
	
	Concerning the nonlinearity, we suppose that~$f$ is subcritical in the sense that there exists some constant~$f_0>0$ such that
	\begin{equation}
	\label{eq:ip-f-sub}
		\abs*{f(u)} \leq f_0 u^p \quad \mbox{for every } 0 \le u \le C_0,
	\end{equation}
	where, from now on, we adopt the notation
	\begin{equation*}
		p \coloneqq 2^\ast - 1 = \frac{n + 2}{n - 2}.
	\end{equation*}
	Note that, in view of this condition and~$\kappa \in L^\infty(\R^n)$,~$\df$ is a well-defined real number. Moreover, we further prescribe the behavior of~$f$ by requiring that there exists a constant~$L>0$ such that
	\begin{equation}
		\label{eq:ip-f-tipolip}
		\frac{f(u_2)-f(u_1)}{u_2-u_1} \leq L u_2^{p - 1} \quad \text{for every } 0\leq u_1<u_2\leq C_0.
	\end{equation}
	Here, the constant~$C_0$ is the one appearing in~\eqref{eq:C0_def}. This condition can be found in~\cite{dam-ram} and can be traced back to~\cite{c-li}.
	
	Under these structural assumptions we are able to prove the following result.

	\begin{theorem}
	\label{th:maintheorem-f}
	Let~$n\geq 3$ be an integer,~$\kappa \in L^\infty(\R^n)$ be a non-negative function, and~$f \in C^{0,1}_{\loc}\left([0,+\infty)\right)$ be satisfying~\eqref{eq:ip-f-sub}-\eqref{eq:ip-f-tipolip}. Let~$u \in \mathcal{D}^{1,2}(\R^n)$ be a non-negative, non-trivial weak solution to~\eqref{eq:mainprob-f} satisfying
	\begin{equation}
		\label{eq:decadimento}
		u(x) \leq \frac{C_0}{1+\abs{x}^{n-2}} \quad \mbox{for a.e.~} x \in \R^n,
	\end{equation}
	for some~$C_0 \ge 1$.
	
	There exist a point~$\mathcal{O} \in \R^n$ and a large constant~$C \ge 1$ such that, if~$\df \in (0,1)$, then
	\begin{equation} \label{eq:mainLinftyest}
		\abs*{u(x)-u(y)} \leq C \,\abs{\log\df}^{-\vartheta}
	\end{equation}
	for every~$x,y \in \R^n$ satisfying~$\abs*{x-\mathcal{O}}=\abs*{y-\mathcal{O}}$, where~$\df$ is given by~\eqref{eq:deficit2} and~$\vartheta = \frac{n-2}{12}$. Moreover, if~$u_\Theta$ denotes any rotation of~$u$ with center~$\mathcal{O}$, we also have
	\begin{equation} \label{eq:mainD12est_thm1.1}
		\norma*{u-u_\Theta}_{\mathcal{D}^{1,2}(\R^n)} 
		\leq C \,\abs{\log\df}^{-\vartheta}.
	\end{equation}
	Finally, if we also suppose that~$\kappa \in C^1(\R^n)$ with $\nabla\kappa \in L^\infty(\R^n)$, then it holds
	\begin{equation}
		\label{eq:quasi-radmon}
		\partial_r u(x) \leq C \left(\abs*{\log\df}^{-\vartheta}  + \norma*{\nabla\kappa}_{L^\infty(\R^n)}\right) \!,
	\end{equation}
	for every~$x \in \R^n \setminus \{\mathcal{O}\}$, where~$\partial_r$ denotes the radial derivative with respect to~$\mathcal{O}$. The constant~$C$ depends only on~$n$,~$C_0$,~$f_0$,~$L$,~$\| f \|_{C^{0, 1}(\R^n)}$, and~$\| \kappa \|_{L^\infty(\R^n)}$.
	\end{theorem}

	Note that the non-negativity of~$u$ along with assumption~\eqref{eq:decadimento} imply in particular the validity of the~$L^\infty$ bound~\eqref{eq:C0_def}.

	Theorem~\ref{th:maintheorem-f} provides a quantitative stability result for~\eqref{eq:mainprob-f} with a logarithmic-type dependence on~$\df$, and it quantitatively describes how close~$u$ is to being symmetric. When~$\df \to 0$, this type of dependence is consistent with the symmetry of solutions to~\eqref{eq:mainprob-f} with~$\kappa$ constant. Moreover, when~$\kappa$ is constant also the second term on the right-hand side of~\eqref{eq:quasi-radmon} vanishes, coherently with the radial monotonicity of these solutions.
	
	Theorem~\ref{th:maintheorem-f} also applies to the equation with critical nonlinearity~\eqref{eq:eq_semil_crti}, but in this case it provides weaker results in comparison to those of~\cite{cfm,dsw,fg}. Nevertheless, we stress that the point of view of these works is inherently different from ours, which is focused on approximate radial symmetry. While it is true that proximity to a single Talenti bubble can be deduced from Theorem~\ref{th:maintheorem-f}, the presence of two or more bubbles does not seem to be detectable via an~$L^\infty$-estimate such as~\eqref{eq:mainLinftyest}. Indeed, a sum of two bubbles solves~\eqref{eq:eq_semil_crti} with a function~$\kappa$ having deficit~\eqref{eq:def_cfm} that can be made arbitrarily small provided the centers are taken far apart. This configuration, however, would not be almost radially symmetric -- in the sense of~\eqref{eq:mainLinftyest} -- and, coherently, would fail to satisfy~\eqref{eq:decadimento}.

	We now comment more extensively on assumption~\eqref{eq:decadimento}. In the case of the critical nonlinearity, it is crucial to prevent bubbling when an energy bound analogous to~\eqref{eq:bound-energia} is not imposed. Moreover, since in this case the equation is scaling invariant, condition~\eqref{eq:decadimento} plays the role of fixing a scale and the constant~$C_0$ provides an upper bound for the~$L^\infty$-norm of~$u$. In addition,~\eqref{eq:decadimento} is quite natural if we suppose that~$\kappa \in L^\infty(\R^n)$ and that~\eqref{eq:ip-f-sub} is in force. Indeed, in Theorem 1.1 of~\cite{vet}, V\'etois proved that any solution to~\eqref{eq:mainprob-f} behaves like
	\begin{equation*}
		F(x) \coloneqq \frac{1}{1+\abs{x}^{n-2}},
	\end{equation*}
	up to a multiplicative constant. However, since the equation is invariant under translations, the constant provided therein certainly depends on the function~$u$ -- actually, on explicit quantities related to~$u$, see~\cite[Remark 4.1]{vet}. This suggests also that we cannot aim to prove quasi-symmetry with respect to a fixed point of the space. Instead, the existence of such an approximate center~$\mathcal{O}$ should be a result of the proof itself. Moreover, in view of~\eqref{eq:decadimento}, we expect to show that~$\mathcal{O}$ is close to the origin -- see~\eqref{eq:Oinlambdacube} below -- and is a point of approximate maximum by~\eqref{eq:quasi-radmon}.
	
	This last fact highlights a major difference between our approach and that of~\cite{cicopepo}. First, in the case of a ball considered in~\cite{cicopepo}, there exists a geometrically preferred point -- the origin -- and the authors aimed to reach it while carrying out the moving plane procedure. In addition, one can take advantage of their bound from below on the~$L^\infty$-norm -- see Equation~(1.3) in~\cite{cicopepo} --, the Dirichlet boundary datum and the superharmonicity of the solution to get a pointwise bound from below for the solution.

	In our case, when arguing by contradiction, the absurd should come from the fact that~$u$ is not almost equal to its reflection, in a suitable quantitative sense -- see~\eqref{eq:stima-assurda}. Additionally, since we will not use any pointwise bound from below, we have to deal with energy or mass estimates. Note also that, in our case, such a lower bound would be a simple consequence of the non-triviality of the solution together with the decay estimate~\eqref{eq:decadimento}, without any further assumptions.
	
	Moreover, a relevant difference between the present paper and~\cite{cicopepo} lies in the way we apply the method of moving planes: while in~\cite{cicopepo} the moving planes method is performed in a \emph{pointwise way}, here we opt for an integral approach in the spirit of~\cite{dam-ram,sciu}. After comparing these two approaches, we believe the second one to be simpler for deriving quantitative results and we expect it to be more suitable for a potential future extension of the present work to the~$p$-Laplace operator.

	Finally, we comment on our choice of deficit~\eqref{eq:deficit2}. Its design is clearly based on the deficit~\eqref{eq:ref-const}-\eqref{eq:def_cfm} adopted in~\cite{cfm}. However, it is more flexible, since it does not require the specification of a reference constant~$\kappa_0$. Note that another possibility, more adherent to~\eqref{eq:ref-const}-\eqref{eq:def_cfm}, would be to consider the deficit
	\begin{equation*}
		\tilde{\delta}_f(u, \kappa) \coloneqq \norma*{\left(\kappa-\kappa_0\right)f(u)}_{L^{(2^\ast)'}(\R^n)}, \quad \mbox{with } \kappa_0 \coloneqq \frac{\int_{\R^n} \kappa f(u) u \, dx}{\int_{\R^n} f(u) u \, dx} = \frac{\int_{\R^n} |\nabla u|^2 \, dx}{\int_{\R^n} f(u) u \, dx},
	\end{equation*}
	at least when~$f$ is non-negative -- so that the denominator of the fraction defining~$\kappa_0$ is well-defined. Nevertheless, when comparing this alternative deficit with~\eqref{eq:deficit2}, we observe that the latter is stronger than the former, meaning that~$\delta_f(u, \kappa) \le \tilde{\delta}_f(u, \kappa)$. The deficit~\eqref{eq:deficit2} is also stronger than the pointwise choices of~\cite{cicopepo}, since, by appropriately choosing~$t$ in~\eqref{eq:deficit2}, one sees that
	\begin{equation*}
		\delta_f(u, \kappa) \le \norma*{f(u)}_{L^{(2^\ast)'}(\R^n)} \, \mathrm{osc}_{\R^n} \kappa \, \leq C \,\mathrm{osc}_{\R^n} \kappa,
	\end{equation*}
	thanks to hypotheses~\eqref{eq:ip-f-sub} and~\eqref{eq:decadimento} on~$f$ and~$u$. We point out however that other, more precise deficits were considered in~\cite{cicopepo} -- see, e.g.,~Equation~(1.6) there -- in order to better capture functions~$\kappa$ which are radially symmetric and non-increasing -- hypotheses under which, in the context of~\cite{cicopepo}, the radial symmetry of~$u$ can still be inferred. In the setting of the present paper such choices of deficits would not be fully appropriate, since Theorem~\ref{th:maintheorem-f} aims to understand the quasi-symmetry of the solution with respect to an a priori unspecified center~$\mathcal{O}$. 

	
	\subsection{Quantitative results for non-energy solutions to semilinear equations} As it follows from~\cite{cgs}, any positive solution to~\eqref{eq:intro_critical} with $\kappa_0=1$ is a Talenti bubble and hence, a posteriori, it is an energy solution which belongs to~$\mathcal D^{1,2}(\R^n)$. If we do not consider~$u \in \mathcal D^{1,2}(\R^n)$ as an assumption, then a quantitative analysis of~\eqref{eq:eq_semil_crti} cannot start with an argument like the ones in~\cite{cfm,dsw,fg} and already described in the previous subsection. 

	Moreover, for~$\kappa$ radially symmetric and non-increasing, non-constant, and bounded from below by a positive constant, Ding \& Ni~\cite{dn} proved that there exist non-energy positive solutions of~\eqref{eq:eq_semil_crti} which decay at infinity as~$\abs{x}^{-\frac{n-2}{2}}$. Consequently, any small perturbation of~$\kappa$ in the~$L^\infty$-norm may lead to the appearance of non-energy solutions. This particular example reveals that the condition~\eqref{eq:decadimento} cannot hold in general for classical solutions to~\eqref{eq:eq_semil_crti} and, more generally, for solutions to~\eqref{eq:mainprob-f}.
	
	In light of these remarks, we now consider a non-negative local weak solution~$u \in W^{1,2}_{\loc}(\R^n) \cap L^\infty(\R^n)$ to~\eqref{eq:mainprob-f} and we define the deficit  
	\begin{equation*}
	\label{eq:def-essosc}
		\defi(\kappa) \coloneqq \underset{x\in\R^n}{\mathrm{osc}} \,\kappa(x).
	\end{equation*}
	As before, we assume that~$u$ fulfills the~$L^\infty$-bound~\eqref{eq:C0_def}, for some constant~$C_0 \ge 1$. On top of it, we require~$u$ to satisfy the decay condition 
	\begin{equation}
	\label{eq:u-below}
		\frac{1}{C_0} \frac{\defi(\kappa)^{\sigma}}{\abs*{x}^{n-2-\mu}} \leq u(x) \leq  \frac{C_0}{\defi(\kappa)^{\sigma} \,\abs*{x}^\nu} \quad \mbox{for a.e.~} x \in \R^n \setminus B_{R_0}(0),
	\end{equation}
	for some small constants~$\mu, \nu \in \left(0,\frac{n-2}{2}\right]$,~$\sigma>0$, and some large radius~$R_0 \geq 1$.
	
	Concerning the nonlinearity~$f$, we require it to be subcritical in the sense of~\eqref{eq:ip-f-sub} and that
	\begin{equation}
	\label{eq:ip-f-lip-monot}
		0 \leq \frac{f(u_2) - f(u_1)}{u_2 - u_1}  \le L u_2^{p - 1} \quad \mbox{for every } 0 \leq u_1 < u_2 \leq C_0.
	\end{equation}
	Note, in particular, that this condition implies that~$f$ is non-decreasing. Finally, we need that
	\begin{equation}
	\label{eq:f/up}
		\mbox{the function } u \mapsto \frac{f(u)}{u^{p}} \mbox{ is non-increasing in } (0, +\infty).
	\end{equation}
	This type of assumption can be found, for instance, in~\cite{li-zhang}, where a Liouville-type result for~\eqref{eq:eq_semil_crti} with~$\kappa \equiv 1$ is proven.
	
	In this setting, we have the following result. 

	\begin{theorem}
		\label{th:maintheorem-kelvin}
		Let~$n\geq 3$ be an integer,~$\kappa \in L^\infty(\R^n)$ be a non-negative function, and~$f \in C_\loc^{0,1}([0, +\infty])$ be satisfying~\eqref{eq:ip-f-sub} and~\eqref{eq:ip-f-lip-monot}-\eqref{eq:f/up}. Let~$u \in W^{1,2}_{\loc}(\R^n) \cap L^\infty(\R^n)$ be a non-negative local weak solution to~\eqref{eq:mainprob-f} satisfying~\eqref{eq:C0_def} and~\eqref{eq:u-below}.

		There exist two small constants~$\sigma_0, \vartheta \in (0, 1)$ and a large~$C \ge 1$ such that if~$\sigma \le \sigma_0$ then
		\begin{equation}
			\label{eq:alm_sym_thm1}
			\abs*{u(x)-u(y)} \leq C \defi(\kappa)^\vartheta
		\end{equation}
		for every~$x,y \in \R^n$ with~$\abs*{x}=\abs*{y}$. Moreover, if~$\kappa \in C^{0,\tau}(\R^n)$ for some~$\tau \in (0,1)$ and~$u_\Theta$ denotes any rotation of~$u$ centered at the origin, then there exists a large constant~$C' \ge 1$ such that
		\begin{equation}
			\label{eq:alm_symC2_thm2}
			\norma*{u-u_\Theta}_{C^2(\R^n)} \leq C' \! \left(1+\left[\kappa\right]_{C^{0,\tau}(\R^n)}\right) \defi(\kappa)^{\vartheta'},
		\end{equation}
		with~$\vartheta' = \frac{\vartheta}{2}$. The constant~$C$ depends only on~$n$,~$C_0$,~$L$,~$f_0$,~$\norma*{\kappa}_{L^\infty(\R^n)}$,~$\nu$, and~$R_0$, the exponent~$\vartheta$ depends also on~$\mu$, while~$C'$ depends on~$\tau$ as well.
	\end{theorem}
	
	Theorem~\ref{th:maintheorem-kelvin} establishes the quasi-symmetry of \emph{non-energy} solutions of~\eqref{eq:mainprob-f} -- the fact that~$u$ does not belong to~$\mathcal{D}^{1,2}(\R^n)$ follows from the lower bound in~\eqref{eq:u-below} and Theorem~1.1 of~\cite{vet}. For this reason, our result is new even for the perturbed critical equation~\eqref{eq:eq_semil_crti} -- in particular, the estimates of~\cite{cfm,dsw,fg} do not apply.
	
	The proof of Theorem~\ref{th:maintheorem-kelvin} is carried out by considering the Kelvin transform~$v$ of the solution~$u$ and taking advantage of the left-most inequality in~\eqref{eq:u-below} to deduce that~$v$ blows up at the origin and, as a result, the positivity of an appropriate function in a small region -- see~\eqref{eq:posballforvlambda-v}. From there on, positivity is then expanded by means of standard tools, such as the weak Harnack inequality. Note that the use of a pointwise lower bound to obtain the inceptive positivity makes our proof more in the spirit of Theorem~1.2 in~\cite{cicopepo} and, especially, Theorem~1.1 in~\cite{OSV20}. We also point out that a drawback of this method is that no almost monotonicity result seems to be  obtainable from it, due to the use of the Kelvin transform.

	We now comment more thoroughly on hypotheses~\eqref{eq:C0_def} and~\eqref{eq:u-below}. The quantification of the boundedness of~$u$ given by~$C_0$ in~\eqref{eq:C0_def} is needed, as for Theorem~\ref{th:maintheorem-f}, to fix a scale for the solution~$u$ -- notice that, for some choices of~$f$ and~$\kappa$, equation~\eqref{eq:mainprob-f} may possess some scaling-invariance features. In regards to~\eqref{eq:u-below}, as observed earlier its left-most inequality gives a quantification of the fact that~$u$ does not belong to the energy space~$\mathcal{D}^{1,2}(\R^n)$, whereas its right-most inequality is a technical decay assumption needed only to complete the final step of the proof. Note also that~\eqref{eq:u-below} fixes, in a certain sense, a center of approximate symmetry: as a result, statements~\eqref{eq:alm_sym_thm1} and~\eqref{eq:alm_symC2_thm2} provide symmetry with respect to the origin. Finally, we stress that the presence of~$\defi(\kappa)$ in~\eqref{eq:u-below} makes the inequalities less and less significant as~$\defi(\kappa) \rightarrow 0$, coherently with the exact symmetry results of~\cite{cgs,cl} for~$\defi(\kappa) = 0$.
	
	We conclude by mentioning that it should be possible to remove the monotonicity assumption expressed by the left-most inequality of~\eqref{eq:ip-f-lip-monot}. However, this would likely come at the price of having a worse, logarithmic-type dependence on~$\defi(\kappa)$ in estimates~\eqref{eq:alm_sym_thm1} and~\eqref{eq:alm_symC2_thm2} -- see Remark~\ref{rem:monot-f} at the end of Section~\ref{sec:proof3} for more details on this.
	
	
	\section{Proof of Theorem~\ref{th:maintheorem-f}}
	\label{sec:proof1}
	
	\renewcommand\thesubsection{\bfseries Step \arabic{subsection}}
	
	\noindent
	We first show that any non-trivial solution to~\eqref{eq:mainprob-f} enjoys a quantitative lower bound for the mass. This is the content of the following result, which is essentially Lemma~2.3 in~\cite{vet}.
	
	\begin{lemma}
		Suppose that~$n \geq 3$ and~$u \in \mathcal{D}^{1,2}(\R^n)$ is a non-negative, non-trivial weak solution to~\eqref{eq:mainprob-f} with~$f:[0,+\infty) \to \R$ satisfying~\eqref{eq:ip-f-sub} and~$\kappa \in L^{\infty}(\R^n)$. Then, there exists a constant $\mathcal{M}>0$, depending only on~$n$,~$f_0$, and~$\norma{\kappa}_{L^\infty(\R^n)}$, such that
		\begin{equation}
			\label{eq:boundmassa}
			\int_{\R^n} u^{2^\ast} dx \geq \mathcal{M}.
		\end{equation}
	\end{lemma}
	\begin{proof}
		Since~$u$ is non-trivial, by Sobolev inequality, testing~\eqref{eq:mainprob-f} with~$u$, and hypothesis~\eqref{eq:ip-f-sub}, we get
		\begin{equation*}
			\begin{split}
				0 &< \left(\int_{\R^n} u^{2^\ast} dx\right)^{\!\!\frac{2}{2^\ast}} \leq S^{-2} \int_{\R^n} \,\abs*{\nabla u}^2 \, dx = S^{-2}
				\int_{\R^n} \kappa f(u) u \, dx \\
				&\leq S^{-2} f_0 \norma{\kappa}_{L^\infty(\R^n)} \int_{\R^n} u^{2^\ast} dx.
			\end{split}
		\end{equation*}
		The conclusion easily follows since~$2^\ast>2$.
	\end{proof}

	In the remainder of this section, we will refer to a constant as \textit{universal} if it depends only on~$n$,~$C_0$,~$L$,~$f_0$,~$\norma*{\kappa}_{L^\infty(\R^n)}$, and~$\norma*{f}_{C^{0,1}\left([0,C_0]\right)}$.

	We are now ready to prove our first main result.
	
	
	\subsection{Preliminary observations.} \label{step:prel-obs}
	
	First of all, standard elliptic regularity yields that~$u$ is of class~$C^{1,\tau}_{\loc}$ for every~$\tau \in (0,1)$ and, exploiting also assumption~\eqref{eq:decadimento} on the decay of~$u$ and~\eqref{eq:ip-f-sub}, we can deduce a
	natural bound for the decay of its gradient, that is
	\begin{equation} \label{eq:decadgrad}
		\abs*{\nabla u(x)} \leq \frac{C_1}{1 + |x|^{n - 1}} \quad \mbox{for all } x \in \R^n,
	\end{equation}
	for some constant~$C_1 \ge 1$ depending only on~$n$,~$C_0$,~$f_0$, and~$\norma*{\kappa}_{L^\infty(\R^n)}$. In particular, from this and~\eqref{eq:decadimento} it follows that
	\begin{equation} \label{eq:C1boundonu}
		\norma*{u}_{C^1(\R^n)} \le C_1,
	\end{equation}
	up to taking a larger~$C_1$.
	
	Secondly, we observe that it suffices to prove Theorem~\ref{th:maintheorem-f} when~$\df$ is smaller than a universal constant~$\gamma \in \left( 0, \frac{1}{2} \right]$. Indeed, when~$\df > \gamma$, by~\eqref{eq:decadgrad} and~\eqref{eq:C1boundonu}
	\begin{equation*}
		\abs*{u(x)-u(y)} \leq \abs*{u(x)} + \abs*{u(y)} \leq 2C_1 \leq \frac{2C_1}{\abs{\log\gamma}^{-\vartheta}} \,\abs{\log\df}^{-\vartheta}
	\end{equation*}
	for every~$x,y \in \R^n$, 
	\begin{equation*}
		\norma*{u-u_\Theta}^2_{\mathcal{D}^{1,2}(\R^n)} 
		\leq 4C_1^2 \int_{\R^n} \frac{dx}{\left(1 + |x|^{n-1}\right)^2} \leq \frac{4C_1^2C_n}{\abs{\log\gamma}^{-\vartheta}} \,\abs{\log\df}^{-\vartheta},
	\end{equation*}
	and
	\begin{equation*}\partial_r u(x) \leq \abs*{\nabla u (x)} \leq C_1 
		\leq \frac{C_1}{\abs{\log\gamma}^{-\vartheta}} \left(\abs*{\log\df}^{-\vartheta}  + \norma*{\nabla\kappa}_{L^\infty(\R^n)}\right) \!.
	\end{equation*}
	Therefore, all the inequalities stated in Theorem~\ref{th:maintheorem-f} are trivial and in what follows we will assume that
	\begin{equation}
	\label{eq:def-piccolo}
		\df \le \gamma,
	\end{equation}
	for some small~$\gamma \in \left( 0, \frac{1}{2} \right]$ to be later determined in dependence of universal quantities.

	Lastly, we point out that
	\begin{equation*}
		\df = \inf_{t \in [ \inf_{\R^n} \kappa, \, \sup_{\R^n} \kappa]} \, \norma*{\left(\kappa-t\right)f(u)}_{L^{(2^\ast)'}(\R^n)},
	\end{equation*}
	and that there actually exists~$\kappa_1 \in [ \inf_{\R^n} \kappa, \, \sup_{\R^n} \kappa]$ such that
	\begin{equation} \label{eq:dfattainbykappa1}
		\df = \norma*{\left(\kappa-\kappa_1\right)f(u)}_{L^{(2^\ast)'}(\R^n)}.
	\end{equation}
	To verify the first claim it suffices to notice that, if~$t < \inf_{\R^n} \kappa$, then
	\begin{align*}
		& \norma*{\left(\kappa-t\right)f(u)}_{L^{(2^\ast)'}(\R^n)}^{(2^\ast)'} = \int_{\R^n} (\kappa - t)^{(2^\ast)'} \, |f(u)|^{(2^\ast)'} dx \\
		& \hspace{50pt} \ge \int_{\R^n} (\kappa - \inf_{\R^n} \kappa)^{(2^\ast)'} \, |f(u)|^{(2^\ast)'} dx = \norma*{\left(\kappa - \inf_{\R^n} \kappa \right)f(u)}_{L^{(2^\ast)'}(\R^n)}^{(2^\ast)'},
	\end{align*}
	and similarly for~$t > \sup_{\R^n} \kappa$. The second claim follows from the continuity of the map~$[ \inf_{\R^n} \kappa, \, \sup_{\R^n} \kappa] \ni t \mapsto \norma*{\left(\kappa-t\right)f(u)}_{L^{(2^\ast)'}(\R^n)}$ and Weierstrass theorem.
	
	In the remainder of the proof we establish the validity of the quasi-symmetry statements of Theorem~\ref{th:maintheorem-f}. 
	To this aim, given a direction~$\omega \in \sfera^{n - 1}$ and~$\lambda \in \R$, we define the sets
	\begin{equation*}
		\Sigma_{\omega, \lambda} \coloneqq \Big\{ {x \in \R^n \, \big\lvert \, \omega \cdot x > \lambda} \Big\}, \quad
		T_{\omega, \lambda} \coloneqq \partial \Sigma_{\omega, \lambda} = \Big\{ {x \in \R^n \, \big\lvert \, \omega \cdot x = \lambda} \Big\},
	\end{equation*}
	and we indicate the reflection of a point~$x \in \R^n$ across the hyperplane~$T_{\omega, \lambda}$ by
	\begin{equation*}
		x^{\omega, \lambda} \coloneqq x + 2 (\lambda - \omega \cdot x) \omega.
	\end{equation*}
	Finally, we write
	\begin{equation*}
		u_{\omega, \lambda}(x) \coloneqq u(x^{\omega, \lambda}) \quad \mbox{for } x \in \R^n.
	\end{equation*}
	Our goal in Steps~2-6 is to show that corresponding to every~$\omega \in \sfera^{n - 1}$ there is a number~$\lambda_\star = \lambda_\star(\omega) \in \R$ for which
	\begin{equation} \label{eq:mainMPclaim}
		\begin{gathered}
			u(x) - u_{\omega, \lambda}(x) \le C_\sharp \abs*{\log \df}^{1 - \frac{n}{2}} \quad \mbox{for every } x \in \Sigma_{\omega, \lambda} \mbox{ and } \lambda \ge \lambda_\star, \\
			\norma*{u - u_{\omega, \lambda_\star}}_{L^\infty(\R^n)} + \norma*{\nabla (u - u_{\omega, \lambda_\star})}_{L^2(\R^n)} \le C_\sharp \abs*{\log \df}^{1 - \frac{n}{2}},
		\end{gathered}
	\end{equation}
	for some universal constant~$C_\sharp \ge 1$ and provided~$\gamma$ is smaller than a universal constant~$\gamma_{\sharp} \in \left( 0, \frac{1}{2} \right]$.
	
	
	\subsection{Starting the moving planes procedure.} \label{step:MPstart}
	
	Of course, after a rotation it suffices to verify~\eqref{eq:mainMPclaim} for~$\omega = e_n$. Under this assumption, we drop any reference to the dependence on~$\omega$ and simply write
	\begin{gather*}
		\Sigma_\lambda \coloneqq \Big\{ {x \in \R^n \, \big\lvert \, x_n > \lambda} \Big\}, \quad
		T_\lambda \coloneqq \Big\{ {x \in \R^n \, \big\lvert \, x_n = \lambda} \Big\}, \\
		x^\lambda \coloneqq (x', 2 \lambda - x_n), \quad \mbox{and} \quad u_\lambda(x) \coloneqq u(x^\lambda),
	\end{gather*}
	for every~$x = (x', x_n) \in \R^n$.
	
	We claim here that
	\begin{equation} \label{eq:MPstarts}
		\norma*{(u - u_\lambda)_+}_{L^{2^\ast}\!(\Sigma_\lambda)} \le 4 S^{-2} \,\df \quad \mbox{for all } \lambda \ge \lambda_1,
	\end{equation}
	for a universal constant~$\lambda_1 > 0$ and with~$S>0$ given by~\eqref{eq:Sob_quot}.
	
	In order to verify~\eqref{eq:MPstarts}, note that, without loss of generality, we may assume
	\begin{equation} \label{eq:L2gradnormnot0}
		\norma*{\nabla (u - u_\lambda)_+}_{L^2(\Sigma_\lambda)} > 0.
	\end{equation}
	Otherwise, we would have that~$(u - u_\lambda)_+ = 0$ in~$\Sigma_\lambda$ and, consequently, that~\eqref{eq:MPstarts} holds true.
	
	Assuming~\eqref{eq:L2gradnormnot0}, we test equation~\eqref{eq:mainprob-f} against the function $\left(u-u_\lambda\right)_+ \chi_{\Sigma_\lambda}$. This can be rigorously justified by means of a simple approximation argument. By doing so, we obtain
	\begin{equation*}
		\begin{split}
			\int_{\Sigma_\lambda} \nabla u \cdot \nabla (u-u_\lambda)_+ \, dx &= \int_{\Sigma_\lambda} \kappa f(u) (u-u_\lambda)_+ \, dx \\
			&= \int_{\Sigma_\lambda} (\kappa-\kappa_1) f(u) (u-u_\lambda)_+ \, dx + \kappa_1 \int_{\Sigma_\lambda} f(u) (u-u_\lambda)_+ \, dx,
		\end{split}
	\end{equation*}
	with~$\kappa_1 \in \left[ 0, \| \kappa \|_{L^\infty(\R^n)} \right]$ being the constant found in~\ref{step:prel-obs} -- see~\eqref{eq:dfattainbykappa1} in particular. Since $u_\lambda$ satisfies the same equation as~$u$ with~$\kappa$ replaced by~$\kappa_\lambda$, we also get
	\begin{equation*}
		\int_{\Sigma_\lambda} \nabla u_\lambda \cdot \nabla (u-u_\lambda)_+ \, dx = \int_{\Sigma_\lambda} (\kappa_\lambda-\kappa_1) f(u_\lambda) (u-u_\lambda)_+ \, dx + \kappa_1 \int_{\Sigma_\lambda} f(u_\lambda) (u-u_\lambda)_+ \, dx.
	\end{equation*}
	Subtracting these two identities we deduce
	\begin{equation}
		\label{eq:test-finale}
		\begin{split}
			&\int_{\Sigma_\lambda} \,\abs*{\nabla (u-u_\lambda)_+}^2 \, dx = \int_{\Sigma_\lambda} (\kappa-\kappa_1) f(u) (u-u_\lambda)_+ \, dx \\
			&\hspace{15pt} \quad - \int_{\Sigma_\lambda} (\kappa_\lambda-\kappa_1) f(u_\lambda) (u-u_\lambda)_+  \, dx + \kappa_1 \int_{\Sigma_\lambda} \left(f(u)-f(u_\lambda)\right) (u-u_\lambda)_+ \, dx.
		\end{split}
	\end{equation}
	We now estimate the terms on the right-hand side of~\eqref{eq:test-finale}. For the first one, we use H\"older and Sobolev inequalities
	\begin{equation*}
		\begin{split}
			\int_{\Sigma_\lambda} (\kappa-\kappa_1) f(u) (u-u_\lambda)_+ \, dx &\leq \norma*{(\kappa-\kappa_1) f(u)}_{L^{(2^\ast)'}(\Sigma_\lambda)} \,\norma*{(u-u_\lambda)_+}_{L^{2^\ast}\!(\Sigma_\lambda)} \\
			&\leq S^{-1} \,\df \norma*{\nabla (u-u_\lambda)_+}_{L^2(\Sigma_\lambda)}.
		\end{split}
	\end{equation*}
	The second term can be handled in a completely analogous way. Finally, to deal with the last summand we exploit~\eqref{eq:ip-f-tipolip} and apply, again, H\"older and Sobolev inequalities to get
	\begin{equation*}
		\begin{split}
			\kappa_1 \int_{\Sigma_\lambda} \left(f(u)-f(u_\lambda)\right) (u-u_\lambda)_+ \, dx & \leq L \kappa_1 \int_{\Sigma_\lambda} u^{p-1} (u-u_\lambda)_+^2 \, dx \\
			&\leq L \kappa_1 \left(\int_{\Sigma_\lambda} u^{2^\ast} dx \right)^{\!\! \frac{2^\ast-2}{2^\ast}} \left(\int_{\Sigma_\lambda} (u-u_\lambda)_+^{2^\ast} dx \right)^{\!\! \frac{2}{2^\ast}} \\
			&\leq L \kappa_1 S^{-2} \norma*{u}_{L^{2^\ast}\!(\Sigma_\lambda)}^{2^\ast-2} \int_{\Sigma_\lambda} \,\abs*{\nabla (u-u_\lambda)_+}^{2} \, dx.
		\end{split}
	\end{equation*}
	Going back to~\eqref{eq:test-finale}, we deduce that
	\begin{equation} \label{eq:est-1-def}
		\left( 1- \frac{L \kappa_1}{S^2} \,\norma*{u}_{L^{2^\ast}\!(\Sigma_\lambda)}^{2^\ast-2} \right) \norma*{\nabla (u-u_\lambda)_+}^2_{L^2(\Sigma_\lambda)} \leq 2 S^{-1} \,\df \norma*{\nabla (u-u_\lambda)_+ }_{L^2(\Sigma_\lambda)}.
	\end{equation}
	
	We now show that~\eqref{eq:est-1-def} leads to~\eqref{eq:MPstarts} up to taking~$\lambda$ large enough. Recalling~\eqref{eq:decadimento}, we see that
	\begin{equation} \label{eq:tailubound}
		\norma*{u}_{L^{2^\ast}\!(\Sigma_\lambda)}^{2^\ast} \le C_0^{2^\ast} \int_{\R^n \setminus B_\lambda(0)} \frac{dx}{\abs{x}^{2n}} = \frac{C_0^{2^\ast} \Haus^{n - 1}(\sfera^{n - 1})}{n \lambda^n}.
	\end{equation}
	Hence, as~$\kappa_1 \le \| \kappa \|_{L^\infty(\R^n)}$,
	\begin{equation*}
		1- \frac{L \kappa_1}{S^2} \,\norma*{u}_{L^{2^\ast}\!(\Sigma_\lambda)}^{2^\ast-2} \ge 1 - \frac{L \| \kappa \|_{L^\infty(\R^n)} \Haus^{n - 1}(\sfera^{n - 1})^{\frac{2}{n}} C_0^{\frac{4}{n - 2}}}{n^{\frac{2}{n}} S^2 \lambda^2} \ge \frac{1}{2},
	\end{equation*}
	provided~$\lambda \geq \lambda_1$ with
	\begin{equation*} \label{eq:lambda1def}
		\lambda_1 \coloneqq \frac{\sqrt{2 L \| \kappa \|_{L^\infty(\R^n)}}}{S} \left( \frac{\Haus^{n - 1}(\sfera^{n - 1}) C_0^{2^\ast}}{n} \right)^{\!\! \frac{1}{n}} \!.
	\end{equation*}
	By combining this with~\eqref{eq:est-1-def} and Sobolev inequality, we immediately deduce that
	\begin{equation} \label{eq:stima2ast}
		\norma*{(u - u_\lambda)_+} _{L^{2^\ast} \! (\Sigma_\lambda)} \leq 4 S^{-2} \,\df,
	\end{equation}
	for every~$\lambda \ge \lambda_1$ -- recall that assumption~\eqref{eq:L2gradnormnot0} is in force. Hence, claim~\eqref{eq:MPstarts} holds true.
	
	Estimate~\eqref{eq:MPstarts} yields that the set
	\begin{equation*}
		\Lambda \coloneqq \Big\{ \lambda \in \R \, \big\lvert \, \norma*{ (u-u_\mu)_+}_{L^{2^\ast} \! (\Sigma_\mu)} \leq 4 S^{-2} \,\df \,\, \mbox{ for every } \mu \geq \lambda \Big\}
	\end{equation*}
	is non-empty. Thus,
	\begin{equation*}
		\lambda_\star \coloneqq \inf\Lambda
	\end{equation*}
	is well-defined and~$\lambda_\star \in [-\infty,\lambda_1]$. We conclude this step of the proof by showing that~$\lambda_\star$ is actually a real number universally bounded from below, at least for~$\gamma$ sufficiently small. To see this, let~$\lambda \in \Lambda \cap (-\infty, -1]$. By the triangular inequality,
	\begin{equation*}
		4 S^{-2} \,\df \ge \norma*{u - u_\lambda}_{L^{2^\ast}\!(\Sigma_\lambda \cap \{ u > u_\lambda \})} \geq \norma*{u}_{L^{2^\ast}\!(\Sigma_\lambda \cap \{ u > u_\lambda \})} - \norma*{u_\lambda}_{L^{2^\ast}\!(\Sigma_\lambda \cap \{ u > u_\lambda \})}.
	\end{equation*}
	Taking advantage of the estimate~\eqref{eq:boundmassa}, we then have
	\begin{align*}
		\mathcal{M} & \leq  \norma*{u}_{L^{2^\ast}\!(\R^n)}^{2^\ast} = \norma*{u}_{L^{2^\ast}\!(\Sigma_\lambda \cap \{ u > u_\lambda \})}^{2^\ast} + \norma*{u}_{L^{2^\ast}\!(\Sigma_\lambda \cap \{ u \le u_\lambda \})}^{2^\ast} + \norma*{u}_{L^{2^\ast}\!(\R^n \setminus \Sigma_\lambda)}^{2^\ast} \\
		& \le \left( \norma*{u_\lambda}_{L^{2^\ast}\!(\Sigma_\lambda \cap \{ u > u_\lambda \})} + 4 S^{-2} \,\df  \right)^{\! 2^\ast} + \norma*{u_\lambda}_{L^{2^\ast}\!(\Sigma_\lambda \cap \{ u \le u_\lambda \})}^{2^\ast} + \norma*{u}_{L^{2^\ast}\!(\R^n \setminus \Sigma_\lambda)}^{2^\ast} \\
		& \le 2^{2^\ast} \norma*{u_\lambda}_{L^{2^\ast}\!(\Sigma_\lambda)}^{2^\ast} + 8^{2^\ast} S^{- 2 \cdot 2^\ast} \df^{2^\ast} + \norma*{u}_{L^{2^\ast}\!(\R^n \setminus \Sigma_\lambda)}^{2^\ast}.
	\end{align*}
	A change of variables gives that
	\begin{equation*}
		\norma*{u_\lambda}_{L^{2^\ast}\!(\Sigma_\lambda)} = \norma*{u}_{L^{2^\ast}\!(\R^n \setminus \Sigma_\lambda)}.
	\end{equation*}
	Hence, using the decay hypothesis~\eqref{eq:decadimento}, computing as in~\eqref{eq:tailubound}, and taking advantage of~\eqref{eq:def-piccolo}, we find
	\begin{equation*}
		\mathcal{M} \le 4^{2^\ast} \norma*{u}_{L^{2^\ast}\!(\R^n \setminus \Sigma_\lambda)}^{2^\ast} + 8^{2^\ast} S^{- 2 \cdot 2^\ast} \df^{2^\ast} \le 4^{2^\ast} \frac{C_0^{2^\ast} \Haus^{n - 1}(\sfera^{n - 1})}{n (- \lambda)^n} + 8^{2^\ast} S^{- 2 \cdot 2^\ast} \gamma^{2^\ast}.
	\end{equation*}
	By taking
	\begin{equation*}
		\gamma \le \gamma_1 \coloneqq \frac{S^{2}}{8}\left(\frac{\mathcal{M}}{2}\right)^{\!\frac{1}{2^\ast}} \!,
	\end{equation*}
	we conclude that~$\lambda_\star \ge - \lambda_2$, with
	\begin{equation*}
		\lambda_2 \coloneqq \max \left\{ \left( 4 C_0 \right)^{\! \frac{2^\ast}{n}} \left( \frac{2f_0 \Haus^{n - 1}(\sfera^{n - 1})}{n \mathcal{M}} \right)^{\!\! \frac{1}{n}}, 1 \right\},
	\end{equation*}
	as claimed. Combining this with~\eqref{eq:MPstarts}, we find in particular that
	\begin{equation} \label{eq:limit-lstar}
		\lambda_\star \in [-\lambda_0, \lambda_0],
	\end{equation}
	with~$\lambda_0 \coloneqq \max \{ \lambda_1, \lambda_2 \}$.
	
	
	\subsection{From integral to pointwise estimates.} \label{step:int2point}
	
	Now, we show that, as long as~\eqref{eq:stima2ast} holds true, we can obtain a pointwise bound on $\left(u-u_{\lambda}\right)_+$.
	
	We know that~\eqref{eq:stima2ast} holds for any~$\lambda > \lambda_\star$,~by definition of~$\Lambda$. Fatou's lemma ensures that it also holds for~$\lambda = \lambda_\star$, that is
	\begin{equation} \label{eq:stimaperlambdastar}
		\norma*{(u - u_{\lambda_\star})_+}_{L^{2^\ast} \! (\Sigma_{\lambda_\star})} \leq 4 S^{-2} \,\df.
	\end{equation}
	Let~$\lambda \geq \lambda_\star$ be fixed and observe that
	\begin{equation*}
		\begin{split}
			\Delta (u-u_{\lambda}) &= \kappa_{\lambda} f(u_{\lambda}) - \kappa f(u) \\
			&= (\kappa_{\lambda}-\kappa_1) f(u_{\lambda}) - (\kappa-\kappa_1) f(u) - \kappa_1 (f(u)-f(u_{\lambda})) \quad \mbox{in } \Sigma_\lambda,
		\end{split}
	\end{equation*}
	that is
	\begin{equation*} \label{eq:eq-differenza}
		\Delta (u-u_{\lambda}) + \kappa_1 \, c_{\lambda} (u-u_{\lambda}) = g_\lambda \quad \mbox{in } \Sigma_\lambda,
	\end{equation*}
	with
	\begin{equation} \label{eq:clambdadef}
		c_{\lambda} (x) \coloneqq
		\begin{dcases}
			\frac{f(u(x))-f(u(x^{\lambda}))}{u(x)-u(x^{\lambda})}	& \quad \mbox{if } u(x) \neq u(x^{\lambda}), \\
			0													& \quad \mbox{if } u(x) = u(x^{\lambda}),
		\end{dcases}
	\end{equation}
	and
	\begin{equation} \label{eq:glambdadef}
		g_\lambda \coloneqq \left(\kappa_{\lambda}-\kappa_1\right) f(u_{\lambda}) - \left(\kappa-\kappa_1\right) f(u).
	\end{equation}
	Note that~$c_{\lambda} \in L^\infty(\R^n)$ with~$\abs*{c_{\lambda}}  \leq \norma*{f}_{C^{0,1}\left([0,C_0]\right)}$, and~$g_\lambda \in L^\infty(\R^n) \cap L^{\left(2^\ast\right)'}(\R^n)$. Hence, we are in a position to apply Theorem~8.17 in~\cite{gt} and get that, for~$B_2(y) \subseteq \Sigma_{\lambda}$ and~$q>n/2$, there holds
	\begin{equation}
		\label{eq:stima-puntuale-lontano}
		\begin{split}
			\sup_{B_1(y)} (u-u_{\lambda}) & \le C \Big( {\norma*{(u-u_{\lambda})_+}_{L^{2^\ast} \! \left(B_2(y)\right)} +\norma*{g_\lambda}_{L^q(B_2(y))}} \Big) \\
			& \le C \Big( {\df +\norma*{g_\lambda}_{L^q(\Sigma_\lambda)}} \Big),
		\end{split}
	\end{equation}
	for some constant~$C > 0$ depending only on~$n$,~$\norma*{f}_{C^{0,1}\left([0,C_0]\right)}$,~$\norma*{\kappa}_{L^\infty(\R^n)}$, and~$q$. To proceed further, we need an estimate for~$\norma*{g_\lambda}_{L^q(\Sigma_\lambda)}$. Note that~$q>\left(2^\ast\right)'$ for any~$n\geq 3$. Thus,~\eqref{eq:ip-f-sub} and H\"older inequality yield
	\begin{equation} \label{eq:glambdaLinftyest}
		\| g_\lambda \|_{L^q(\Sigma_\lambda)} \le 2 \norma*{(\kappa-\kappa_1) f(u)}_{L^q(\R^n)} \le 2 f_0 \norma*{(\kappa-\kappa_1) u^p}_{L^\infty(\R^n)}^{1-\left(2^\ast\right)'/q} \,\df^{\left(2^\ast\right)'/q}.
	\end{equation}
	With this estimate at hand,~\eqref{eq:stima-puntuale-lontano} becomes
	\begin{equation}
		\label{eq:stima-puntuale-lontano-1}
		\sup_{B_1(y)} \left(u-u_{\lambda}\right) \le C \,\df^{\left(2^\ast\right)'\slash q} \quad \mbox{for every } y \in \Sigma_\lambda \mbox{ with } y_n > \lambda + 2,
	\end{equation}
	where~$C$ may now depend on~$C_0$ and~$f_0$ as well. This provides a uniform $L^\infty$-bound away from the hyperplane~$T_{\lambda}$.
	
	In order to get the bound in a neighborhood of~$T_{\lambda}$, we exploit the Dirichlet boundary datum. Since~$u-u_{\lambda} = 0$ on~$T_{\lambda}$, we can apply Theorem~8.25 in~\cite{gt} and obtain, arguing as above,
	\begin{equation}
		\label{eq:stima-puntuale-vicino}
		\sup_{B_1^+(y)} (u-u_{\lambda})_+ \le C \,\df^{\left(2^\ast\right)' /  q} \quad \mbox{for every } y \in T_\lambda,
	\end{equation}
	where~$B_1^+(y) \coloneqq B_1(y) \cap \Sigma_{\lambda}$.
	
	By setting, e.g.,~$q=n$ in~\eqref{eq:stima-puntuale-lontano-1} and~\eqref{eq:stima-puntuale-vicino}, we deduce
	\begin{equation}
		\label{eq:stima-puntuale}
		\| (u-u_{\lambda})_+ \|_{L^\infty(\Sigma_\lambda)} \leq C_2 \,\df^\frac{2}{n+2} \quad \text{for every } \lambda \ge \lambda_\star,
	\end{equation}
	for some constant~$C_2 \geq 1$ depending only on~$n$,~$C_0$,~$f_0$,~$\norma*{\kappa}_{L^\infty(\R^n)}$, and~$\norma*{f}_{C^{0,1}\left([0,C_0]\right)}$.
	
	
	\subsection{Finding mass in a small region.} \label{step:massa}
	
	We shall show that
	\begin{equation} \label{eq:u-ulambdastar-small}
		\norma*{(u-u_{\lambda_\star})_-}_{L^{2^\ast}\!(\Sigma_{\lambda_\star})} \le C \, \abs{\log\df}^{-\alpha},
	\end{equation}
	for some universal constants~$C \ge 1$ and~$\alpha > 0$. We argue by contradiction and suppose that
	\begin{equation}
		\label{eq:stima-assurda}
		\norma*{(u-u_{\lambda_\star})_-}_{L^{2^\ast}\!(\Sigma_{\lambda_\star})} > \mathcal{B} \,\abs{\log\df}^{-\alpha}
	\end{equation}
	for some~$\alpha > 0$ and some large~$\mathcal{B} \ge 1$ to be determined later depending on universal constants. Our purpose here is to prove that, as a consequence of~\eqref{eq:stima-assurda}, the function~$(u - u_{\lambda_\star})_- = (u_{\lambda_\star}-u)_+$ has positive mass in a small cube. Later, we will enlarge the positivity region and, ultimately, derive a contradiction from this.
	
	First, we claim that there exists~$R_0 \ge \lambda_0 + 2$, depending on~$\df$, such that
	\begin{equation}
		\label{eq:massa-bolla}
		\int_{B_{2R_0}(x_\star) \cap \Sigma_{\lambda_\star}} (u_{\lambda_\star}-u)_+^{2^\ast} \, dx \geq \frac{\mathcal{B}^{2^\ast}}{2} \, \abs{\log\df}^{-2^\ast \alpha},
	\end{equation}
	where $x_\star=\left(0',\lambda_\star\right) \in T_{\lambda_\star}$. With the aid of~\eqref{eq:decadimento} we estimate
	\begin{equation*}
		\begin{split}
			\int_{\Sigma_{\lambda_\star} \setminus B_{2R_0}(x_\star)} u_{\lambda_\star}^{2^\ast} \, dx &\leq
			\int_{\R^n \setminus B_{2R_0}(x_\star)} u_{\lambda_\star}^{2^\ast} \, dx \\
			&\leq C_0^{2^\ast} \int_{\R^n \setminus B_{2R_0}(x_\star)} \frac{dx}{\abs*{x^{\lambda_\star}}^{2n}} = C_0^{2^\ast} \int_{\R^n \setminus B_{2R_0}(x_\star)} \frac{dx}{\abs*{x}^{2n}}
		\end{split}
	\end{equation*} 
	since~$x_\star = x_\star^{\lambda_\star}$, as~$x_\star \in T_{\lambda_\star}$. Recalling~\eqref{eq:limit-lstar}, we have that~$B_{R_0}(0) \subseteq B_{2R_0}(x_\star)$, if~$R_0 \geq \lambda_0$. Computing as in~\eqref{eq:tailubound}, we find
	\begin{equation*}
		\begin{split}
			\int_{\Sigma_{\lambda_\star} \setminus B_{2R_0}(x_\star)} u_{\lambda_\star}^{2^\ast} \, dx &\leq C_0^{2^\ast} \int_{\R^n \setminus B_{R_0}(0)} \frac{dx}{\abs*{x}^{2n}} = \frac{C_0^{2^\ast} \Haus^{n- 1}(\sfera^{n - 1})}{n} \, R_0^{-n} \\
			& = \frac{\mathcal{B}^{2^\ast}}{2} \,\abs{\log\df}^{-2^\ast \alpha},
		\end{split}
	\end{equation*}
	choosing
	\begin{equation} \label{eq:R0def}
		R_0 \coloneqq \left( \frac{2 C_0^{2^\ast} \Haus^{n-1}(\sfera^{n-1})}{n \mathcal{B}^{2^\ast}} \right)^{\!\! \frac{1}{n}}  \, \abs{\log\df}^{\frac{2 \alpha}{n - 2}}.
	\end{equation}
	Claim~\eqref{eq:massa-bolla} then follows by writing
	\begin{equation*}
		\int_{B_{2R_0}(x_\star) \cap \Sigma_{\lambda_\star}} (u_{\lambda_\star}-u)_+^{2^\ast} \, dx = \int_{\Sigma_{\lambda_\star}} (u_{\lambda_\star}-u)_+^{2^\ast} \, dx - \int_{\Sigma_{\lambda_\star} \setminus B_{2R_0}(x_\star)} (u_{\lambda_\star}-u)_+^{2^\ast} \, dx,
	\end{equation*}
	exploiting the above computation, and taking advantage of~\eqref{eq:stima-assurda}. Note that the condition~$R_0\geq \lambda_0 + 2$ holds true provided
	\begin{equation*}
		\gamma \leq \gamma_2 \coloneqq \exp\left\{-\left( \frac{n \mathcal{B}^{2^\ast}}{2C_0^{2^\ast} \Haus^{n-1}(\sfera^{n-1})}
		\right)^{\!\! \frac{1}{2^\ast\alpha}} (\lambda_0 + 2)^\frac{n-2}{2\alpha} \right\}.
	\end{equation*}
	
	Consider the slab~$\mathcal{S}_{\lambda_\star,\delta} \coloneqq \Sigma_{\lambda_\star} \setminus \Sigma_{\lambda_\star+\sqrt{n}\delta}$, for~$\delta \in (0, 1)$ to be determined later. We claim that, provided~$\delta$ is sufficiently small and exploiting the regularity of~$u$, it is possible to obtain an estimate analogous to~\eqref{eq:massa-bolla} over~$B_{2R_0}(x_\star) \cap \Sigma_{\lambda_\star+\sqrt{n}\delta}$. Indeed,~\eqref{eq:C1boundonu} and the fact that~$u_{\lambda_\star} - u = 0$ on~$T_{\lambda_\star}$ give
	\begin{equation*}
		\left(u_{\lambda_\star}-u\right)_+ \leq 2\sqrt{n} \, C_1\delta \quad \text{in } \mathcal{S}_{\lambda_\star,\delta}.
	\end{equation*}
	Therefore,
	\begin{equation*}
		\int_{B_{2R_0}(x_\star) \cap \mathcal{S}_{\lambda_\star,\delta}} (u_{\lambda_\star}-u)_+^{2^\ast} \, dx \leq 4^{n-1} \sqrt{n} \left(2\sqrt{n} \, C_1\right)^{\!2^\ast} \! R_0^{n - 1} \delta^{2^\ast + 1}.
	\end{equation*}
	By requiring
	\begin{equation}
		\label{eq:cond-delta-1}
		\delta \leq \left( \frac{ \mathcal{B}^{2^\ast} \,\abs{\log\df}^{-2^\ast \alpha}}{4^{n} \left( 2\sqrt{n} \, C_1 \right)^{\! 2^\ast} \! R_0^{n-1}} \, \right)^{\!\! \frac{1}{2^\ast + 1}} \!,
	\end{equation}
	we get
	\begin{equation*}
		\int_{B_{2R_0}(x_\star) \cap \mathcal{S}_{\lambda_\star,\delta}} (u_{\lambda_\star}-u)_+^{2^\ast} \, dx \leq \frac{\mathcal{B}^{2^\ast}}{4} \,\abs{\log\df}^{-2^\ast \alpha},
	\end{equation*}
	and from~\eqref{eq:massa-bolla} we conclude that
	\begin{equation}
		\label{eq:massa-bolla-1}
		\int_{B_{2R_0}(x_\star) \cap \Sigma_{\lambda_\star+\sqrt{n}\delta}} (u_{\lambda_\star}-u)_+^{2^\ast} \, dx \geq \frac{\mathcal{B}^{2^\ast}}{4} \,\abs{\log\df}^{-2^\ast \alpha},
	\end{equation}
	as desired.
	
	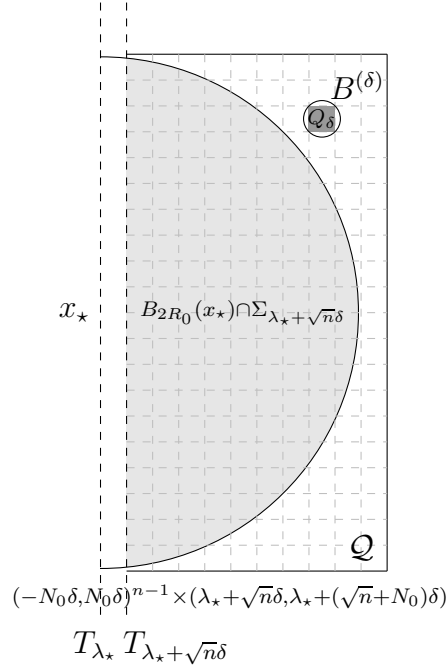
\begin{figure}
		\centering
		\begin{tikzpicture}[x=0.75pt,y=0.75pt,yscale=-1.3,xscale=1.3]
			
			\begin{scope}
				\clip (10,50) rectangle (110,250);
				\path[fill=black!10] (0,150) circle(99);
			\end{scope}
		
				\begin{scope}
				\clip (0,50) rectangle (110,250);
				\draw (0,150) circle(99);
			\end{scope}
			
			\node[anchor=east] at (0,150) {$x_\star$};
			
			\draw (10,50) -- (110,50);
			\draw (10,250) -- (110,250);
			\draw (110,50) -- (110,250);
			
			\draw[dashed, black!25] (20,50) -- (20,250);
			\draw[dashed, black!25] (30,50) -- (30,250);
			\draw[dashed, black!25] (40,50) -- (40,250);
			\draw[dashed, black!25] (50,50) -- (50,250);
			\draw[dashed, black!25] (60,50) -- (60,250);
			\draw[dashed, black!25] (70,50) -- (70,250);
			\draw[dashed, black!25] (80,50) -- (80,250);
			\draw[dashed, black!25] (90,50) -- (90,250);
			\draw[dashed, black!25] (100,50) -- (100,250);
			
			\draw[dashed, black!25] (10,60) -- (110,60);
			\draw[dashed, black!25] (10,70) -- (110,70);
			\draw[dashed, black!25] (10,80) -- (110,80);
			\draw[dashed, black!25] (10,90) -- (110,90);
			\draw[dashed, black!25] (10,100) -- (110,100);
			\draw[dashed, black!25] (10,110) -- (110,110);
			\draw[dashed, black!25] (10,120) -- (110,120);
			\draw[dashed, black!25] (10,130) -- (110,130);
			\draw[dashed, black!25] (10,140) -- (110,140);
			\draw[dashed, black!25] (10,150) -- (110,150);
			\draw[dashed, black!25] (10,160) -- (110,160);
			\draw[dashed, black!25] (10,170) -- (110,170);
			\draw[dashed, black!25] (10,180) -- (110,180);
			\draw[dashed, black!25] (10,190) -- (110,190);
			\draw[dashed, black!25] (10,200) -- (110,200);
			\draw[dashed, black!25] (10,210) -- (110,210);
			\draw[dashed, black!25] (10,220) -- (110,220);
			\draw[dashed, black!25] (10,230) -- (110,230);
			\draw[dashed, black!25] (10,240) -- (110,240);
			
			\path [fill=black!40] (80,70) rectangle (90,80);
			\draw (85,75) circle(7.0710678118655);
			
			\draw[name path=int, dashed] (0,30) -- (0,270);
			\draw[name path=ext, dashed] (10,30) -- (10,270);
			
			\draw (-2,285.5) node[anchor=south, inner sep=0.75pt, align=center] {$T_{\lambda_\star}$};
			\draw (29,287) node[anchor=south, inner sep=0.75pt, align=center] {$T_{\lambda_\star+\sqrt{n}\delta}$};
			\draw (85,75) node[inner sep=0.75pt, align=left] {$\scriptstyle Q_\delta$};
			\draw (87,62) node[anchor=west, inner sep=0.75pt, align=left] {$B^{(\delta)}$};
			\draw (95,150) node [anchor=east, inner sep=0.75pt, align=right] {$\scriptstyle B_{2R_0}(x_\star) \cap \Sigma_{\lambda_\star+\sqrt{n}\delta}$};
			\draw (-35,260) node [anchor=west, inner sep=0.75pt] {$\scriptstyle ( {-N_0 \delta, N_0 \delta})^{n - 1} \times ( {\lambda_\star + \sqrt{n} \delta, \lambda_\star + (\sqrt{n} + N_0) \delta} )$};
			\draw (100,235) node [anchor=north, inner sep=0.75pt] {$\mathcal{Q}$};
			
		\end{tikzpicture}
		\caption{The decomposition of~\ref{step:massa}.}
		\label{fig:cube}
	\end{figure}
	
	We now enclose~$B_{2R_0}(x_\star) \cap \Sigma_{\lambda_\star + \sqrt{n} \delta}$ within the half-cube
	\begin{equation*}
		\big( {-N_0 \delta, N_0 \delta} \big)^{\! n-1} \times \big( {\lambda_\star + \sqrt{n} \delta, \lambda_\star + (\sqrt{n} + N_0) \delta} \big),
	\end{equation*}
	with~$N_0 \coloneqq \left\lceil \frac{2 R_0}{\delta} \right\rceil$. Then, we further subdivide the latter into the family~$\mathcal{Q}$ of contiguous, essentially disjoint closed cubes of sides~$\delta$. We refer to Figure~\ref{fig:cube} for a graphical representation of this construction. Using that~$\mbox{card}(\mathcal{Q}) = 2^{n - 1} N_0^n$, we easily conclude from~\eqref{eq:massa-bolla-1} that there exists a cube~$Q_\delta \in \mathcal{Q}$ for which
	\begin{equation*}
		\int_{Q_\delta} (u_{\lambda_\star}-u)_+^{2^\ast} \, dx \ge \frac{1}{2^{n - 1} N_0^n}\frac{\mathcal{B}^{2^\ast}}{4} \,\abs{\log\df}^{-2^\ast \alpha} \ge \frac{\mathcal{B}^{2^\ast} \delta^n}{2\, 8^n R_0^n} \,\abs{\log\df}^{-2^\ast \alpha},
	\end{equation*}
	that is, recalling~\eqref{eq:R0def},
	\begin{equation}
		\label{eq:massa-cubo}
		\dashint_{Q_\delta} (u_{\lambda_\star}-u)_+^{2^\ast} \, dx \ge \frac{n \mathcal{B}^{2 \cdot 2^\ast}}{16^n C_0^{2^\ast} \Haus^{n-1}(\sfera^{n-1})} \,\abs{\log\df}^{-2 \cdot 2^\ast \alpha}.
	\end{equation}
	
	
	\subsection{Propagating positivity to a large region.}
	
	Now, we take advantage of  the previous step to establish the positivity of $\left(u_{\lambda_\star}-u\right)_+$ in a universally large region.
	
	Let~$r \in [\lambda_0 + 1, R_0]$ to be determined later, depending only on universal constants. From~\eqref{eq:limit-lstar},~$B_{r}(0) \cap \Sigma_{\lambda_\star} \neq \varnothing$. Take~$\delta \in (0, 1)$ as above and define
	\begin{equation}
	\label{eq:defK}
		K_\delta \coloneqq \overline{B_{r}(0) \cap \Sigma_{\lambda_\star+\sqrt{n}\delta}}.
	\end{equation}
	Then, the set~$K_\delta$ is a compact subset of~$\Sigma_{\lambda_\star}$, disjoint from the hyperplane~$T_{\lambda_\star}$ where the function~$u-u_{\lambda_\star}$ vanishes. We shall show that~$\left(u_{\lambda_\star}-u\right)_+$ is positive in~$K_\delta$.
	
	By~\eqref{eq:stima-puntuale},~$v \coloneqq u_{\lambda_\star}-u + 2C_2 \,\df^{2 / (n + 2)}$ is non-negative and solves
	\begin{equation}
		\label{eq:eqperwH}
		\begin{split}
			\Delta v + \kappa_1 \, c_{\lambda_\star} v &= (\kappa-\kappa_1) f(u) - (\kappa_{\lambda_\star}-\kappa_1) f(u_{\lambda_\star}) + \kappa_1 \, c_{\lambda_\star} (u-u_{\lambda_\star}) +\kappa_1 \, c_{\lambda_\star} v \\
			&= -g_{\lambda_\star} + 2 C_2 \,\kappa_1 \, c_{\lambda_\star} \df^{\frac{2}{n+2}} \qquad\text{in } \Sigma_{\lambda_\star},
		\end{split}
	\end{equation}
	with~$c_{\lambda_\star}$ and~$g_{\lambda_\star}$ as in~\eqref{eq:clambdadef} and~\eqref{eq:glambdadef}. Notice that~\eqref{eq:glambdaLinftyest} with~$q = n$ gives
	\begin{equation} \label{eq:RHSest}
		\big\| {\big( {-g_{\lambda_\star} + 2 C_2 \,\kappa_1 \, c_{\lambda_\star} \df^{\frac{2}{n+2}}} \big)_{\! +}} \big\|_{L^n(B)} \le C \,\df^{\frac{2}{n + 2}},
	\end{equation}
	for any ball~$B \subseteq \Sigma_{\lambda_\star}$ of radius smaller or equal to~$2$ and for some constant~$C > 0$ depending only on~$n$,~$C_0$,~$f_0$,~$\| \kappa \|_{L^\infty(\R^n)}$, and~$\norma*{f}_{C^{0,1}\left([0,C_0]\right)}$.
	
	\begin{figure}
		\centering
		\begin{tikzpicture}[x=0.75pt,y=0.75pt,yscale=-1.2,xscale=1.2,rotate=180]
			
			\tikzstyle{mark coordinate}=[minimum size=0pt,inner sep=0,outer sep=0,fill=none,circle]
			
			\coordinate (su) at (130,30);
			\coordinate (giu) at (130,270);
			
			\draw[name path=circ] (250,63.39745962) arc (-120:-240:100);
			
			\path[fill=black!5,even odd rule](su) rectangle (240,270) (250,63.39745962) arc (-120:-240:100);
			
			\path[fill=white!100] (250,63.39745962) arc (-120:-245:100);
			
			\begin{scope}
				\clip (240,30) -- (240,270) -- (230,270) -- (230,30) -- cycle;
				\fill[black!30] (250,63.39745962) arc (-120:-240:100);
			\end{scope}
			
			\draw (240,250) -- (140,250);
			\draw (240,50) -- (140,50);
			\draw (140,250) -- (140,50);
			
			\path[dashed, name path=int] (240,30) -- (240,270);
			\path[dashed, name path=cent] (230,30) -- (230,270);
			
			\tikzfillbetween[of=cent and circ,on layer=,split, every even segment/.style={fill=none,draw=none},]{black!15}
			
			\draw[name path=circ] (250,63.39745962) arc (-120:-240:100);
			
			\draw[dashed] (240,30) -- (240,270);
			\draw[dashed] (230,30) -- (230,270);
			
			\draw (240,267) node [anchor=south, inner sep=0.75pt, align=center] {$T_{\lambda_\star-\epsilon}$};
			\draw (217,17) node [anchor=south, inner sep=0.75pt, align=center] {$T_{\lambda_\star+\sqrt{n}\delta}$};
			\draw (175,150) node [anchor=west, inner sep=0.75pt, align=center] {$\displaystyle E_{r}^\epsilon$};
			\draw (228.5,150) node [anchor=east, inner sep=0.75pt, align=center] {$\displaystyle S_\delta^\epsilon$};
			\draw (207,150) node [anchor=east, inner sep=0.75pt, align=center] {$\displaystyle K_\delta$};
			\draw (152,225) node [anchor=south, inner sep=0.75pt, align=center] {$\displaystyle B^{(\delta)}$};
			\draw (237,170) node [anchor=south, inner sep=0.75pt, align=center] {$\displaystyle B^\star$};
			
			\begin{scope}[yscale=-1,xscale=1]
				\draw[fill=black!50] (152,386/5-300) circle (2);
				\draw (152,386/5-300) circle (2);
				\draw[ultra thin] (155,79-300) circle (3);
				\draw[ultra thin] (159,407/5-300) circle (4.5);
				\draw[ultra thin] (165,85-300) circle (7);
				\draw[ultra thin] (170,88-300) circle (7);
				\draw[ultra thin] (175,91-300) circle (7);
				\draw[ultra thin] (180,94-300) circle (7);
				\draw[ultra thin] (185,97-300) circle (7);
				\draw[ultra thin] (190,100-300) circle (7);
				\draw[ultra thin] (195,103-300) circle (7);
				\draw[ultra thin] (200,106-300) circle (7);
				\draw[ultra thin] (205,109-300) circle (7);
				\draw[ultra thin] (210,112-300) circle (7);
				\draw[ultra thin] (215,115-300) circle (7);
				\draw[ultra thin] (221,593/5-300) circle (4.5);
				\draw[ultra thin] (225,121-300) circle (3);
				\draw (228,614/5-300) circle (2);
			\end{scope}
		\end{tikzpicture}
		\caption{The decomposition of $\Sigma_{\lambda_\star -\epsilon}$ and the Harnack chain.}
		\label{fig:dom-dec}
	\end{figure}
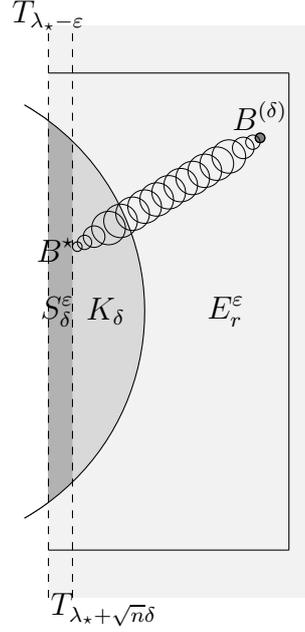
	
	Let $B^{(\delta)}$ be the ball of radius $\sqrt{n}\delta/2$ such that $Q_\delta \subseteq B^{(\delta)}$ and $B^\star$ a ball of the same radius such that~$\min_{K_\delta} v = \inf_{B^\star} v$ and~$2B^\star \subseteq \Sigma_{\lambda_\star}$ -- here,~$2B^\star$ denotes the ball with the same center as~$B^\star$ and double the radius. Then, let~$B^{(1)}$ a ball of unit radius with center at distance at most~$2$ from that of~$B^{(\delta)}$ and such that~$2B^{(1)} \subseteq \Sigma_{\lambda_\star}$. Given~$N \in \N$, let~$B^{(N)}$ be analogously defined in relation to~$B^\star$. Finally, we join~$B^{(1)}$ and~$B^{(N)}$ through a Harnack chain $\{B^{(k)}\}_{k=1}^N$ made of balls having radius~$1$ and centers belonging to the segment connecting those of~$B^{(1)}$ and~$B^{(N)}$. Moreover, we require that
	\begin{equation}
		\label{eq:misurainter}
		\frac{\abs{B^{(k)}}}{\abs*{B^{(k-1)}\cap B^{(k)}}} \le 4^n,
	\end{equation}
	for~$k=2,\dots,N$. We refer to Figure~\ref{fig:dom-dec} for a graphical representation of this construction. Using that~$K_\delta \subseteq B_{2 R_0}(x_\star) \cap \Sigma_{\lambda_\star + \sqrt{n} \delta}$, it is not hard to see that we can choose these balls in such a way that
	\begin{equation} \label{eq:Nupperbound}
		N \leq 16 \sqrt{n}R_0.
	\end{equation}
	
	Recalling the bound~\eqref{eq:RHSest} for the right-hand side of~\eqref{eq:eqperwH}, we apply the weak Harnack inequality for the Laplacian -- see, e.g.,~\cite[Theorem~7.1.2]{ps} or~\cite[Theorem~8.18]{gt} -- and get that
	\begin{equation}
	\label{eq:wHarnack}
		\left( \, \dashint_{B^{(k)}} v^s \,dx\right)^{\!\! \frac{1}{s}} \leq C_\star \left(\inf_{B^{(k)}} v + \df^\frac{2}{n+2}\right) \quad \mbox{for all } k = 1, \ldots, N,
	\end{equation}
	for any~$s \in \left( 0, 2^\ast/2\right)$ and for some~$C_\star \geq 1$ depending only on~$n$,~$C_0$,~$f_0$,~$\norma*{\kappa}_{L^\infty(\R^n)}$, $\norma*{f}_{C^{0,1}\left([0,C_0]\right)}$, and~$s$. By means of~\eqref{eq:misurainter} and~\eqref{eq:wHarnack}, we compare the average of~$v^s$ on two consecutive balls
	\begin{equation*}
		\label{eq:wH-daiterare}
		\begin{split}
			\left( \, \dashint_{B^{(k-1)}} v^s \,dx\right)^{\!\! \frac{1}{s}} &\leq  C_\star \left(\inf_{B^{(k-1)}} v + \df^\frac{2}{n+2}\right) \leq C_\star \left(\inf_{B^{(k-1)} \cap B^{(k)}} v + \df^\frac{2}{n+2}\right) \\
			&\leq C_\star \left\{\left( \, \dashint_{B^{(k-1)} \cap B^{(k)}} v^s \,dx\right)^{ \!\! \frac{1}{s}} + \df^\frac{2}{n+2}\right\} \\
			&\leq C_\star \left\{4^{n/s}\left( \, \dashint_{B^{(k)}} v^s \,dx\right)^{\!\! \frac{1}{s}} + \df^\frac{2}{n+2}\right\},
		\end{split}
	\end{equation*}
	for all~$k = 2, \ldots, N$. Iterating this estimate along the chain, we obtain
	\begin{equation}
	\label{eq:to-chain-2}
		\left( \, \dashint_{B^{(1)}} v^s \, dx\right)^{\!\! \frac{1}{s}} \leq \left(4^{n/s} C_\star\right)^{\! N-1} \left\{\left( \, \dashint_{B^{(N)}} v^s \, dx\right)^{\!\! \frac{1}{s}} + \df^\frac{2}{n+2}\right\}.
	\end{equation}
	
	Next, we join $B_\delta$ and $B_1$. Arguing as in Lemma 2.2 of~\cite{cicopepo}, we deduce that
	\begin{equation}
		\label{eq:to-chain-1}
		\left( \, \dashint_{B^{(\delta)}} v^s \,dx\right)^{\!\! \frac{1}{s}} \leq C_\flat \, \delta^{- \frac{\beta}{2}}
		\left\{ \left( \, \dashint_{B^{(1)}} v^s \,dx\right)^{\!\! \frac{1}{s}} + \df^\frac{2}{n+2} \right\},
	\end{equation}
	where the constants~$C_\flat\geq 1$ and~$\beta>0$ depend only on~$n$,~$C_0$,~$f_0$,~$s$,~$\norma*{\kappa}_{L^\infty(\R^n)}$, and~$\norma*{f}_{C^{0,1}\left([0,C_0]\right)}$. Analogously,
	\begin{equation*}
		\left( \, \dashint_{B^{(N)}} v^s \,dx\right)^{\!\! \frac{1}{s}} \leq C_\flat \, \delta^{- \frac{\beta}{2}}
		\left\{ \left( \, \dashint_{B^\star} v^s \,dx\right)^{\!\! \frac{1}{s}} + \df^\frac{2}{n+2} \right\},
	\end{equation*}
	from which, thanks to a further application of the weak Harnack inequality and the definition of~$B^\star$, we get
	\begin{equation}
		\label{eq:to-chain-3}
		\left( \, \dashint_{B^{(N)}} v^s \,dx\right)^{\!\! \frac{1}{s}} \leq 2 C_\star C_\flat \, \delta^{- \frac{\beta}{2}}
		\left( \min_{K_\delta} v + \df^\frac{2}{n+2} \right).
	\end{equation}
	Finally, by combining, in this order,~\eqref{eq:to-chain-1},~\eqref{eq:to-chain-2}, and~\eqref{eq:to-chain-3}, we conclude that
	\begin{equation*}
		\left( \, \dashint_{B^{(\delta)}} v^s \,dx\right)^{\!\! \frac{1}{s}} \leq 2 C_\star C_\flat^2 \, \delta^{- \beta} \left(4^{n/s} C_\star\right)^{\! N} \left(\min_{K_\delta} v + \df^\frac{2}{n+2}\right) \!.
	\end{equation*}
	By~\eqref{eq:stima-puntuale}, we also have that~$\abs*{u_{\lambda_\star}-u} \leq v$,  which, in turn, implies
	\begin{equation}
		\label{eq:wH-final}
		\left( \, \dashint_{B^{(\delta)}} \,\abs*{u_{\lambda_\star}-u}^s \,dx\right)^{\! \frac{1}{s}} \leq 2 C_\star C_\flat^2 \, \delta^{- \beta} \left(4^{n/s} C_\star\right)^{\! N} \left\{ \min_{K_\delta} \left(u_{\lambda_\star}-u\right) + C_3 \,\df^\frac{2}{n+2}\right\} \!,
	\end{equation}
	where~$C_3 \coloneqq 2C_2+1$.
	
	We now take~$s =2^\ast/4$ and observe that
	\begin{equation*}
		\dashint_{Q_\delta} \left(u_{\lambda_\star}-u\right)_+^{2^\ast} dx \leq C_0^{2^\ast-s} \dashint_{Q_\delta} \left(u_{\lambda_\star}-u\right)_+^{s} \,dx.
	\end{equation*}
	From this estimate and~\eqref{eq:massa-cubo}, we deduce that
	\begin{equation}
		\label{eq:massa-media}
		\left( \, \dashint_{B^{(\delta)}} \left(u_{\lambda_\star}-u\right)_+^{s} dx\right)^{\!\! \frac{1}{s}} \geq \frac{\mathcal{B}^8}{C_0^7 C_4} \,\abs{\log\df}^{-8 \alpha}.
	\end{equation}
	for some dimensional constant $C_4 \ge 1$. We use~\eqref{eq:massa-media} to estimate the left-hand side of~\eqref{eq:wH-final}, obtaining
	\begin{equation*}
		\min_{K_\delta} \left(u_{\lambda_\star}-u\right) \ge \frac{\delta^\beta \mathcal{B}^8 \! \left(4^{n/s} C_\star\right)^{\! -N}}{2 C_\star C_\flat^2 \, C_0^7 C_4} \, \abs{\log\df}^{-8 \alpha}- C_3 \,\df^\frac{2}{n+2}.
	\end{equation*}
	Moreover, since $4^{n/s} C_\star \geq 1$, recalling~\eqref{eq:Nupperbound} and~\eqref{eq:R0def} we have
	\begin{equation} \label{eq:positività-finale}
		\begin{aligned}
			\min_{K_\delta} \left(u_{\lambda_\star}-u\right) & \ge \frac{\delta^\beta \mathcal{B}^8}{2 C_\star C_\flat^2 \, C_0^7 C_4} \, \abs{\log\df}^{-8 \alpha} \, e^{-C_5 \, \mathcal{B}^{- \frac{2}{n - 2}} \abs{\log\df}^{\frac{2\alpha}{n - 2}}} \\
			& \quad - C_3 \,\df^\frac{2}{n+2},
		\end{aligned}
	\end{equation}
	with
	\begin{equation*}
		C_5 \coloneqq 16 \sqrt{n} \left( \frac{2 C_0^{2^\ast} \Haus^{n - 1}(\sfera^{n - 1})}{n} \right)^{\!\! \frac{1}{n}} \log \left(4^{n/s} C_\star \right) \!.
	\end{equation*}
	
	At this point, we need to choose the parameters~$\alpha$,~$\mathcal{B}$ -- depending only on universal constants -- and $\delta$ -- in relation to~$\df$ as well -- in a way that the right-hand side of~\eqref{eq:positività-finale} is non-negative. First, let
	\begin{align*}
		\mathcal{B} \coloneqq \max \left\{ \left( 2 C_\star C_\flat^2 C_0^7 C_3 C_4 \right)^{\! \frac{1}{8}}, \big( {(n + 2) C_5} \big)^{\! \frac{n - 2}{2}} \right\}
	\end{align*}
	and observe that it only depends on~$n$,~$C_0$,~$f_0$,~$\norma*{\kappa}_{L^\infty(\R^n)}$, and~$\norma*{f}_{C^{0,1}\left([0,C_0]\right)}$. With this choice,~\eqref{eq:positività-finale} simplifies to
	\begin{equation*}
		\min_{K_\delta} \left(u_{\lambda_\star}-u\right) \ge C_3 \left\{\delta^\beta \, e^{-\frac{1}{n+2} \,\abs{\log\df}^{\frac{2\alpha}{n-2}}} \abs{\log\df}^{-8 \alpha} - \df^\frac{2}{n+2} \right\}.
	\end{equation*}
	Then, we take
	\begin{equation} \label{eq:alphaanddeltadef}
		\alpha \coloneqq \frac{n}{2^\ast} = \frac{n-2}{2} \quad \text{and} \quad \delta \coloneqq \left(\df^\frac{1}{n+2} \,\abs{\log\df}^{4n} \right)^{\! \frac{1}{\beta}} \!,
	\end{equation}
	so that
	\begin{equation}
		\label{eq:positività-finale-1}
		\min_{K_\delta} \left(u_{\lambda_\star}-u\right) \ge C_3 \,\df^\frac{2}{n+2} \left( \abs{\log\df}^{8}- 1 \right) \ge C_3 \,\df^\frac{2}{n+2},
	\end{equation}
	for
	\begin{equation*} \label{eq:gammale1/e2}
		\gamma \leq \gamma_3 \coloneqq \frac{1}{e^2}.
	\end{equation*}
	Note that~$\delta \to 0$ as $\df \to 0$.
	
	We conclude this part of the proof by showing that the previous choice of~$\delta$ is consistent with condition~\eqref{eq:cond-delta-1}, at least if~$\gamma$ is small enough. Recalling the definition of~$R_0$ given in~\eqref{eq:R0def} and noticing that~$2^\ast+1 \geq 3$,~$\delta\leq 1$, and~$\abs{\log\df} \geq 1$, we deduce that~\eqref{eq:cond-delta-1} is verified if
	\begin{equation*}
		\delta^3 \leq \frac{1}{C_6^3} \,\abs{\log\df}^{-2n},
	\end{equation*}
	for some constant~$C_6 \ge 1$ depending only on~$n$,~$C_0$,~$f_0$,~$\norma*{\kappa}_{L^\infty(\R^n)}$, and~$\norma*{f}_{C^{0,1}\left([0,C_0]\right)}$. Substituting the definition of $\delta$ in this inequality, we get that~\eqref{eq:cond-delta-1} holds if
	\begin{equation*}
		\df^\frac{1}{n+2} \,\abs{\log\df}^{2n\left(2+\frac{\beta}{3}\right)} \leq \frac{1}{C_6^\beta},
	\end{equation*}
	which is true, provided~$\gamma$ is smaller than a positive constant~$\gamma_4$ depending only on~$n$,~$C_0$,~$f_0$,~$\norma*{\kappa}_{L^\infty(\R^n)}$, and~$\norma*{f}_{C^{0,1}\left([0,C_0]\right)}$.
	
	
	\subsection{Deriving a contradiction.} \label{step:contr}
	
	We aim to show here that
	\begin{equation} \label{eq:contrest}
		\norma*{(u - u_{\lambda_\star - \varepsilon})_+}_{L^{2^\ast} \! (\Sigma_{\lambda_\star - \varepsilon})} \le 4 S^{-2} \,\df,
	\end{equation}
	for every~$\varepsilon > 0$ small enough and provided~$\gamma$ is small enough. Clearly, this would lead to a contradiction with the definition of~$\lambda_\star$ and thus to the validity of~\eqref{eq:u-ulambdastar-small}.
	
	For~$\epsilon > 0$ to be determined later, define
	\begin{equation*}
		E_r^\epsilon \coloneqq \big( {\R^n \setminus B_r(0)} \big) \cap \Sigma_{\lambda_\star-\epsilon} \quad \text{and} \quad S_\delta^\epsilon \coloneqq B_r(0) \cap \left(\Sigma_{\lambda_\star-\epsilon} \setminus \overline{\Sigma_{\lambda_\star+\sqrt{n}\delta}} \right).
	\end{equation*}
	In this way,~$E_r^\epsilon$,~$K_\delta$, and~$S_\delta^\epsilon$ form a decomposition of~$\Sigma_{\lambda_\star-\epsilon}$ made of essentially disjoint sets -- recall that~$K_\delta$ is defined in~\eqref{eq:defK}. We refer to Figure~\ref{fig:dom-dec} for a graphical representation of this decomposition. Exploiting the gradient bound~\eqref{eq:C1boundonu} and~\eqref{eq:positività-finale-1}, we get
	\begin{equation*}
		u_{\lambda_\star-\epsilon}-u \geq u_{\lambda_\star}-u - 2 C_1 \,\epsilon \geq  C_3 \,\df^\frac{2}{n+2} - 2 C_1 \,\epsilon \quad\text{in} \; K_\delta.
	\end{equation*}
	By taking
	\begin{equation*}
		\epsilon \in \left( 0, \frac{C_3}{4C_1} \,\df^\frac{2}{n+2} \right] \!,
	\end{equation*}
	we obtain that $u_{\lambda_\star-\epsilon}-u  \geq 0$ in $K_\delta$. In addition, note that~$\epsilon \to 0$ as~$\df \to 0$.
	
	In order to reach a contradiction, we follow here the argument of Step~2 in the proof of Theorem~3.1 in~\cite{sciu}. Proceeding as in the deduction of~\eqref{eq:test-finale}, we get
	\begin{equation*}
		\begin{split}
			\int_{\Sigma_{\lambda_\star-\epsilon}} \,\abs*{\nabla (u-u_{\lambda_\star-\epsilon})_+}^2 \, dx &= \int_{\Sigma_{\lambda_\star-\epsilon}} \! (\kappa-\kappa_1) f(u) (u-u_{\lambda_\star-\epsilon})_+ \, dx \\
			&\quad- \int_{\Sigma_{\lambda_\star-\epsilon}} \! (\kappa_{\lambda_\star-\epsilon}-\kappa_1) f(u_{\lambda_\star-\epsilon}) (u-u_{\lambda_\star-\epsilon})_+ \, dx \\
			&\quad+ \kappa_1 \int_{\Sigma_{\lambda_\star-\epsilon}} \! \left(f(u)-f(u_{\lambda_\star-\epsilon})\right) (u-u_{\lambda_\star-\epsilon})_+ \, dx.
		\end{split}
	\end{equation*}
	We estimate the first two terms on the right-hand side as in~\ref{step:MPstart}, deducing
	\begin{equation} \label{eq:first2termsest}
		\begin{aligned}
			& \int_{\Sigma_{\lambda_\star-\epsilon}} \! \left(\kappa-\kappa_1\right) f(u) \left(u-u_{\lambda_\star-\epsilon}\right)_+ dx - \int_{\Sigma_{\lambda_\star-\epsilon}} \! (\kappa_{\lambda_\star-\epsilon}-\kappa_1) f(u_{\lambda_\star-\epsilon}) (u-u_{\lambda_\star-\epsilon})_+ \, dx \\
			& \hspace{187pt} \le 2 S^{-1} \,\df \norma*{\nabla (u-u_{\lambda_\star-\epsilon})_+ }_{L^2(\Sigma_{\lambda_\star-\epsilon})}.
		\end{aligned}
	\end{equation}
	To deal with the third term, we exploit the decomposition of~$\Sigma_{\lambda_\star-\epsilon}$ made previously. Using that~$(u - u_{\lambda_\star - \varepsilon})_+ = 0$ in~$K_\delta$, assumptions~\eqref{eq:ip-f-tipolip} and~\eqref{eq:decadimento}, as well as H\"older and Sobolev inequalities, we obtain
	\begin{equation} \label{eq:thirdintest}
		\begin{split}
			& \int_{\Sigma_{\lambda_\star-\epsilon}} \! \left(f(u)-f(u_{\lambda_\star-\epsilon})\right) (u-u_{\lambda_\star-\epsilon})_+ \, dx \le L \int_{\Sigma_{\lambda_\star-\epsilon} \setminus K_\delta} \! u^{p - 1} (u-u_{\lambda_\star-\epsilon})_+^2 \, dx \\
			& \hspace{14pt} \leq L \left\{ \norma*{u}_{L^{2^\ast}\! \left(E_r^\epsilon\right)}^{2^\ast-2} \int_{E_r^\epsilon} (u-u_{\lambda_\star-\epsilon})_+^{2^\ast} \,dx + \norma*{u}_{L^\infty(S_\delta^\varepsilon)}^{p - 1} \int_{S^\epsilon_\delta} (u-u_{\lambda_\star-\epsilon})_+^2 \, dx \right\}, \\
			& \hspace{14pt} \leq L \left\{ S^{-2} \norma*{u}_{L^{2^\ast}\!\left(E_r^\epsilon\right)}^{2^\ast-2} \int_{\Sigma_{\lambda_\star-\epsilon}} \,\abs*{\nabla (u-u_{\lambda_\star-\epsilon})_+}^{2} \,dx + C_0^{p-1} \int_{S^\epsilon_\delta} (u-u_{\lambda_\star-\epsilon})_+^2 \, dx \right\}. \hspace{-10pt}
		\end{split}
	\end{equation}
	To estimate the last integral, we take advantage of the Poincar\'e inequality on slabs. Let~$\mathcal{S} \coloneqq \Sigma_{\lambda_\star - \varepsilon} \setminus \overline{\Sigma_{\lambda_\star + 2 \sqrt{n} \delta + \varepsilon}}$ and~$w$ be the even reflection across~$T_{\lambda_\star + \sqrt{n} \delta}$ of the restriction of~$(u - u_{\lambda_\star - \varepsilon})_+$ to~$\Sigma_{\lambda_\star - \varepsilon} \setminus \Sigma_{\lambda_\star + \sqrt{n} \delta}$. That is, for~$x \in \mathcal{S}$,
	\begin{equation*}
		w(x) \coloneqq \begin{dcases}
			(u - u_{\lambda_\star - \varepsilon})_+(x) & \quad \mbox{if } x \in \Sigma_{\lambda_\star - \varepsilon} \setminus \Sigma_{\lambda_\star + \sqrt{n} \delta}, \\
			(u - u_{\lambda_\star - \varepsilon})_+(x) & \quad \mbox{if } x^{\lambda_\star + \sqrt{n} \delta} \in \Sigma_{\lambda_\star - \varepsilon} \setminus \overline{\Sigma_{\lambda_\star + \sqrt{n} \delta}}.
		\end{dcases}
	\end{equation*}
	Since~$w$ is Lipschitz continuous and vanishes on~$\partial \mathcal{S}$, we may apply Theorem~13.19 in~\cite{leoni} and infer that
	\begin{equation*}
		\int_{\mathcal{S}} w^2 \, dx \le \mathcal{C}_P(S^\epsilon_\delta) \int_{\mathcal{S}} \,\abs*{\nabla w}^2 \, dx,
	\end{equation*}
	with
	\begin{equation} \label{eq:Cpoindef}
		\mathcal{C}_P (S^\epsilon_\delta) \coloneqq 2n \big( {\sqrt{n}\delta+\epsilon} \big)^{\!2}.
	\end{equation}
	From this and the inclusion~$S^\epsilon_\delta \subseteq \Sigma_{\lambda_\star-\epsilon} \setminus \Sigma_{\lambda_\star+\sqrt{n}\delta}$, we immediately get
	\begin{equation*}
		\int_{S^\epsilon_\delta} \left(u-u_{\lambda_\star-\epsilon}\right)_+^2 dx \le \mathcal{C}_P (S^\epsilon_\delta) \int_{\Sigma_{\lambda_\star-\epsilon} \setminus \Sigma_{\lambda_\star+\sqrt{n}\delta}} \,\abs*{\nabla (u-u_{\lambda_\star-\epsilon})_+}^{2} \, dx.
	\end{equation*}
	Combining this with~\eqref{eq:first2termsest} and~\eqref{eq:thirdintest}, we conclude that
	\begin{equation}
		\label{eq:contrad-1}
		\mathcal{A}(r,\epsilon) \, \norma*{\nabla (u-u_{\lambda_\star-\epsilon})_+ }^2_{L^2\left(\Sigma_{\lambda_\star-\epsilon}\right)} 
		\leq  2S^{-1} \,\df \norma*{\nabla (u-u_{\lambda_\star-\epsilon})_+ }_{L^2(\Sigma_{\lambda_\star-\epsilon})},
	\end{equation}
	where
	\begin{equation*}
		\mathcal{A}(r,\epsilon) \coloneqq 1- L\kappa_1 \, S^{-2} \norma*{u}_{L^{2^\ast}\!\left(E_r^\epsilon\right)}^{2^\ast-2} - L\kappa_1 \, C_0^{p-1} \mathcal{C}_P(S^\epsilon_\delta).
	\end{equation*}
	
	Now, we either have that~$\norma*{\nabla (u-u_{\lambda_\star-\epsilon})_+ }_{L^2\left(\Sigma_{\lambda_\star-\epsilon}\right)} = 0$ -- which would lead to~$(u - u_{\lambda_\star - \varepsilon})_+ = 0$ in~$\Sigma_{\lambda_\star - \varepsilon}$ and thus, trivially, to~\eqref{eq:contrest} -- or~$\norma*{\nabla (u-u_{\lambda_\star-\epsilon})_+ }_{L^2\left(\Sigma_{\lambda_\star-\epsilon}\right)} > 0$. Assuming the latter, inequality~\eqref{eq:contrad-1} becomes
	\begin{equation*}
		\label{eq:contrad-1-2}
		\mathcal{A}(r,\epsilon) \norma*{\nabla (u-u_{\lambda_\star-\epsilon})_+ }_{L^2\left(\Sigma_{\lambda_\star-\epsilon}\right)} \leq 2S^{-1} \,\df.
	\end{equation*}
	In order to deduce~\eqref{eq:contrest} from this, it suffices to show that the parameters~$r$ and~$\epsilon$ can be chosen in a way that
	\begin{equation}
		\label{eq:contrad-2}
		\mathcal{A}(r,\epsilon) \geq \frac{1}{2},
	\end{equation}
	that is
	\begin{equation} \label{eq:contrad-3}
		L\kappa_1 \, S^{-2} \norma*{u}_{L^{2^\ast}\!\left(E_r^\epsilon\right)}^{2^\ast-2} + L\kappa_1 \, C_0^{p-1} \mathcal{C}_P (S^\epsilon_\delta) \leq \frac{1}{2}.
	\end{equation}
	
	It is clear that the second summand can be made arbitrarily small by taking~$\varepsilon$ and~$\gamma$ sufficiently tiny. Indeed, recalling definitions~\eqref{eq:alphaanddeltadef},~\eqref{eq:Cpoindef}, and taking advantage of the numerical inequality~$- \log t \le \frac{8 n^2 (n + 2)}{e (n - 2)} \, t^{- \frac{n - 2}{8 n^2 (n + 2)}}$ for every~$t \in \left( 0, e^{- \frac{8 n^2 (n + 2)}{n - 2}} \right]$, we have
	\begin{equation*}
		L \kappa_1 \, C_0^{p-1} \mathcal{C}_P (S^\epsilon_\delta) \le 2 n L \| \kappa \|_{L^\infty(\R^n)} \, C_0^{p-1} \left\{ \sqrt{n} \left( \frac{8 n^2 (n + 2)}{e (n - 2)} \right)^{\!\! \frac{4 n}{\beta}} \! \gamma^\frac{1}{2n\beta} + \epsilon \right\}^{\! 2} \le \frac{1}{4},
	\end{equation*}
	provided~$\varepsilon \le 1 \big/ 8 \sqrt{n L \| \kappa \|_{L^\infty(\R^n)}} C_0^{p - 1}$ and
	\begin{equation*}
		\gamma \le \gamma_4 \coloneqq \min \left\{ e^{- \frac{8 n^2 (n + 2)}{n - 2}}, \left( \frac{1}{32 n^2 L \| \kappa \|_{L^\infty(\R^n)} C_0^{p - 1}} \right)^{\!\! n \beta} \! \left( \frac{e (n - 2)}{8 n^2 (n + 2)} \right)^{\!\! 8 n^2} \right\}.
	\end{equation*}
	On the other hand, the smallness of the first summand in~\eqref{eq:contrad-3} can be achieved by choosing~$r$ sufficiently large. Indeed, recalling~\eqref{eq:decadimento}, we estimate
	\begin{align*}
		L\kappa_1\, S^{-2} \norma*{u}_{L^{2^\ast}\!\left(E_r^\epsilon\right)}^{2^\ast-2} & \le L \| \kappa \|_{L^\infty(\R^n)} \, S^{-2} \left( C_0^{2^\ast} \int_{\R^n \setminus B_r(0)} \frac{dx}{\abs{x}^{2n}} \right)^{\!\! \frac{2^\ast - 2}{2^\ast}} \\
		& = L \| \kappa \|_{L^\infty(\R^n)} \, S^{-2} \left( \frac{C_0^{2^\ast} \Haus^{n - 1}(\sfera^{n - 1})}{n} \right)^{\!\! \frac{2}{n}} r^{-2} \le \frac{1}{4},
	\end{align*}
	provided
	\begin{equation*}
		r \coloneqq \max \left\{ \frac{2 \sqrt{L \| \kappa \|_{L^\infty(\R^n)}}}{S} \left( \frac{C_0^{2^\ast} \Haus^{n - 1}(\sfera^{n - 1})}{n} \right)^{\!\! \frac{1}{n}} \! , \lambda_0 + 1 \right\}.
	\end{equation*}
	Note that~$r \le R_0$ if we restrict ourselves to
	\begin{equation*}
		\gamma \le \gamma_5 \coloneqq \exp \left\{- \frac{2^{\frac{n - 1}{n}} \sqrt{L \| \kappa \|_{L^\infty(\R^n)}} \, \mathcal{B}^{\frac{2}{n - 2}}}{S} \right\}.
	\end{equation*}
	
	From the above computations it follows that~\eqref{eq:contrad-2} holds true and thus that
	\begin{equation*}
		\norma*{\nabla (u-u_{\lambda_\star-\epsilon})_+ }_{L^2\left(\Sigma_{\lambda_\star-\epsilon}\right)} \leq 4S^{-1} \,\df.
	\end{equation*}
	Hence, an application of Sobolev inequality yields the validity of~\eqref{eq:contrest}.
	
	We thus established~\eqref{eq:u-ulambdastar-small} with~$\alpha = \frac{n - 2}{2}$, which, together with~\eqref{eq:stimaperlambdastar}, gives that
	\begin{equation}
		\label{eq:2*-Rn}
		\norma*{u-u_{\lambda_\star}}_{L^{2^\ast}\!(\R^n)} = 2\norma*{u-u_{\lambda_\star}}_{L^{2^\ast}\!(\Sigma_{\lambda_\star})} \leq C \,\abs{\log\df}^{1-\frac{n}{2}},
	\end{equation}
	for some universal constant~$C \ge 1$.
	
	We conclude this step of the proof by showing the validity of the claims made in~\eqref{eq:mainMPclaim}. In~\eqref{eq:stima-puntuale} we have already obtained the pointwise estimate appearing on the first line. The~$L^\infty$-bound on the second line can be easily established by applying the arguments of~\ref{step:int2point} to the functions~$u-u_{\lambda_\star}$ and~$u_{\lambda_\star}-u$. Finally, the~$L^2$-estimate for the gradient can be obtained by observing that
	\begin{equation*}
		\begin{split}
			\left(f(u)- f(u_{\lambda_\star}) \right) \left(u-u_{\lambda_\star}\right) &= \left(f(u)- f(u_{\lambda_\star}) \right) \left(u-u_{\lambda_\star}\right)_+ + \left(f(u_{\lambda_\star}) - f(u) \right) \left(u-u_{\lambda_\star}\right)_{-} \\
			&\leq L \, u^{p-1} \left(u-u_{\lambda_\star}\right)_+^2 + L \, u_{\lambda_\star}^{p-1} \left(u-u_{\lambda_\star}\right)_{-}^2 \\
			&\leq L \,\max\left\{u,u_{\lambda_\star}\right\}^{p-1} \abs*{u-u_{\lambda_\star}}^2,
		\end{split}
	\end{equation*}
	and taking advantage of H\"older inequality, hypothesis~\eqref{eq:decadimento}, and~\eqref{eq:2*-Rn}
	\begin{align*}
		\int_{\R^n} &\,\abs*{\nabla (u-u_{\lambda_\star})}^2 \, dx = \int_{\R^n} \left(\kappa f(u)-\kappa_{\lambda_\star} f(u_{\lambda_\star})\right)  \left(u-u_{\lambda_\star}\right) dx \\
		&= \int_{\R^n} (\kappa-\kappa_1) f(u) \left(u-u_{\lambda_\star}\right) dx - \int_{\R^n} (\kappa_{\lambda_\star}-\kappa_1) f(u_{\lambda_\star}) \left(u-u_{\lambda_\star}\right) dx \\
		& \quad + \kappa_1 \int_{\R^n} \left(f(u)-f(u_{\lambda_\star})\right) \left(u-u_{\lambda_\star}\right) dx \\
		&\leq 2 \df \,\norma*{u-u_{\lambda_\star}}_{L^{2^\ast}\!(\R^n)} + L\kappa_1 \int_{\R^n} \max\left\{u,u_{\lambda_\star}\right\}^{p-1}  \abs*{u-u_{\lambda_\star}}^2 \, dx \\
		&\leq 2 \df \,\norma*{u-u_{\lambda_\star}}_{L^{2^\ast}\!(\R^n)} + L\kappa_1 \left(\int_{\R^n} \max\left\{u,u_{\lambda_\star}\right\}^{2^\ast} \! dx \right)^{\!\! \frac{2^\ast-2}{2^\ast}} \norma*{u-u_{\lambda_\star}}_{L^{2^\ast}\!(\R^n)}^2 \\
		&\leq C \,\abs{\log\df}^{2\left(1-\frac{n}{2}\right)},
	\end{align*}
	for some constant~$C>0$ depending only on~$n$,~$C_0$,~$L$,~$f_0$,~$\norma*{\kappa}_{L^\infty(\R^n)}$, and~$\norma*{f}_{C^{0,1}\left([0,C_0]\right)}$.
	
	
	\subsection{Almost radial symmetry with respect to an approximate center.} \label{step:almostradsym}
	
	Repeating the above procedure along~$n$ linearly independent directions -- say,~$e_1,\ldots,e_n$ -- we identify~$n$ values~$\lambda_k = \lambda_\star(e_k) \in \R$,~$k = 1, \dots, n$, for which~\eqref{eq:mainMPclaim} holds true. Considering the corresponding maximal hyperplanes~$T_{e_k, \lambda_k}$,~$k=1,\dots,n$, we define the approximate center of symmetry~$\mathcal{O}$ as their intersection and the reflection with respect to $\mathcal{O}$ by $u_{\mathcal{O}}(x) \coloneqq u\left(2\lambda_1-x_1,\dots,2\lambda_n-x_n\right)$. As a consequence of this definition and~\eqref{eq:mainMPclaim}, by applying the triangular inequality we deduce
	\begin{equation} \label{eq:u-uOests}
		\norma*{u - u_{\mathcal{O}}}_{L^\infty(\R^n)} + \norma*{\nabla (u - u_{\mathcal{O}})}_{L^2(\R^n)} \le n C_\sharp \,\abs{\log\df}^{1-\frac{n}{2}}.
	\end{equation}
	Also notice that, by virtue of~\eqref{eq:limit-lstar}, we know that
	\begin{equation} \label{eq:Oinlambdacube}
		\mathcal{O} \in {[- \lambda_0, \lambda_0]}^n.
	\end{equation}
	
	Let~$\Sigma_{\omega, \lambda_\star}$ and~$T_{\omega, \lambda_\star}$, with~$\lambda_\star = \lambda_\star(\omega)$, be respectively the maximal half-space and maximal hyperplane corresponding to a direction~$\omega \in \sfera^{n - 1}$. We shall prove that
	\begin{equation} \label{eq:claimhyperclose}
		\mbox{either} \quad \mathcal{O} \in \overline{\Sigma_{\omega, \lambda_\star}} \quad \mbox{or} \quad \mathrm{dist}(\mathcal{O}, T_{\omega, \lambda_\star}) \le C_7 \, \abs*{\log \df}^{\frac{1}{3} \left( 1 - \frac{n}{2} \right)},
	\end{equation}
	for every~$\omega \in \sfera^{n - 1}$ and for some universal constant~$C_7 > 0$.
	
	To establish~\eqref{eq:claimhyperclose}, let~$\omega \in \sfera^{n - 1}$ be fixed and assume that
	\begin{equation} \label{eq:OnotinSigma}
		\mathcal{O} \notin \overline{\Sigma_{\lambda_\star}},
	\end{equation}
	where here and in what follows we drop the reference to~$\omega$ in the notation. We subdivide the space into slabs of fixed width of order~$\mathrm{dist}(\mathcal{O}, T_{\lambda_\star})$ in which the energy has an upper bound, thanks to~\eqref{eq:mainMPclaim}. By comparing this information with the lower bound~\eqref{eq:boundmassa}, we are eventually led to the validity of the second alternative in~\eqref{eq:claimhyperclose}. This geometric construction is partially borrowed from Section~4 in~\cite{cfmnov}. Following are the details of its implementation in our setting.
	
	In view of a change of variables and of the~$\mathcal{D}^{1,2}$-estimate in~\eqref{eq:mainMPclaim}, we see that
	\begin{align*}
		\left| \| \nabla u \|_{L^2(\Sigma_{\lambda_\star})} - \| \nabla u \|_{L^2(\R^n \setminus \Sigma_{\lambda_\star})} \right| & = \left| \| \nabla u \|_{L^2(\Sigma_{\lambda_\star})} - \| \nabla u_{\lambda_\star} \|_{L^2(\Sigma_{\lambda_\star})} \right| \\
		& \le \| \nabla (u - u_{\lambda_\star}) \|_{L^2(\Sigma_{\lambda_\star})} \le \| \nabla (u - u_{\lambda_\star}) \|_{L^2(\R^n)} \\
		& \le C_\sharp \,\abs*{\log \df}^{1 - \frac{n}{2}}.
	\end{align*}
	Using this and~\eqref{eq:decadgrad}, we compute
	\begin{equation} \label{eq:energ-semispazio}
	\begin{aligned}
		& \left| \int_{\Sigma_{\lambda_\star}} \,\abs*{\nabla u}^2 \, dx - \frac{1}{2} \int_{\R^n} \,\abs*{\nabla u}^2 \, dx \right| = \frac{1}{2} \left| \| \nabla u \|_{L^2(\Sigma_{\lambda_\star})}^2 - \| \nabla u \|_{L^2(\R^n \setminus \Sigma_{\lambda_\star})}^2 \right| \\
		& \hspace{45pt} \le \left| \| \nabla u \|_{L^2(\Sigma_{\lambda_\star})} - \| \nabla u \|_{L^2(\R^n \setminus \Sigma_{\lambda_\star})} \right| \, \| \nabla u \|_{L^2(\R^n)} \\
		& \hspace{45pt} \le C_\sharp \,\abs*{\log \df}^{1 - \frac{n}{2}} \, C_1 \left( \int_{\R^n} \frac{dx}{\left( 1 + |x|^{n - 1} \right)^2} \right)^{\!\! \frac{1}{2}} \le C \, \abs*{\log \df}^{1 - \frac{n}{2}},
	\end{aligned}
	\end{equation}
	for some universal constant~$C \ge 1$.
	
	Denote with~$\Sigma_{\lambda_\star \mathcal{O}}$ and~$u_{\lambda_\star \mathcal{O}}$ the reflections of~$\Sigma_{\lambda_\star}$ and~$u_{\lambda_\star}$ with respect to~$\mathcal{O}$. Consider the slab~$\mathcal{S}_1 \coloneqq \R^n \setminus \left( \Sigma_{\lambda_\star} \cup \Sigma_{\lambda_\star\mathcal{O}} \right)$ and observe that~$\mathcal{S}_1 \neq \varnothing$, thanks to~\eqref{eq:OnotinSigma}. The~$\mathcal{D}^{1,2}$-bounds of~\eqref{eq:mainMPclaim} and~\eqref{eq:u-uOests} give that
	\begin{equation*}
		\norma*{ \nabla (u-u_{\lambda_{\star\mathcal{O}}}) }_{L^2(\R^n)} \le C \, \abs{\log\df}^{1-\frac{n}{2}}.
	\end{equation*}
	Through this, we easily obtain estimate~\eqref{eq:energ-semispazio} with~$\Sigma_{\lambda_\star\mathcal{O}}$ in place of~$\Sigma_{\lambda_\star}$. This, together with~\eqref{eq:energ-semispazio}, implies that
	\begin{equation}
		\label{eq:step1}
		\sigma_1 \coloneqq \| \nabla u \|_{L^2(\mathcal{S}_1)} \leq C_8 \,\abs{\log\df}^{\frac{1}{2} \left( 1-\frac{n}{2} \right)},
	\end{equation}
	for some constant~$C_8>0$ depending only on~$n$,~$C_0$,~$L$,~$f_0$,~$\norma*{\kappa}_{L^\infty(\R^n)}$, and~$\norma*{f}_{C^{0,1}\left([0,C_0]\right)}$. Define then the slab~$\mathcal{S}_2$ as the reflection of~$\mathcal{S}_1$ with respect to~$T_{\lambda_\star}$. Since~$|\nabla u_{\lambda_\star}(x)| = |\nabla u(x^{\lambda_\star})|$ for every~$x \in \R^n$, we deduce
	\begin{equation*}
		\int_{\mathcal{S}_2} \,\abs*{\nabla u_{\lambda_\star}}^2 \, dx = \int_{\mathcal{S}_1} \,\abs*{\nabla u}^2 \, dx.
	\end{equation*}
	By combining this estimate with~\eqref{eq:mainMPclaim} and~\eqref{eq:step1}, we get
	\begin{equation}
		\label{eq:step2}
		\sigma_2 \coloneqq \norma*{\nabla u}_{L^2\left(\mathcal{S}_2\right)} \leq \left( C_\sharp + C_8 \right) \abs{\log\df}^{\frac{1}{2}\left(1-\frac{n}{2}\right)}.
	\end{equation}
	Let~$\mathcal{S}_3$ be the reflection of~$\mathcal{S}_2$ with respect to $\mathcal{O}$. Clearly,
	\begin{equation*}
		\int_{\mathcal{S}_3} \,\abs*{\nabla u_{\mathcal{O}}}^2 \, dx = \int_{\mathcal{S}_2} \,\abs*{\nabla u}^2 \, dx.
	\end{equation*}
	Putting together this with~\eqref{eq:u-uOests} and~\eqref{eq:step2}, we get
	\begin{equation*}
		\label{eq:step3}
		\sigma_3 \coloneqq \norma*{\nabla u}_{L^2\left(\mathcal{S}_3\right)} \leq \big( {(n+1)C_\sharp+C_8}\big) \abs{\log\df}^{\frac{1}{2}\left(1-\frac{n}{2}\right)}.
	\end{equation*}
	
	\begin{figure}
		\centering
		\begin{tikzpicture}[x=0.75pt,y=0.75pt,yscale=-1,xscale=1,rotate=45]
			
			\path[name path=s11] (15,-100)--(15,100);
			\path[dashed,name path=s12] (-15,-100)--(-15,100);
			
			\node[inner sep=0.75pt] (A) at (0,90) {$\displaystyle \mathcal{S}_1$};
			\node[inner sep=0.75pt] (B) at (30,90) {$\displaystyle \mathcal{S}_2$};
			
			\node[anchor=south,inner sep=0.75pt] (T) at (16.5,120) {$\displaystyle T_{\lambda_\star}$};
			
			\path[dashed,name path=s2] (45,-100)--(45,100);
			\path[dashed,name path=s4] (-45,-100)--(-45,100);
			
			\node[inner sep=0.75pt] (C) at (-30,90) {$\displaystyle \mathcal{S}_3$};
			\node[inner sep=0.75pt] at (60,90) {$\displaystyle \mathcal{S}_4$};
			
			\path[dashed,name path=s3] (75,-100)--(75,100);
			\path[dashed,name path=s5] (-75,-100)--(-75,100);
			
			\node[inner sep=0.75pt] at (-60,90) {$\displaystyle \mathcal{S}_5$};
			
			\tikzfillbetween[of=s11 and s12,on layer=,]{black!10}
			\tikzfillbetween[of=s2 and s3,on layer=,]{black!10}
			\tikzfillbetween[of=s4 and s5,on layer=,]{black!10}
			
			\draw (0,0) circle (70) node [anchor=south west, left=52pt, inner sep=0.75pt]{$\displaystyle B_{2R_1}(\mathcal{O})$};
			
			\node (OO) at (0,0) {$\mathcal{O}$};
			
			\draw [->] (A) to [bend left=35] (T);
			\draw [->] (T) to [bend left=35] (B);
			
			\draw [->] (B) to [bend left=45] (OO);
			\draw [->] (OO) to [bend left=45] (C);
			
			\draw[<->] (15,-90) to (45,-90);
			\node[inner sep=0.75pt] at (30,-95) {$\displaystyle \mathfrak{s}$};
			
			\filldraw[black] (0,0) circle (0.1pt);
			
			\draw (15,-100)--(15,100);
			\draw[dashed] (-15,-100)--(-15,100);
			
			\draw[dashed] (45,-100)--(45,100);
			\draw[dashed] (-45,-100)--(-45,100);
			
			\draw[dashed] (75,-100)--(75,100);
			\draw[dashed] (-75,-100)--(-75,100);
		\end{tikzpicture}
		\caption{Example of construction of the slabs $\mathcal{S}_k$.}
		\label{fig:slabdiv}
	\end{figure}
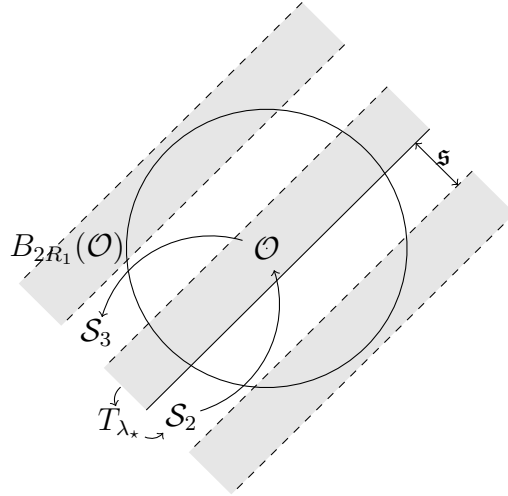
	
	Proceeding inductively, we define~$\mathcal{S}_{2k}$ as the reflection of~$\mathcal{S}_{2k-1}$ with respect to the hyperplane~$T_{\lambda_\star}$ while~$\mathcal{S}_{2k+1}$ as the reflection of~$\mathcal{S}_{2k}$ with respect to~$\mathcal{O}$. We refer to Figure~\ref{fig:slabdiv} for a graphical representation of this construction. Arguing as before, one proves by induction that
	\begin{equation} \label{eq:sigmakest}
		\sigma_k \coloneqq \| \nabla u \|_{L^2(\mathcal{S}_k)} \le k (n+1) (C_\sharp + C_8) \abs{\log\df}^{\frac{1}{2}\left(1-\frac{n}{2}\right)} \quad \mbox{for every } k \in \N.
	\end{equation}
	
	We now show that~$u$ has universally positive energy inside a sufficiently large ball centered at~$\mathcal{O}$. To this aim, using~\eqref{eq:Oinlambdacube} we compute, for~$R_1 \ge \sqrt{n} \, \lambda_0$,
	\begin{align*}
		\int_{\R^n \setminus B_{2R_1}(\mathcal{O})} \,\abs*{\nabla u}^2 \, dx & \le \int_{\R^n \setminus B_{R_1}(0)} \,\abs*{\nabla u}^2 \, dx \le C_1^2 \, \Haus^{n - 1}(\sfera^{n - 1}) \int_{R_1}^{+\infty} \frac{d\rho}{\rho^{n - 1}} \\
		& = \frac{C_1^2 \, \Haus^{n - 1}(\sfera^{n - 1})}{(n - 2) R_1^{n - 2}} \le \frac{1}{2} \, S^2 \mathcal{M}^{\frac{2}{2^\ast}},
	\end{align*}
	where the last inequality holds provided
	\begin{equation*}
		R_1 \ge \mathcal{R} \coloneqq \left( \frac{2 C_1^2 \, \Haus^{n - 1}(\sfera^{n - 1})}{(n - 2) S^2 \mathcal{M}^{\frac{2}{2^\ast}}} \right)^{\!\! \frac{1}{n - 2}} \!.
	\end{equation*}
	Thus, setting~$R_1 \coloneqq \max \left\{\sqrt{n}\,\lambda_0, \mathcal{R} \right\}$ and exploiting Sobolev inequality as well as~\eqref{eq:boundmassa}, we deduce
	\begin{equation*}
		S^2 \mathcal{M}^{\frac{2}{2^\ast}} \leq S^2 \left(\int_{\R^n} u^{2^\ast} dx\right)^{\!\!\frac{2}{2^\ast}} \leq \int_{\R^n} \,\abs*{\nabla u}^2 \, dx.
	\end{equation*}
	As a consequence,
	\begin{equation}
		\label{eq:ener-bolla}
		\int_{B_{2R_1}\left(\mathcal{O}\right)} \,\abs*{\nabla u}^2 \, dx \geq \frac{1}{2} \, S^2 \mathcal{M}^{\frac{2}{2^\ast}}.
	\end{equation}
	
	We are finally in position to prove that~$\mathcal{O}$ is~$O \! \left( |{\log \df}|^{\frac{1}{3} \left( 1 - \frac{n}{2} \right)} \right)$-close to~$T_{\lambda_\star}$. Let~$\mathfrak{s}>0$ be the width of each slab~$\mathcal{S}_k$ constructed above -- i.e.,~$\mathfrak{s} = 2 \, \mathrm{dist} \left( \mathcal{O}, T_{\lambda_\star} \right)$. First, observe that~$\mathfrak{s} < 4R_1$. If not, we would have~$B_{2R_1}(\mathcal{O}) \subseteq \mathcal{S}_1$, which, taking into account~\eqref{eq:step1} and~\eqref{eq:ener-bolla}, is impossible if we take
	\begin{equation*}
		\gamma \le \gamma_6 \coloneqq \exp\left\{-\left( \frac{4 C_8^2}{S^2 \mathcal{M}^{\frac{2}{2^\ast}}} \right)^{\!\! \frac{2}{n - 2}}\right\}.
	\end{equation*}
	It is easy to see that~$B_{2R_1} (\mathcal{O})$ is covered by the first~$k_0 \coloneqq \left\lceil \frac{4 R_1}{\mathfrak{s}} \right\rceil + 1$ slabs~$\mathcal{S}_k$. Using that~$k_0 \le \frac{12 R_1}{\mathfrak{s}}$ and~\eqref{eq:sigmakest}, we deduce
	\begin{equation*}
		\begin{split}
			\int_{B_{2R_1}\left(\mathcal{O}\right)} \,\abs*{\nabla u}^2 \, dx &\leq \sum_{k=1}^{k_0} \sigma_k^2 \leq (n+1)^2 (C_\sharp + C_8)^2 \, \abs{\log\df}^{1-\frac{n}{2}} \sum_{k = 1}^{k_0} k^2 \\
			&\leq \frac{12^3 R_1^3}{\mathfrak{s}^3} \, (n+1)^2 (C_\sharp + C_8)^2 \, \abs{\log\df}^{1-\frac{n}{2}}.
		\end{split}
	\end{equation*}
	Combining this with~\eqref{eq:ener-bolla}, we conclude that
	\begin{equation*}
		\mathrm{dist}\left(\mathcal{O},T_{\lambda_\star}\right)^{3} = \left(\frac{\mathfrak{s}}{2} \right)^{\!\!3} \le \frac{12^3 R_1^3}{4 S^2 \mathcal{M}^{\frac{2}{2^\ast}}} \, (n+1)^2 (C_\sharp + C_8)^2 \, \abs{\log\df}^{1-\frac{n}{2}}.
	\end{equation*}
	Thus, claim~\eqref{eq:claimhyperclose} is established.
	
	Given a direction~$\omega \in \sfera^{n - 1}$, let~$\mu_\star = \mu_\star(\omega) \in \R$ be the unique real number for which~$\mathcal{O} \in T_{\omega, \mu_\star}$, i.e.,~$\mu_\star = \omega \cdot \mathcal{O}$. We now claim that there exists a universal constant~$C_9 > 0$ such that, for every~$\omega \in \sfera^{n - 1}$,
	\begin{equation} \label{eq:2ndmainMPclaim}
		\begin{gathered}
			u(x) - u_{\omega, \lambda}(x) \le C_9 |{\log \df}|^{\frac{1}{3} \left( 1 - \frac{n}{2} \right)} \quad \mbox{for every } x \in \Sigma_{\omega, \lambda} \mbox{ and } \lambda \ge \mu_\star, \\
			\norma*{ u - u_{\omega, \mu_\star} }_{L^\infty(\R^n)} + \norma*{ \nabla (u - u_{\omega, \mu_\star}) }_{L^2(\R^n)}^2 \le C_9 |{\log \df}|^{\frac{1}{3} \left( 1 - \frac{n}{2} \right)}.
		\end{gathered}
	\end{equation}
	Note that this is a version of~\eqref{eq:mainMPclaim} with~$\lambda_\star$ replaced by~$\mu_\star$ and a slightly weaker dependence on~$\df$. By taking advantage of~\eqref{eq:claimhyperclose}, we see that the pointwise inequality on the first line of~\eqref{eq:2ndmainMPclaim} follows easily from that of~\eqref{eq:mainMPclaim} -- trivially in case the first alternative holds in~\eqref{eq:claimhyperclose}, with the help of~\eqref{eq:C1boundonu} if the second one is true. Applying this inequality for both directions~$\omega$ and~$-\omega$, one deduces the validity of the~$L^\infty$-estimate stated on the second line of~\eqref{eq:2ndmainMPclaim} -- here we are taking advantage of the fact that~$\mu_\star(-\omega) = - \mu_\star(\omega)$ and therefore~$u_{-\omega, \mu_\star(-\omega)} = u_{\omega, \mu_\star(\omega)}$ in~$\R^n$. Finally, the~$L^2$-bound for the gradient can be established by testing the equations for~$u$ and~$u_{\omega, \mu_\star}$ against~$u - u_{\omega, \mu_\star}$, changing variables, and using hypotheses~\eqref{eq:ip-f-sub} and~\eqref{eq:decadimento}
	\begin{equation*}
	\begin{split}
		\int_{\R^n} \abs*{\nabla (u-u_{\omega, \mu_\star})}^2 \, dx
		& = 2 \int_{\R^n} \kappa f(u) (u-u_{\omega, \mu_\star}) \, dx \\
		& \le 2 \, C_0^p \, f_0 \, \norma*{ \kappa }_{L^\infty(\R^n)} \, \norma*{ u - u_{\omega, \mu_\star} }_{L^\infty(\R^n)} \int_{\R^n} \frac{dx}{\left(1 + |x|^{n - 2}\right)^p} \\
		& \le C |{\log \df}|^{\frac{1}{3} \left( 1 - \frac{n}{2} \right)},	\end{split}
	\end{equation*}
	for some universal constant~$C > 0$.
		
	We can now prove that~$u$ is almost radially symmetric with respect to~$\mathcal{O}$. Indeed, take~$x,y \in \R^n \setminus \{\mathcal{O}\}$ with~$\abs*{x-\mathcal{O}}=\abs*{y-\mathcal{O}}$ and let~$\omega \in \sfera^{n - 1}$ be a direction parallel to~$y - x$. In this way,~$T_{\omega, \mu_\star}$ is the hyperplane orthogonal to the line segment joining~$x$ and~$y$ passing through their midpoint~$\mathcal{O}$. Thus,~$y=x^{\omega, \mu_\star}$ and the~$L^\infty$-estimate in~\eqref{eq:2ndmainMPclaim} can be read as
	\begin{equation*}
		\label{eq:stima-puntiale-O}
		\abs*{u(x)-u(y)} = \abs*{u(x) - u_{\omega, \mu_\star}(x)} \leq C_9 \abs{\log\df}^{\frac{1}{3}\left(1-\frac{n}{2}\right)}.
	\end{equation*}
	
	Finally, we establish the quantitative symmetry in terms of the~$\mathcal{D}^{1,2}$-norm. Let~$\Theta$ be any rotation with center~$\mathcal{O}$ and write~$u_\Theta(x) \coloneqq u(\Theta x)$ for~$x \in \R^n$. By the Cartan-Dieudonn\'e theorem -- see, e.g.,~\cite[Theorem 4.8.1]{cliff-alg} or~\cite[Theorem 7.1]{q-form} --, there exist~$k \in \{ 1, \ldots, n \}$ directions~$\omega_1, \ldots, \omega_k \in \sfera^{n - 1}$ such that~$\Theta$ can be written as the composition of the reflections across the hyperplanes~$T_{\omega_1, \mu_\star(\omega_1)}, \ldots, T_{\omega_k, \mu_\star(\omega_k)}$. By iterating~$k$ times the~$\mathcal{D}^{1,2}$-estimate in~\eqref{eq:2ndmainMPclaim} we deduce that
	\begin{equation*}
		\norma*{ \nabla (u - u_\Theta) }_{L^2(\R^n)} \le n \, C_9^{\frac{1}{2}} \, |{\log \df}|^{\frac{1}{6} \left( 1 - \frac{n}{2} \right)}.
	\end{equation*}
	
	
	\subsection{Almost monotonicity with respect to $\mathcal{O}$.}
	
	Assuming that~$\kappa \in C^1(\R^n)$, we establish here the validity of~\eqref{eq:quasi-radmon}. To do it, we argue as in~\cite{cicopepo} via a barrier and the weak comparison principle for narrow domains -- the only difference being that, in our case, the domain is an unbounded strip. Up to a rotation with center~$\mathcal{O}$, it suffices to show that
	\begin{equation} \label{eq:step8claim}
		\partial_n u(x) \le C \, \abs{\log\df}^{\frac{1}{3}\left(1-\frac{n}{2}\right)} \quad \mbox{for all } x \in \Sigma_{\mu_\star} \mbox{ with } x' = 0,
	\end{equation}
	for some universal constant~$C > 0$ and where~$\mu_\star = \mu_\star(e_n) = \mathcal{O} \cdot e_n$ is as in~\ref{step:almostradsym} and~$\Sigma_\mu = \Sigma_{e_n, \mu}$ as in~\ref{step:MPstart}.

	Let~$\mu \geq \mu_\star$ be fixed and~$\epsilon > 0$ to be determined later. Define~$N_{\mu,\epsilon} \coloneqq \Sigma_{\mu} \setminus \Sigma_{\mu+\epsilon}$ and note that
	\begin{equation*}
		-\Delta v - \kappa_{\mu} \,c_\mu v = \left(\kappa-\kappa_{\mu}\right) f(u) \eqqcolon f_\mu \quad \text{in}\; N_{\mu,\epsilon},
	\end{equation*}
	where~$v=u-u_{\omega, \mu}$. We estimate
	\begin{equation} \label{eq:est-kcl}
		\norma*{\kappa_{\mu}c_\mu}_{L^\infty(\R^n)} \leq \norma*{\kappa}_{L^\infty(\R^n)} \norma*{f}_{C^{0,1}\left([0,C_0]\right)} \eqqcolon \Gamma
	\end{equation}
	and
	\begin{equation} \label{eq:est-fl}
	\begin{aligned}
		f_\mu(x) &= \big( {\kappa(x',x_n) - \kappa(x',2\mu-x_n)} \big) f(u(x)) \leq 2 f_0 C_0^p \norma*{\partial_n \kappa}_{L^\infty(\R^n)} \left(x_n-\mu\right) \\
		&\leq 2 f_0 C_0^p \norma*{\nabla\kappa}_{L^\infty(\R^n)} \left(x_n-\mu\right) \\
		&\leq 2 f_0 C_0^p \left(\abs{\log\df}^{\frac{1}{3}\left(1-\frac{n}{2}\right)} + \norma*{\nabla\kappa}_{L^\infty(\R^n)}\right) \left(x_n-\mu\right),
	\end{aligned}
	\end{equation}
	for all~$x \in N_{\mu, \varepsilon}$. Consider the barrier given by
	\begin{equation*}
		\label{eq:defbarriera}
		\bar{v}(x) \coloneqq Z \sin \big( {\zeta\left(x_n-\mu\right)} \big),
	\end{equation*}
	for some~$\zeta, Z > 0$ to be specified. Recalling~\eqref{eq:est-kcl}, for every~$x \in N_{\mu,\epsilon}$ we have
	\begin{equation}
		\label{eq:barrier-1}
		-\Delta \bar{v}(x) - \kappa_{\mu}(x) \,c_\mu(x) \bar{v}(x) \geq \left(\zeta^2-\Gamma\right) \bar{v}(x) \geq \bar{v}(x) \geq \frac{2Z\zeta}{\pi} \left(x_n-\mu\right),
	\end{equation}
	for $\zeta^2 \geq \Gamma+1$ and $\zeta\epsilon \leq \pi/2$. Taking into account~\eqref{eq:est-fl}, this implies that
	\begin{equation*}
		\label{eq:barrier-2}
		-\Delta \bar{v} - \kappa_{\mu} \,c_\mu \bar{v} \geq f_\mu \quad \text{in } N_{\mu,\epsilon},
	\end{equation*}
	provided that
	\begin{equation*}
		Z\zeta \geq \pi f_0 C_0^p \left(\abs{\log\df}^{\frac{1}{3}\left(1-\frac{n}{2}\right)} + \norma*{\nabla\kappa}_{L^\infty(\R^n)}\right).
	\end{equation*}

	We take care of the boundary conditions. First, note that~$v=\bar{v}=0$ on~$T_\mu = \partial \Sigma_\mu$, whereas, by~\eqref{eq:2ndmainMPclaim},
	\begin{equation}
		\label{eq:barrier-3}
		v - \bar{v} \leq C_9 \left(\abs{\log\df}^{\frac{1}{3}\left(1-\frac{n}{2}\right)} + \norma*{\nabla\kappa}_{L^\infty(\R^n)}\right) - \frac{2Z\zeta\epsilon}{\pi}\leq 0 \quad\mbox{on } T_{\mu+\epsilon},
	\end{equation}
	provided that
	\begin{equation*}
		Z\zeta\epsilon \geq \frac{C_9\pi}{2} \left(\abs{\log\df}^{\frac{1}{3}\left(1-\frac{n}{2}\right)} + \norma*{\nabla\kappa}_{L^\infty(\R^n)}\right) \!.
	\end{equation*}
	Thus, $\bar{v}-v \geq 0$ on $\partial N_{\mu,\epsilon}$.
	
	Let~$\epsilon_0>0$, depending on~$\norma*{\kappa}_{L^\infty(\R^n)}$ and~$\norma*{f}_{C^{0,1}\left([0,C_0]\right)}$, be so small that the narrow region principle of Corollary~7.4.1 in~\cite{cl-book} holds in~$N_{\mu,\epsilon}$ for every~$\epsilon\in(0,\epsilon_0]$. In light of this, we infer that~$\bar{v} - v \ge 0$ in~$N_{\mu, \varepsilon}$, provided
	\begin{equation}
		\label{eq:cond-liminf}
		\liminf_{x \in N_{\mu,\epsilon}, \,\abs{x}\to+\infty} \left(\bar{v}-v\right)(x) \geq 0.
	\end{equation}
	We now fix the parameters in a way that~\eqref{eq:barrier-1}-\eqref{eq:barrier-3} hold true
	\begin{gather*}
		\zeta \coloneqq \sqrt{\Gamma+1}, \quad \epsilon \coloneqq \min\left\{\frac{\epsilon_0}{2},\frac{\pi}{2\zeta}\right\},\\
		Z \coloneqq \frac{\pi}{\zeta}\max\left\{f_0 C_0^p,\frac{C_9}{2\epsilon}\right\} \left(\abs{\log\df}^{\frac{1}{3}\left(1-\frac{n}{2}\right)} + \norma*{\nabla\kappa}_{L^\infty(\R^n)}\right) \!.
	\end{gather*}
	Moreover, observe that~\eqref{eq:cond-liminf} is trivially true since $v$ decays at infinity and~$\bar{v}\geq0$ in~$N_{\mu, \varepsilon}$. Hence, we deduce that~$\bar{v} - v \ge 0$ in~$N_{\mu, \varepsilon}$, meaning that
	\begin{equation*}
		\begin{split}
			u(x)-u(x^\mu) &= v(x) \leq \bar{v}(x) \\
			&\leq 2 C_{10} \left(\abs{\log\df}^{\frac{1}{3}\left(1-\frac{n}{2}\right)} + \norma*{\nabla\kappa}_{L^\infty(\R^n)}\right) \left(x_n-\mu\right) \quad \text{for every} \; x\in N_{\mu,\epsilon},
		\end{split}
	\end{equation*}
	for some constant~$C_{10} > 0$ depending only on~$n$,~$C_0$,~$L$,~$f_0$,~$\norma*{\kappa}_{L^\infty(\R^n)}$, and~$\norma*{f}_{C^{0,1}\left([0,C_0]\right)}$.
	
	Dividing both sides of the above inequality by~$x_n - \mu$ and letting~$x_n \to \mu^+$, we get
	\begin{equation*}
		\partial_n u(x) \leq C_{10} \left(\abs{\log\df}^{\frac{1}{3}\left(1-\frac{n}{2}\right)} + \norma*{\nabla\kappa}_{L^\infty(\R^n)}\right) \quad\mbox{for all } x\in N_{\mu,\epsilon} \mbox{ and } \mu \geq \mu_\star.
	\end{equation*}
	That is
	\begin{equation*}
		\label{eq:almostmonot}
		\partial_n u(x) \leq C_{10} \left(\abs{\log\df}^{\frac{1}{3}\left(1-\frac{n}{2}\right)} + \norma*{\nabla\kappa}_{L^\infty(\R^n)}\right) \quad \mbox{for all}\; x\in\Sigma_{\mu_\star}.
	\end{equation*}
	By specializing this to the points of the form~$(0,x_n)$ with~$x_n \ge \mu_\star$, claim~\eqref{eq:step8claim} follows.

	
	\section{Proof of Theorem~\ref{th:maintheorem-kelvin}}
	\label{sec:proof3}
	
	\noindent
	In this section, we will refer to a constant as \textit{universal} if it depends only on~$n$,~$C_0$,~$L$, $f_0$,~$\norma*{\kappa}_{L^\infty(\R^n)}$,~$\mu$,~$\nu$, and~$R_0$.
	
	\renewcommand\thesubsection{\bfseries Step \arabic{subsection}}
	
	\subsection{Preliminary observations and introduction of the Kelvin transform.}
	First, observe that, as in~\ref{step:prel-obs} of the proof of Theorem~\ref{th:maintheorem-f}, we can easily deduce that
	\begin{equation}
		\label{eq:C1boundonu-1}
		\norma*{u}_{C^1(\R^n)} \leq C_1
	\end{equation}
	for some constant~$C_1>0$ depending only on~$n$,~$C_0$, and~$\norma*{\kappa}_{L^\infty(\R^n)}$. 
	
	Secondly, we note that it suffices to prove Theorem~\ref{th:maintheorem-kelvin} when~$\defi(\kappa)$ is smaller than a universal constant~$\gamma \in \left( 0, \frac{1}{2} \right]$. Indeed, when~$\defi(\kappa) > \gamma$, inequality~\eqref{eq:alm_sym_thm1} is a trivial consequence of the bound~\eqref{eq:C1boundonu-1} as in~\ref{step:prel-obs} of the proof of Theorem~\ref{th:maintheorem-f}. Therefore, in what follows we will assume that
	\begin{equation*}
		\defi(\kappa) \le \gamma,
	\end{equation*}
	for some small~$\gamma \in \left( 0, \frac{1}{2} \right]$ to be later determined in dependence of universal quantities.
	
	For simplicity of exposition, in what follows we shall suppose~$f$ to be defined in the whole~$[0, +\infty)$ and to satisfy 
		\begin{equation} \label{eq:fextprops}
			\begin{alignedat}{4}
				|f(u)| & \le f_0 u^p && \quad \mbox{for every } u \ge 0, \\
				0 \leq \frac{f(u_2) - f(u_1)}{u_2 - u_1} & \le L u_2^{p - 1} && \quad \mbox{for every } 0 \leq u_1 < u_2, \\
				&&& \hspace{-120pt} \mbox{the function } u \mapsto \frac{f(u)}{u^p} \mbox{ is non-increasing in } (0, +\infty).
			\end{alignedat}
		\end{equation}
	This can be achieved for instance by replacing~$f$ with its continuous extension to~$[0, +\infty)$ that is constant past~$C_0$.
	
	We define the Kelvin transform of~$u$ as
	\begin{equation*}
		\label{eq:def-KT}
		v(x) \coloneqq \mathcal{K}[u](x) = \frac{1}{\abs{x}^{n-2}} \, u\left(\frac{x}{\abs{x}^2}\right) \quad \mbox{for } x \in \R^n \setminus\{0\}.
	\end{equation*}
	It is standard to verify that~$v \in W^{1,2}_{\loc}(\R^n \setminus\{0\}) \cap C^1(\R^n \setminus\{0\})$ weakly solves
	\begin{equation}
		\label{eq:eqv}
		\Delta v + \kappa^\ast g(\cdot, v) = 0 \quad \mbox{in } \R^n \setminus\{0\},
	\end{equation}
	where
	\begin{equation*}
		\kappa^\ast(x) \coloneqq \kappa\left(\frac{x}{\abs{x}^2}\right) \quad \mbox{and} \quad g(x, v) \coloneqq \frac{1}{|x|^{n + 2}} \, f \big( {|x|^{n - 2} v} \big) \quad \mbox{for } x \in \R^n \setminus \{ 0 \}, \, v \ge 0.
	\end{equation*}
	Note that, since~$\kappa \in L^\infty(\R^n)$, then also~$\kappa^\ast \in L^\infty(\R^n)$ with equality of norms and oscillations. Moreover, from~\eqref{eq:fextprops} it easily follows that~$g$ satisfies
	\begin{alignat}{4}
	\label{eq:ggrowth}
		|g(x, v)| & \le f_0 v^p && \quad \mbox{for every } x \in \R^n \setminus \{ 0 \}, \, v \ge 0, \\
		\label{eq:gLipgrowth}
		0 \leq \frac{g(x, v_2) - g(x, v_1)}{v_2 - v_1} & \le L v_2^{p - 1} && \quad \mbox{for every } x \in \R^n \setminus \{ 0 \}, \, 0 \leq v_1 < v_2,
		\end{alignat}
	and, from the third condition in~\eqref{eq:fextprops}, that
	\begin{equation} \label{eq:graddecr}
		g(\cdot, v) \mbox{ is radially symmetric and radially non-increasing, for every } v > 0.
	\end{equation}
	Note that~\eqref{eq:gLipgrowth} implies that the map~$v \mapsto g(\cdot,v)$ is non-decreasing.
	
	We further observe that~$v$ is a weak solution of~\eqref{eq:eqv} in~$\R^n \setminus\{0\}$, it satisfies
	\begin{equation}
		\label{eq:bound-v}
		v(x) \leq \frac{C_0}{\abs{x}^{n-2}} \quad \mbox{for every } x \in \R^n \setminus\{0\},
	\end{equation}
	and, for any direction~$\omega \in \sfera^{n-1}$ and any~$\lambda>0$, it holds
	\begin{equation}
		\label{eq:regular-v}
		v \in L^{2^\ast} \! (\Sigma_{\omega,\lambda}) \cap L^{\infty}(\Sigma_{\omega,\lambda}) \cap C^0(\overline{\Sigma_{\omega,\lambda}}).
	\end{equation}
	Here and in the remainder of the proof, we make use of the notation introduced in~\ref{step:prel-obs} of the proof of Theorem~\ref{th:maintheorem-f}. In addition, we rewrite the left-hand inequality of~\eqref{eq:u-below} as well as the~$C^1$-bound~\eqref{eq:C1boundonu-1} in terms of the Kelvin transform, deducing
	\begin{equation}
		\label{eq:v-below}
		v(x) \geq \frac{1}{C_0} \,\frac{\defi(\kappa)^{\sigma}}{\abs*{x}^\mu} \quad \mbox{for every } x \in B_{1/R_0}(0) \setminus \{0\}
	\end{equation}
	and
	\begin{equation}
		\label{eq:bound-gradv}
		\abs*{\nabla v(x)} \leq C_2 \, \frac{1 + |x|}{|x|^n} \quad \mbox{for every } x \in \R^n \setminus \{0\},
	\end{equation}
	for some~$C_2 \ge 1$ depending on universal constants.
	
	By arguing as in~\cite{esp-fs}, we deduce that, for every direction~$\omega \in \sfera^{n-1}$ and every~$\lambda>0$,
	\begin{equation*}
		(v-v_{\omega,\lambda})_+ \in L^{2^\ast} \! (\Sigma_{\omega,\lambda}) \quad \mbox{and} \quad \nabla (v-v_{\omega,\lambda})_+ \in L^{2}(\Sigma_{\omega,\lambda}),
	\end{equation*}
	with
	\begin{equation*}
		\begin{aligned}
			\label{eq:sobolev-w}
			\norma*{(v-v_{\omega,\lambda})_+}^2_{L^{2^\ast}\!(\Sigma_{\omega,\lambda})} & \leq S^{-2} \int_{\Sigma_{\omega,\lambda}} \abs*{\nabla(v-v_{\omega,\lambda})_+}^2 \, dx \\
			& \leq 4 f_0 S^{-2} \| \kappa \|_{L^\infty(\R^n)} \norma*{v}^{2^\ast}_{L^{2^\ast}\!(\Sigma_{\omega,\lambda})}.
		\end{aligned}
	\end{equation*}
	To obtain the second inequality here, one essentially needs to test equation~\eqref{eq:eqv} and the one for~$v_{\omega, \lambda}$ against~$(v - v_{\omega, \lambda})_+$ times a sequence of cutoffs encoding the fact that the reflection of the origin across~$T_{\omega, \lambda}$ -- the singular set in~$\Sigma_{\omega, \lambda}$ -- has zero capacity. See the proof of Lemma~4.3 in~\cite{esp-fs} for the details of this computation.
	
	Our goal in Steps~2-5 is to show that for to every~$\omega \in \sfera^{n - 1}$ we have
	\begin{equation}
	\label{eq:mainclaim-Kelvin}
		\abs*{v(x)-v_{\omega,0}(x)} \leq  C_\flat \defi(\kappa)^{\alpha_\flat} \quad \mbox{for every } x \in \R^n \setminus B_{r_\flat}(0),
	\end{equation}
	with
	\begin{equation}
	\label{eq:rflatdef}
		r_\flat \coloneqq C_\flat \defi(\kappa)^{\frac{\alpha_\flat}{n}},
	\end{equation}	
	for some universal constants~$C_\flat \ge 1$ and~$\alpha_\flat \in (0, 1)$, provided~$\gamma$ is smaller than a universal constant~$\gamma_{\flat} \in \left( 0, \frac{1}{2} \right]$. Note that~$r_\flat$ goes to zero as~$\defi(\kappa) \rightarrow 0$.
	
	
	\subsection{Starting the moving planes procedure.} \label{step:MPstart-kelvin}
	
	Of course, after a rotation it suffices to verify~\eqref{eq:mainclaim-Kelvin} for~$\omega = e_n$. Under this assumption, we drop any reference to the dependence on~$\omega$ and adopt the notation of~\ref{step:MPstart} in the proof of Theorem~\ref{th:maintheorem-f}.
	
	We show here that the moving planes method can be started. To this aim, we first claim that
	\begin{equation} \label{eq:L2starboundMPbegins}
		\norma*{(v-v_\lambda)_+}_{L^{2^\ast} \! (\Sigma_{\lambda})} \leq C_3 \defi(\kappa) \quad \mbox{for every } \lambda \ge \lambda_0,
	\end{equation}
	for some constants~$C_3, \lambda_0 \ge 1$ depending only on~$n$,~$C_0$, and~$\| \kappa \|_{L^\infty(\R^n)}$.
	
	Our strategy to obtain~\eqref{eq:L2starboundMPbegins} is similar to the one of the aforementioned~\ref{step:MPstart} -- the only difference lying in the need to avoid the singularity. To do so, we argue as in Step~1 of the proof of Theorem~1.5 in~\cite{esp-fs} and test both equations for~$v$ and~$v_\lambda$ against the function~$\psi_\varepsilon^2 \varphi_R^2 (v - v_\lambda)_+ \chi_{\Sigma_\lambda}$. Here,~$\varphi_R$ is a standard cutoff function acting between the balls~$B_R(0)$ and~$B_{2 R}(0)$ with~$R > 2 \lambda$ -- needed to deal with the unboundedness of~$\Sigma_\lambda$ --, while the family~$\{\psi_\varepsilon\}_\epsilon \subseteq C^{0, 1}(\R^n; [0, 1])$ encodes, as~$\varepsilon \rightarrow 0^+$, the capacity of~$\{0^\lambda\}$ being zero -- see the proof of~\cite[Lemma~4.3]{esp-fs} for more details. We get
	\begin{equation*}
		\int_{\Sigma_{\lambda}} \abs*{\nabla (v-v_\lambda)_+}^2 \psi_\epsilon^2 \,\varphi_R^2 \, dx = I_1 + I_2 + I_3,
	\end{equation*}
	where, for~$\lambda > 0$ large,
		\begin{align*}
			I_1 & \coloneqq -2 \int_{\Sigma_{\lambda}} \nabla (v-v_\lambda)_+ \cdot \nabla\psi_\epsilon (v-v_\lambda)_+ \psi_\epsilon \,\varphi_R^2 \, dx, \\
			I_2 & \coloneqq -2 \int_{\Sigma_{\lambda}} \nabla (v-v_\lambda)_+ \cdot \nabla\varphi_R (v-v_\lambda)_+ \psi_\epsilon^2 \,\varphi_R \, dx, \\
			I_3 & \coloneqq \int_{\Sigma_{\lambda}} \big( {\kappa^\ast G - \kappa^\ast_\lambda G_\lambda} \big) (v-v_\lambda)_+ \psi_\epsilon^2 \,\varphi_R^2 \, dx,
		\end{align*}
		with
		\begin{equation*}
			G(x) \coloneqq g(x, v(x)) \quad \mbox{and} \quad G_\lambda(x) = g(x^\lambda, v(x^\lambda)) = g(x^\lambda, v_\lambda(x)) \quad \mbox{for } x \in \Sigma_\lambda.
		\end{equation*}
	We estimate~$I_1$ and~$I_2$ as in the proof of Lemma~4.3 in~\cite{esp-fs}, obtaining
	\begin{align*}
		|I_1| & \le \frac{1}{4} \int_{\Sigma_{\lambda}} \abs*{\nabla (v-v_\lambda)_+}^2 \psi_\epsilon^2 \,\varphi_R^2 \, dx + 16 \varepsilon \| v \|_{L^\infty(\Sigma_\lambda)}, \\
		|I_2| & \le \frac{1}{4} \int_{\Sigma_{\lambda}} \abs*{\nabla (v-v_\lambda)_+}^2 \psi_\epsilon^2 \,\varphi_R^2 \, dx + 64 |B_1|^{\frac{2}{n}} \left( \int_{\Sigma_\lambda \setminus B_R(0)} v^{2^\ast} dx \right)^{\!\! \frac{2}{2^\ast}} \!.
	\end{align*}
	On the other hand, using~\eqref{eq:graddecr} -- since~$|x| \geq |x^\lambda|$ --,~\eqref{eq:ggrowth}, and~\eqref{eq:gLipgrowth}, we have that
		\begin{align*}
			& \kappa^*(x) G(x) - \kappa_\lambda^*(x) G_\lambda(x) = \kappa^*(x) g(x, v(x)) - \kappa^*(x^\lambda) g(x^\lambda, v_\lambda(x)) \\
			& \hspace{40pt} \le \kappa^*(x) g(x, v(x)) - \kappa^*(x^\lambda) g(x, v_\lambda(x)) \\
			& \hspace{40pt} = \left( \kappa^*(x) - \kappa^*(x^\lambda) \right) g(x, v(x)) + \kappa^*(x^\lambda) \left( g(x, v(x)) - g(x, v_\lambda(x)) \right) \\
			& \hspace{40pt} \le f_0 \, \mbox{osc}(\kappa^*) v(x)^p + L \| \kappa^* \|_{L^\infty(\R^n)} v(x)^{p - 1} \left(v(x) - v_\lambda(x)\right)_+
		\end{align*}
		for every~$x \in \Sigma_\lambda \setminus \{ 0^\lambda \}$. Then, using H\"older and Sobolev inequalities,
		\begin{equation*}
			\begin{split}
				I_3 & \le f_0 \, \mathrm{osc}(\kappa^*) \int_{\Sigma_{\lambda}} v^p (v-v_\lambda)_+ \psi_\epsilon^2 \,\varphi_R^2 \, dx + L \| \kappa^* \|_{L^\infty(\R^n)} \int_{\Sigma_{\lambda}} v^{p - 1} (v-v_\lambda)_+^2 \psi_\epsilon^2 \,\varphi_R^2 \, dx \\
				& \le f_0 \defi(\kappa) \| v \|_{L^{2^\ast} \! (\Sigma_\lambda)}^p \| (v - v_\lambda)_+ \|_{L^{2^\ast} \! (\Sigma_\lambda)} + L \| \kappa \|_{L^\infty(\R^n)} \| v \|_{L^{2^\ast} \! (\Sigma_\lambda)}^{p - 1} \| (v - v_\lambda)_+ \|_{L^{2^\ast} \! (\Sigma_\lambda)}^2 \\
				& \le f_0 S^{-1} \defi(\kappa) \| v \|_{L^{2^\ast} \! (\Sigma_\lambda)}^p \| \nabla (v - v_\lambda)_+ \|_{L^2(\Sigma_\lambda)} \\
				& \quad + L S^{-2} \| \kappa \|_{L^\infty(\R^n)} \| v \|_{L^{2^\ast} \! (\Sigma_\lambda)}^{p - 1} \| \nabla (v - v_\lambda)_+ \|_{L^2 (\Sigma_\lambda)}^2.
			\end{split}
		\end{equation*}
	Piecing these three estimates together and letting~$\epsilon \to 0^+$ and~$R \to +\infty$, by Fatou's lemma and recalling~\eqref{eq:regular-v}, we deduce
	\begin{equation}
	\label{eq:v-vl-toest}
		\begin{split}
				& \left( 1 - \frac{2 L}{S^2} \, \| \kappa \|_{L^\infty(\R^n)} \| v \|_{L^{2^\ast} \! (\Sigma_\lambda)}^{p - 1} \right) \norma*{\nabla (v-v_\lambda)_+}^2_{L^{2}(\Sigma_{\lambda})} \\
				& \hspace{80pt} \leq 2 f_0 S^{-1} \defi(\kappa) \| v \|_{L^{2^\ast} \! (\Sigma_\lambda)}^p \| \nabla (v - v_\lambda)_+ \|_{L^2(\Sigma_\lambda)} ,
		\end{split}
	\end{equation}
	which resembles~\eqref{eq:est-1-def}. Then, we take advantage of the bound~\eqref{eq:bound-v} and of the computation~\eqref{eq:tailubound} to see that 
	\begin{equation*} \label{eq:L2starestforv}
		\norma*{v}^{p - 1}_{L^{2^\ast}\!(\Sigma_{\lambda})} \leq \frac{S^2}{4 L \,\norma*{\kappa}_{L^\infty(\R^n)}},
	\end{equation*}
	provided~$\lambda$ is larger that
	\begin{equation*}
		\widehat{\lambda}_0 \coloneqq \frac{\sqrt{ 4 L \norma*{\kappa}_{L^\infty(\R^n)}}}{S} \left( \frac{\Haus^{n - 1}(\sfera^{n - 1}) \, C_0^{2^\ast}}{n} \right)^{\!\! \frac{1}{n}} \!.
	\end{equation*}
	Hence, for~$\lambda \ge \widehat{\lambda}_0$, inequality~\eqref{eq:v-vl-toest} gives that
	\begin{equation}
		\label{eq:est-toreach}
		\norma*{\nabla (v-v_\lambda)_+}^2_{L^{2}(\Sigma_{\lambda})} \leq C \defi(\kappa) \,\norma*{\nabla (v-v_\lambda)_+}_{L^{2}(\Sigma_{\lambda})}
	\end{equation}
	for some constant~$C \ge 1$ depending only on~$n$,~$L$,~$f_0$, and~$\| \kappa \|_{L^\infty(\R^n)}$. To achieve~\eqref{eq:L2starboundMPbegins}, we may clearly suppose without loss of generality that
	\begin{equation*}
		\label{eq:grad>0}
		\norma*{\nabla (v-v_\lambda)_+}_{L^{2}(\Sigma_{\lambda})}>0.
	\end{equation*}
	Under this assumption, inequality~\eqref{eq:est-toreach} simplifies to
	\begin{equation*}
		\norma*{\nabla (v-v_\lambda)_+}_{L^{2}(\Sigma_{\lambda})} \leq C \defi(\kappa) \quad \mbox{for every } \lambda \ge \widehat{\lambda}_0,
	\end{equation*}
	which immediately leads to~\eqref{eq:L2starboundMPbegins}, with~$\lambda_0 \coloneqq \max \big\{ {\widehat{\lambda}_0, 1} \big\}$, after an application of Sobolev inequality.
	
	From the integral estimate~\eqref{eq:L2starboundMPbegins}, we plan to deduce a pointwise bound for~$\left(v-v_\lambda\right)_+$. Let~$\lambda \geq \lambda_0$ be fixed and observe that, thanks to~\eqref{eq:ggrowth} and~\eqref{eq:graddecr},
	\begin{equation*}
	\label{eq:eq-differenza-v}
		-\Delta (v-v_{\lambda}) \le \kappa^\ast c_{\lambda} (v-v_{\lambda}) + f_0 \defi(\kappa) \, v_\lambda^p \quad \mbox{in } \Sigma_{\lambda} \setminus \{0^\lambda\},
	\end{equation*}
	where
	\begin{equation}
	\label{eq:clambdadefkelv}
		c_{\lambda}(x) \coloneqq
		\begin{dcases}
			\frac{g(x, v(x)) - g(x, v(x^{\lambda}))}{v(x)-v(x^{\lambda})}	& \quad \mbox{if } v(x) \neq v(x^{\lambda}), \\
			0								& \quad \mbox{if } v(x) = v(x^{\lambda}).
		\end{dcases}
	\end{equation}
	Note, as of now, that~$c_{\lambda} \geq 0$ by the monotonicity of~$g$. Due to the presence of the singularity~$0^\lambda$, we cannot argue exactly as in~\ref{step:int2point} of the proof of Theorem~\ref{th:maintheorem-f}. To circumvent this issue, we use a weak version of Kato's inequality -- see~\cite[Proposition~3.2]{cleon} -- and deduce that
	\begin{equation}
	\label{eq:eq-differenza-v+}
		-\Delta (v-v_{\lambda})_+ \leq \kappa^\ast c_{\lambda} (v-v_{\lambda})_+ + f_0 \defi(\kappa) \, v^p \chi_{\{v>v_\lambda\}} \eqqcolon h_\lambda \quad \mbox{in } \Sigma_{\lambda} \setminus \{ 0^\lambda \},
	\end{equation}
	in the sense that
	\begin{equation} \label{eq:eq-differenza-v+weak}
		\int_{\Sigma_\lambda} \nabla (v-v_{\lambda})_+ \cdot \nabla \varphi \, dx \le \int_{\Sigma_\lambda} h_\lambda \varphi \, dx,
	\end{equation}
	for every non-negative function~$\varphi \in C^\infty_c(\Sigma_\lambda \setminus \{ 0^\lambda \})$. We now observe that, by~\eqref{eq:bound-v} and~\eqref{eq:v-below}, there exists a small radius~$r_0 \in (0, 1)$, depending on~$n$,~$C_0$,~$\mu$,~$\sigma$,~$R_0$, and also on~$\defi(\kappa)$, such that
	\begin{equation}
		\label{eq:0-neighb}
		\left(v-v_{\lambda}\right)_+ = 0 \quad \mbox{in } B_{r_0}(0^\lambda) \setminus \{0^\lambda\}.
	\end{equation}
	From this it is easy to see that~\eqref{eq:eq-differenza-v+} actually holds weakly in the whole~$\Sigma_\lambda$, i.e., that~\eqref{eq:eq-differenza-v+weak} is true for every non-negative~$\varphi \in W^{1,2}(\Sigma_\lambda)$ with compact support. Indeed, take first~$\varphi \in C^\infty_c(\Sigma_\lambda)$ and let~$\eta_\epsilon$ be a smooth cutoff satisfying~$0 \leq \eta_\epsilon \leq 1$ in~$\R^n$,~$\eta_\epsilon=0$ in~$B_\epsilon (0^\lambda)$, and~$\eta_\epsilon=1$ outside of~$B_{2\epsilon} (0^\lambda)$, for~$\varepsilon \in \left(0, 1/2 \right)$. Since~$\varphi \eta_\epsilon \in C^\infty_c\left(\Sigma_{\lambda}\right)$ vanishes in a neighborhood of~$0^\lambda$, it can be used as test function in~\eqref{eq:eq-differenza-v+weak} to get
	\begin{equation*}
		\int_{\Sigma_{\lambda}} \,\eta_\epsilon \nabla (v-v_{\lambda})_+ \cdot \nabla\varphi \, dx + \int_{B_{2\epsilon}(0^\lambda) \setminus B_{\epsilon}(0^\lambda)} \varphi \, \nabla (v-v_{\lambda})_+ \cdot \nabla\eta_\epsilon \, dx \leq \int_{\Sigma_{\lambda}} h_\lambda \varphi \eta_\epsilon \, dx.
	\end{equation*}
	For~$\epsilon$ sufficiently small, we have that~$B_{2\epsilon}(0^\lambda) \setminus B_{\epsilon}(0^\lambda) \subseteq B_{r_0}(0^\lambda) \setminus \{0^\lambda\}$. By this and~\eqref{eq:0-neighb}, the second integral on the left-hand side vanishes, while in the first one the integrand can actually be taken to be~$\nabla (v-v_{\lambda})_+ \cdot \nabla\varphi$. Therefore, by estimating the right-hand side -- observe that~$h_\lambda \ge 0$ in~$\Sigma_\lambda$ --, we conclude that~\eqref{eq:eq-differenza-v+weak} holds for every non-negative~$\varphi \in C^\infty_c(\Sigma_\lambda)$. By density,~\eqref{eq:eq-differenza-v+weak} extends to every non-negative function~$\varphi \in W^{1,2}(\Sigma_\lambda)$ with compact support -- here, we are using that~$(v - v_\lambda)_+ \in W^{1,2}(\Sigma_\lambda \cap B_R(0))$ for every~$R > 0$, which holds true thanks to~\eqref{eq:0-neighb}.
	
	We now rewrite~\eqref{eq:eq-differenza-v+} in the form
	\begin{equation*}
		\Delta (v-v_{\lambda})_+ + \kappa^\ast c_{\lambda} \chi_{\{v>v_\lambda\}} (v-v_{\lambda})_+ \geq - f_0 \defi(\kappa) \, v^p \chi_{\{v>v_\lambda\}} \quad \mbox{weakly in } \Sigma_\lambda.
	\end{equation*}
	We proceed as in~\ref{step:int2point} of the proof of Theorem~\ref{th:maintheorem-f} and estimate, using~\eqref{eq:gLipgrowth},
	\begin{align*}
		\| \kappa^\ast c_{\lambda}\chi_{\{v>v_\lambda\}} \|_{L^\infty(\Sigma_{\lambda})} & \leq L \,\norma{\kappa^\ast}_{L^\infty(\R^n)} \,\norma{v}^{p-1}_{L^\infty(\Sigma_{\lambda_0})} \le L \, C_0^{p - 1} \,\norma{\kappa}_{L^\infty(\R^n)}, \\
		\| {f_0 \defi(\kappa) \,v^p \chi_{\{v>v_\lambda\}}} \|_{L^n(\Sigma_{\lambda})} & \leq f_0 \defi(\kappa) \,\norma*{v}^p_{L^{n p}(\Sigma_{\lambda_0})} \le f_0 C_0^p \, \Haus^{n - 1}(\sfera)^{\frac{1}{n}} \defi(\kappa),
	\end{align*}
	thanks to~\eqref{eq:bound-v} and the fact that~$\lambda_0 \ge 1$. Applying Theorems~8.17 and~8.25 of~\cite{gt} and recalling~\eqref{eq:L2starboundMPbegins}, we conclude that
	\begin{equation}
		\label{eq:poinbound-v+}
		\| (v-v_\lambda)_+ \|_{L^\infty(\Sigma_\lambda)} \leq C_4 \defi(\kappa) \quad \mbox{for all } \lambda \ge \lambda_0,
	\end{equation}
	for some constant~$C_4 \geq 1$ depending only on~$n$,~$C_0$,~$L$,~$f_0$, and~$\norma*{\kappa}_{L^\infty(\R^n)}$.
	
	Define
	\begin{equation*}
		\Lambda \coloneqq \Big\{ {\lambda \geq 0 \, \big\lvert \, \| (v-v_\mu)_+ \|_{L^\infty(\Sigma_\mu)} \leq C_4 \defi(\kappa) \,\, \mbox{ for every } \mu \geq \lambda} \Big\}.
	\end{equation*}
	Estimate~\eqref{eq:poinbound-v+} shows that~$\Lambda \neq \varnothing$ and that, in particular,~$[\lambda_0, +\infty) \subseteq \Lambda$. Thus,
	\begin{equation*}
		\lambda_\star \coloneqq \inf\Lambda
	\end{equation*}
	is a well-defined real number belonging to~$[0,\lambda_0]$. In the following two steps, we shall prove that
	\begin{equation*}
		\lambda_\star \leq C \defi(\kappa)^\alpha,
	\end{equation*}
	for some universal constants~$\alpha>0$ and~$C \geq 1$. To do it, we argue by contradiction and therefore suppose that
	\begin{equation} \label{eq:lambdastar>Bdefalpha}
		\lambda_\star > \lambda_1 \coloneqq \mathcal{B} \defi(\kappa)^\alpha,
	\end{equation}
	for some~$\alpha>0$ and~$\mathcal{B} \geq 1$ to be later determined. Note that~$\lambda_1 \in \left( 0, \frac{1}{8} \right]$, provided we take
	\begin{equation*}
		\gamma \le \gamma_1 \coloneqq (8 \mathcal{B})^{-\frac{1}{\alpha}}.
	\end{equation*}
	
	
	\subsection{Finding and propagating positivity.}
	\label{step:find-pos}
	
	Let~$r \geq \lambda_0+1$ to be later determined in terms of universal quantities. Clearly, we have that~$B_{r}(0) \cap \Sigma_{\lambda_\star} \neq \varnothing$. Take
	\begin{equation} \label{eq:deltaspec}
		\delta \in \left( 0, \frac{\lambda_1}{4} \right] \!,
	\end{equation}
	with~$\lambda_1$ as in~\eqref{eq:lambdastar>Bdefalpha} and define
	\begin{equation}
	\label{eq:defK2}
		K_\delta \coloneqq \overline{B_{r}(0) \cap \Sigma_{\lambda_\star+ 2 \delta}}.
	\end{equation}
	The set~$K_\delta$ is a compact subset of~$\Sigma_{\lambda_\star}$, containing~$0^{\lambda_\star}$ and disjoint from the hyperplane~$T_{\lambda_\star}$ where~$v-v_{\lambda_\star}$ vanishes. We shall show that~$v_{\lambda_\star}-v$ is strictly positive in a small punctured ball centered at~$0^{\lambda_\star}$ and then propagate this information to the whole~$K_\delta \setminus \{0^{\lambda_\star}\}$.
	
	First, we identify this initial ball. Take a radius
	\begin{equation} \label{r1specs}
		r_1 \in \left( 0, \min \left\{ \frac{\lambda_1}{4}, \left(\frac{ \defi(\kappa)^{\sigma}}{2 C_0^2}\right)^{\!\! \frac{1}{\mu}} \! \lambda_1^{\frac{n-2}{\mu}}, \frac{1}{R_0} \right\} \right] \!.
	\end{equation}
	By this choice,~$B_{2 r_1}(0^{\lambda_\star}) \subseteq \Sigma_{\lambda_\star}$. Furthermore, by exploiting~\eqref{eq:bound-v} and~\eqref{eq:v-below}, for every~$x \in B_{r_1}(0^{\lambda_\star}) \setminus \{ 0^{\lambda_\star} \}$ we get
	\begin{equation} \label{eq:posballforvlambda-v}
		v_{\lambda_\star}(x) - v(x) \geq \frac{1}{C_0} \frac{\defi(\kappa)^{\sigma}}{r_1^\mu} - \frac{C_0}{(2 \lambda_1 - r_1)^{n-2}} \ge \frac{1}{C_0}\frac{\defi(\kappa)^{\sigma}}{r_1^\mu} - \frac{C_0}{\lambda_1^{n-2}} \ge \frac{C_0}{\lambda_1^{\mu}} \ge 1,
	\end{equation}
	where in the last inequality we used that~$C_0 \ge 1$ and~$\lambda_1 \le 1$.
	
	We now proceed to propagate the positivity of~$v_{\lambda_\star} - v$ elsewhere in~$K_\delta$. To this aim, we consider the function~$w \coloneqq \min \{ v_{\lambda_\star}-v+2C_4 \defi(\kappa), 1 \}$. Note that~\eqref{eq:posballforvlambda-v} gives
	\begin{equation} \label{w=qinBr1}
		w = 1 \quad \mbox{in } B_{r_1}(0^{\lambda_\star}).
	\end{equation}
	By arguing as in the second part of~\ref{step:MPstart-kelvin}, it is not hard to see that~$w$ weakly solves
	\begin{equation}
	\label{eq:perRemark}
		\Delta w \leq \defi(\kappa) \big( {C_4 \, \kappa_{\lambda_\star}^\ast c_{\lambda_\star} + f_0 v^p} \big) \chi_{\{w < 1\}} \eqqcolon h \quad \mbox{in } \Sigma_{\lambda_\star}.
	\end{equation}
	
	Let~$y \in K_\delta$ and~$\varrho \in [\delta, 1]$ be such that~$B_{2 \varrho}(y) \subseteq \Sigma_{\lambda_\star}$. Using the upper bound~\eqref{eq:bound-v} along with the fact that~$0 \le c_{\lambda_\star} \chi_{\{w < 1\}} \le p \max \{ v, v_{\lambda_\star} \}^{p - 1} \chi_{\{ v_{\lambda_\star} < v + 1 \}} \le p \, (v + 1)^{p - 1}$, we compute
	\begin{align*}
		\| h \|_{L^n(B_{2 \varrho}(y))} & \le 2 \varrho \, |B_1|^{\frac{1}{n}} \defi(\kappa) \left\{ \frac{f_0 \, C_0^p}{\lambda_\star^{n + 2}} + p \, C_4 \| \kappa \|_{L^\infty(\R^n)} \left( \frac{C_0}{\lambda_\star^{n - 2}} + 1 \right)^{\! p - 1} \right\} \\
		& \le \frac{2^p p \, |B_1|^{\frac{1}{n}} \left(C_0 + 1\right)^p  C_4  \left( {1 + \| \kappa \|_{L^\infty(\R^n)}} \right) \left(1+f_0\right)}{\lambda_1^{n + 2}} \defi(\kappa) \\
		& = 2^p p \, |B_1|^{\frac{1}{n}} \left(C_0 + 1\right)^p C_4 \left( {1 + \| \kappa \|_{L^\infty(\R^n)}}\right) \left(1+f_0\right) \frac{\defi(\kappa)^{1 - (n + 2)\alpha}}{\mathcal{B}^{n + 2}}.
	\end{align*}
	By the weak Harnack inequality, we then get that
	\begin{equation*}
		\label{eq:wHarnack-2}
		\left( \, \dashint_{B_\varrho(y)} w^s \,dx\right)^{\!\! \frac{1}{s}} \leq C \left(\inf_{B_{\varrho}(y)} w + \frac{\defi(\kappa)^{1 - (n + 2)\alpha}}{\mathcal{B}^{n + 2}} \right) \!,
	\end{equation*}
	for every~$s \in (0, 2^\ast / 2)$ and for some constant~$C \ge 1$ depending only on~$n$,~$C_0$,~$f_0$, $\| \kappa \|_{L^\infty(\R^n)}$, and~$s$. Let~$x \in K_\delta$. By iterating this inequality -- with, say,~$s = 2^\ast / 4$ -- along a chain of balls connecting~$B_{r_1}(0^{\lambda_\star})$ with~$B_\delta(x)$, we obtain that
	\begin{equation} \label{eq:iterated_wHi}
		\left( \, \dashint_{B_{r_1}(0^{\lambda_\star})} w^{\frac{2^\ast}{4}} \,dx\right)^{\!\! \frac{4}{2^\ast}} \le C_5^{1 + 2 r} \delta^{-\beta} r_1^{-\beta} \left( \inf_{B_\delta(x)} w + \frac{\defi(\kappa)^{1 - (n + 2)\alpha}}{\mathcal{B}^{n + 2}} \right) \!,
	\end{equation}
	for some constants~$C_5 \ge 1$ and~$\beta > 0$ depending only on~$n$,~$C_0$,~$L$,~$f_0$, and~$\norma*{\kappa}_{L^\infty(\R^n)}$. This can be achieved with the following 3-step procedure:
	\begin{enumerate}[leftmargin=*,label=$(\alph*)$]
		\item first, we connect~$B_{r_1}(0^{\lambda_\star})$ to the ball~$B_1(0^{\lambda_\star} + 2 e_n)$ via a finite sequence of linearly enlarging balls, as in the proof of Lemma~2.2 in~\cite{cicopepo} -- this produces the factor~$r_1^{-\beta}$;
		\item then, we join~$B_1(0^{\lambda_\star} + 2 e_n)$ and~$B_1(x + 2 e_n)$ through a standard chain of balls of radius~$1$ -- this produces an exponential factor~$C_5^{|x - 0^{\lambda_\star}|}$, controlled by~$C_5^{2 r}$;
		\item \label{3-step_c} finally, we go from~$B_1(x + 2 e_n)$ to~$B_\delta(x)$ via linearly shrinking balls, producing the factor~$\delta^{-\beta}$.
	\end{enumerate}
	By taking advantage of~\eqref{w=qinBr1}, we infer from~\eqref{eq:iterated_wHi} that
	\begin{equation*}
		\inf_{B_\delta(x)} w \ge \frac{(\delta r_1)^{\beta}}{C_5^{1 + 2 r} } - \frac{\defi(\kappa)^{1-\left(n+2\right)\alpha}}{\mathcal{B}^{n + 2}}.
	\end{equation*}
	Since this holds for every~$x \in K_\delta$, by taking
	\begin{equation} \label{eq:hugeconstdefs}
		\begin{aligned}
			\delta \coloneqq & \defi(\kappa)^{4 n^2 \alpha}, \\
			r_1 \coloneqq & \,  \left(\frac{ \defi(\kappa)^{\sigma}}{2 C_0^2}\right)^{\!\! \frac{1}{\mu}} \! \lambda_1^{\frac{n-2}{\mu}} =  \left(\frac{\mathcal{B}^{n-2}}{2 C_0^2}\right)^{\!\! \frac{1}{\mu}} \defi(\kappa)^{\frac{\sigma+(n-2)\alpha}{\mu}}, \\
			\alpha \coloneqq & \,\frac{1}{4 n^2 \beta+\left(n-2\right)\mu^{-1} \beta+n+2}, \\
			\sigma \le & \, \frac{\mu}{2 \beta}, \\
			\mathcal{B} \coloneqq & \, \left( 2 C_0^2 \right)^{\! \frac{1}{n - 2}} \left( 4 C_5^{1 + 2 r}\right)^{\!\frac{\mu}{\beta\left(n-2\right)}} \!, 
		\end{aligned}
	\end{equation}
	we obtain that
	\begin{equation*}
		\min_{K_\delta} w \ge \defi(\kappa)^{1-\left(n+2\right)\alpha}.
	\end{equation*}
	Observe that~$r$ will be chosen later in dependence of universal quantities. Moreover, note that our choices for~$\delta$ and~$r_1$ are in compliance with~\eqref{eq:deltaspec} and~\eqref{r1specs} respectively, once we require that
	\begin{equation*}
		\gamma \le \gamma_2 \coloneqq \min \left\{ 4^{-\frac{1}{\left(4n^2-1\right)\alpha}}, \left(\frac{\mathcal{B}^{\mu+2-n}}{4^\mu}\right)^{\!\frac{1}{\sigma+\left(n-2-\mu\right)\alpha}}, \left(\frac{1}{\mathcal{B}^{n-2} R_0^\mu}\right)^{\!\frac{1}{\sigma+\left(n-2\right)\alpha}} \right\} \!.
	\end{equation*}
	From the above estimate and the definition of~$w$, we easily conclude that
	\begin{equation} \label{eq:posit-totransf}
		v_{\lambda_\star}(x) - v(x) \ge \frac{1}{2} \defi(\kappa)^{1-\left(n+2\right)\alpha} \quad \mbox{for every } x \in K_\delta \setminus \{ 0^{\lambda_\star} \},
	\end{equation}
	provided
	\begin{equation*}
		\gamma \le \gamma_3 \coloneqq \left( \frac{1}{4 C_4} \right)^{\!\!\frac{1}{\left(n+2\right)\alpha}} \!.
	\end{equation*}
	
	
	\subsection{Deriving a contradiction.}
	
	For~$\epsilon \in (0, \lambda_1)$, define
	\begin{equation*}
		E_r^\epsilon = \left(\R^n \setminus B_r(0)\right) \cap \Sigma_{\lambda_\star-\epsilon} \quad \mbox{and} \quad S_\delta^\epsilon = B_r(0) \cap \left(\Sigma_{\lambda_\star-\epsilon} \setminus \overline{\Sigma_{\lambda_\star+2\delta}} \right) \!.
	\end{equation*}
	Note that~$E_r^\epsilon$,~$K_\delta$, and~$S_\delta^\epsilon$ form a partition of~$\Sigma_{\lambda_\star-\epsilon}$, up to sets of measure zero -- recall that~$K_\delta$ is defined in~\eqref{eq:defK2}. In order to reach a contradiction, we first observe that, by~\eqref{eq:posit-totransf}, the continuity of~$v$ outside of the origin, and the fact that~$\lim_{x \rightarrow 0} v(x) = +\infty$, there exists a small~$\varepsilon_1 > 0$ such that~$0^{\lambda_\star - \varepsilon} \in K_\delta$ and
	\begin{equation} \label{eq:vlambdastar-eps-vge0}
		v_{\lambda_\star - \varepsilon}(x) - v(x) \ge 0 \quad \mbox{for every } x \in K_\delta \setminus \{ 0^{\lambda_\star - \varepsilon} \},
	\end{equation}
	for every~$\varepsilon \in (0, \varepsilon_1]$. We now test the equation for~$v-v_{\lambda_\star-\epsilon}$ against~$(v - v_{\lambda_\star - \varepsilon})_+ \chi_{\Sigma_{\lambda_\star - \varepsilon}}$ times a suitable cutoff. Arguing as in~\ref{step:MPstart-kelvin}, we deduce that
	\begin{equation} \label{eq:asinstep2}
		\begin{split}
			\int_{\Sigma_{\lambda_\star-\epsilon}} \,\abs*{\nabla (v-v_{\lambda_\star-\epsilon})_+}^2 \, dx &\leq 2 f_0 \defi(\kappa) \int_{\Sigma_{\lambda_\star-\epsilon}} v^p
			(v-v_{\lambda_\star-\epsilon})_+ \, dx \\
			&\quad+ 2 L \norma*{\kappa}_{L^\infty(\R^n)} \int_{\Sigma_{\lambda_\star-\epsilon}} v^{p-1}
			(v-v_{\lambda_\star-\epsilon})_+^2 \, dx.
		\end{split}
	\end{equation}
	We take care of the two terms on the right-hand side. Since, by virtue of~\eqref{eq:vlambdastar-eps-vge0}, it holds~$\left(v-v_{\lambda_\star-\epsilon}\right)_+ =0$ a.e.~in~$K_\delta$, we get
	\begin{equation*}
		\begin{split}
			\int_{\Sigma_{\lambda_\star-\epsilon}} v^p
			(v-v_{\lambda_\star-\epsilon})_+ \, dx &= \int_{\Sigma_{\lambda_\star-\epsilon} \setminus K_\delta} v^p
			(v-v_{\lambda_\star-\epsilon})_+ \, dx \\
			&\leq \norma*{v}^p_{L^{2^\ast} \! \left(\Sigma_{\lambda_\star-\epsilon} \setminus K_\delta\right)} \,\norma*{(v-v_{\lambda_\star-\epsilon})_+}_{L^{2^\ast} \! \left(\Sigma_{\lambda_\star-\epsilon} \setminus K_\delta\right)} \\
			&\leq S^{-1}\norma*{v}^p_{L^{2^\ast} \! \left(\Sigma_{\lambda_\star-\epsilon} \setminus K_\delta\right)} \,\norma*{\nabla (v-v_{\lambda_\star-\epsilon})_+}_{L^{2} \left(\Sigma_{\lambda_\star-\epsilon}\right)},
		\end{split}
	\end{equation*}
	where we have used H\"older and Sobolev inequalities. To deal with the other summand we exploit the decomposition made at the beginning of this step and deduce
	\begin{equation*}
		\begin{split}
			\int_{\Sigma_{\lambda_\star-\epsilon}} v^{p-1}
			(v-v_{\lambda_\star-\epsilon})_+^2 \, dx &= \int_{\Sigma_{\lambda_\star-\epsilon} \setminus K_\delta} v^{p-1}
			(v-v_{\lambda_\star-\epsilon})_+^2 \, dx \\
			&= \int_{E_r^\epsilon} v^{p-1}
			(v-v_{\lambda_\star-\epsilon})_+^2 \, dx + \int_{S_\delta^\epsilon} v^{p-1}
			(v-v_{\lambda_\star-\epsilon})_+^2 \, dx.
		\end{split}
	\end{equation*}
	For the first term we have
	\begin{equation*}
		\int_{E_r^\epsilon} v^{p-1}
		(v-v_{\lambda_\star-\epsilon})_+^2 \, dx \leq S^{-2} \,\norma{v}^{2^\ast-2}_{L^{2^\ast} \! \left(E_r^\epsilon\right)} \int_{\Sigma_{\lambda_\star-\epsilon}} \,\abs*{\nabla (v-v_{\lambda_\star-\epsilon})_+}^2 \, dx,
	\end{equation*}
	exploiting again H\"older and Sobolev inequalities. For the second one we make use of the Poincar\'e inequality in slabs -- with constant~$\mathcal{C}_P (S^\epsilon_\delta) \coloneqq 2n \left(2 \delta+\epsilon\right)^2$ --, as in~\ref{step:contr} of the proof of Theorem~\ref{th:maintheorem-f}. We obtain
	\begin{equation*}
		\begin{split}
			\int_{S_\delta^\epsilon} v^{p-1}
			(v-v_{\lambda_\star-\epsilon})_+^2 \, dx &\leq \norma{v}^{p-1}_{L^\infty (S_\delta^\epsilon)} \int_{S_\delta^\epsilon}
			(v-v_{\lambda_\star-\epsilon})_+^2 \, dx \\
			&\leq
			\norma{v}^{2^\ast - 2}_{L^\infty(S_\delta^\epsilon)} \,\mathcal{C}_P(S^\epsilon_\delta) \int_{\Sigma_{\lambda_\star-\epsilon}} \,\abs*{\nabla (v-v_{\lambda_\star-\epsilon})_+}^2 \, dx.
		\end{split}
	\end{equation*}
	Going back to~\eqref{eq:asinstep2}, we get
	\begin{equation}
	\label{eq:contrad-11}
		\begin{split}
			& \mathcal{A}(r,\epsilon) \, \norma*{\nabla (v-v_{\lambda_\star-\epsilon})_+ }^2_{L^2\left(\Sigma_{\lambda_\star-\epsilon}\right)} 
			\\
			& \hspace{60pt} \leq  2 f_0 S^{-1} \defi(\kappa) \,\norma*{v}^p_{L^{2^\ast} \! \left(\Sigma_{\lambda_\star-\epsilon} \setminus K_\delta\right)} \,\norma*{\nabla (v-v_{\lambda_\star-\epsilon})_+}_{L^{2}\left(\Sigma_{\lambda_\star-\epsilon}\right)},
			\end{split}
	\end{equation}
	where
	\begin{equation*}
		\mathcal{A}(r,\epsilon) \coloneqq 1 - 2 L S^{-2} \, \norma*{\kappa}_{L^\infty(\R^n)} \,\norma{v}^{2^\ast-2}_{L^{2^\ast} \! \left(E_r^\epsilon\right)} - 2 L \,\norma*{\kappa}_{L^\infty(\R^n)} \,\norma{v}^{2^\ast - 2}_{L^\infty(S_\delta^\epsilon)} \,\mathcal{C}_P(S^\epsilon_\delta).
	\end{equation*}
	
	We plan to choose the parameters~$r$ and~$\epsilon$ in such a way that
	\begin{equation} \label{eq:Age12}
		\mathcal{A}(r,\epsilon) \geq \frac{1}{2}.
	\end{equation}
	Indeed, taking advantage of~\eqref{eq:bound-v}, we have that
	\begin{equation*}
		2 L S^{-2} \, \norma*{\kappa}_{L^\infty(\R^n)} \,\norma{v}^{2^\ast-2}_{L^{2^\ast} \! \left(E_r^\epsilon\right)} \le 2 L  S^{-2} \, \norma*{\kappa}_{L^\infty(\R^n)}  \left( \frac{C_0^{2^\ast} \Haus^{n - 1}(\sfera^{n - 1})}{n} \right)^{\!\! \frac{2}{n}} r^{-2} \le \frac{1}{4},
	\end{equation*}
	provided
	\begin{equation*}
		r \geq \frac{4 \sqrt{L \| \kappa \|_{L^\infty(\R^n)}}}{S} \left( \frac{C_0^{2^\ast} \Haus^{n - 1}(\sfera^{n - 1})}{n} \right)^{\!\!\frac{1}{n}} \!.
	\end{equation*}
	On the other hand, recalling the definition of~$\mathcal{C}_P(S^\epsilon_\delta)$, those of~$\lambda_1$ and~$\delta$ -- given in~\eqref{eq:lambdastar>Bdefalpha} and~\eqref{eq:hugeconstdefs} --, and again the upper bound~\eqref{eq:bound-v} on~$v$, we find
	\begin{equation}
	\label{eq:deltalambdatech}
		\begin{aligned}
			2 L \,\norma*{\kappa}_{L^\infty(\R^n)} \,\norma{v}^{2^\ast - 2}_{L^\infty(S_\delta^\epsilon)} \,\mathcal{C}_P(S^\epsilon_\delta) & \le 4 n L \,\norma*{\kappa}_{L^\infty(\R^n)} C_0^{2^\ast - 2} \, \frac{(2 \delta + \varepsilon)^2}{(\lambda_1 - \varepsilon)^4} \\
			& \le 4 n L \,\norma*{\kappa}_{L^\infty(\R^n)} C_0^{2^\ast - 2}  \, \frac{(2 \defi(\kappa)^{4 n^2 \alpha} + \varepsilon)^2}{(\defi(\kappa)^\alpha - \varepsilon)^4} \\
			& \le 2^{10} n L \,\norma*{\kappa}_{L^\infty(\R^n)} C_0^{2^\ast - 2} \defi(\kappa)^{4\left(2n^2-1\right)\alpha} \le \frac{1}{4},
		\end{aligned}
	\end{equation}
	provided~$\varepsilon \le \varepsilon_2 \coloneqq \frac{1}{2} \defi(\kappa)^{4 n^2 \alpha} \leq \frac{\lambda_1}{2}$ and
	\begin{equation*}
		\gamma \le \gamma_4 \coloneqq \left( \frac{1}{2^{12} n L \, \| \kappa \|_{L^\infty(\R^n)} C_0^{2^\ast - 2}} \right)^{\!\! \frac{1}{2\left(4n^2-1\right)\alpha}} \!.
	\end{equation*}
	Accordingly,~\eqref{eq:Age12} holds true and it follows from~\eqref{eq:contrad-11} that
	\begin{equation}
	\label{eq:preSobKelvin}
		\begin{aligned}
			& \norma*{\nabla (v-v_{\lambda_\star-\epsilon})_+ }^2_{L^2\left(\Sigma_{\lambda_\star-\epsilon}\right)} \\
			& \hspace{40pt} \le 4 f_0 S^{-1} \defi(\kappa) \,\norma*{v}^p_{L^{2^\ast} \! \left(\Sigma_{\lambda_\star-\epsilon} \setminus K_\delta\right)} \,\norma*{\nabla (v-v_{\lambda_\star-\epsilon})_+}_{L^{2}\left(\Sigma_{\lambda_\star-\epsilon}\right)}.
			\end{aligned}
	\end{equation}
	
	We shall now reach a contradiction with the definition of~$\lambda_\star$ by showing that
	\begin{equation} \label{eq:Kelvinlastclaim}
		\| (v - v_{\lambda_\star - \varepsilon})_+ \|_{L^\infty(\Sigma_{\lambda_\star - \varepsilon})} \le C_4 \defi(\kappa) \quad \mbox{for every } \varepsilon \in \left( 0, \min \left\{ \varepsilon_1, \varepsilon_2 \right\} \right] \!.
	\end{equation}
	In order to establish claim~\eqref{eq:Kelvinlastclaim}, we first notice that, clearly, one can assume without loss of generality that~$\norma*{\nabla (v-v_{\lambda_\star-\epsilon})_+}_{L^{2}\left(\Sigma_{\lambda_\star-\epsilon}\right)} > 0$, since otherwise~\eqref{eq:Kelvinlastclaim} would hold trivially. Under this supposition,~\eqref{eq:preSobKelvin} and the Sobolev inequality give that
	\begin{equation*}
		\norma*{(v-v_{\lambda_\star-\epsilon})_+ }_{L^{2^\ast} \! \left(\Sigma_{\lambda_\star-\epsilon}\right)} \le 4 f_0 S^{-2} \,\norma*{v}^p_{L^{2^\ast} \! \left(\Sigma_{\lambda_\star-\epsilon} \setminus K_\delta\right)} \defi(\kappa).
	\end{equation*}
	We now estimate the~$L^{2^\ast}$-norm of~$v$ over~$\Sigma_{\lambda_\star - \varepsilon} \setminus K_\delta$. Taking advantage of~\eqref{eq:bound-v}, we have
	\begin{equation}
	\label{eq:vErepsest}
		\| v \|_{L^{2^\ast} \! (E_r^\varepsilon)}^{2^\ast} \le C_0^{2^\ast} \int_{\R^n \setminus B_r(0)} \frac{dx}{|x|^{2 n}} = \frac{C_0^{2^\ast} \Haus^{n - 1}(\sfera^{n - 1})}{n} \, r^{- n}
	\end{equation}
	and, calculating as in~\eqref{eq:deltalambdatech},
	\begin{equation} \label{eq:vSdeltaepsest}
		\begin{aligned}
			\| v \|_{L^{2^\ast} \! (S_\delta^\varepsilon)}^{2^\ast} & \le |S^\varepsilon_\delta| \, \| v \|_{L^\infty(S_\delta^\varepsilon)}^{2^\ast} \le C_0^{2^\ast} (2 r)^{n - 1} \, \frac{2 \delta + \varepsilon}{(\lambda_1 - \varepsilon)^{2 n}} \\
			& \le 2^{3n + 1} C_0^{2^\ast} r^n \defi(\kappa)^{2n\left(2n-1\right)\alpha}.
		\end{aligned}
	\end{equation}
	By recalling the decomposition of~$\Sigma_{\lambda_\star - \varepsilon}$ seen before, we get
	\begin{equation}
	\label{eq:L2starestforv-vlambda-eps}
		\norma*{(v-v_{\lambda_\star-\epsilon})_+ }_{L^{2^\ast} \! \left(\Sigma_{\lambda_\star-\epsilon}\right)} \le C \left( r^{- \frac{n + 2}{2}} + r^{\frac{n + 2}{2}} \, \gamma^{(n^2-4)\alpha} \right) \defi(\kappa),
	\end{equation}
	for some constant~$C \ge 1$ depending only on~$n$,~$C_0$, and~$f_0$.
	
	We are now in position to establish~\eqref{eq:Kelvinlastclaim}. In light of~\eqref{eq:vlambdastar-eps-vge0}, we only need to obtain an~$L^\infty$-estimate for~$(v - v_{\lambda_\star- \varepsilon})_+$ on~$E_r^\varepsilon$ and~$S_\delta^\varepsilon$. To this aim, by arguing as in~\ref{step:MPstart-kelvin} we see that~$(v - v_{\lambda_\star - \varepsilon})_+$ satisfies
	\begin{equation*}
		\Delta (v - v_{\lambda_\star-\epsilon})_+ + \kappa^\ast c_{\lambda_\star - \varepsilon} \chi_{\{ v > v_{\lambda_\star - \varepsilon}\}} (v - v_{\lambda_\star - \varepsilon})_+ \ge - f_0 \defi(\kappa) \, v^p \, \chi_{\{ v > v_{\lambda_\star - \varepsilon}\}} \quad \mbox{in } \Sigma_{\lambda_\star - \varepsilon},
	\end{equation*}
	in the weak sense. For~$r \ge 4$, using~\eqref{eq:gLipgrowth} and~\eqref{eq:bound-v} we estimate
	\begin{equation*}
		\| {\kappa^\ast c_{\lambda_\star - \varepsilon} \chi_{\{ v > v_{\lambda_\star - \varepsilon}\}}} \|_{L^\infty(\R^n \setminus B_{r - 2}(0))} \le L \| \kappa \|_{L^\infty(\R^n)} \| v \|_{L^\infty(\R^n \setminus B_1(0))}^{p - 1} \le L C_0^{p - 1} \| \kappa \|_{L^\infty(\R^n)}
	\end{equation*}
	and, computing as in~\eqref{eq:vErepsest},
	\begin{equation*}
		\| {f_0 \defi(\kappa) v^p \chi_{\{ v > v_{\lambda_\star - \varepsilon}\}}} \|_{L^n(\R^n \setminus B_{r - 2}(0))} \le f_0 C_0^p \defi(\kappa) \left( \frac{\Haus^{n - 1}(\sfera^{n - 1})}{n(n + 1)} \right)^{\!\! \frac{1}{n}} r^{-n - 1}.
	\end{equation*}
	Hence, arguing exactly as in~\ref{step:int2point} of the proof of Theorem~\ref{th:maintheorem-f} we deduce, using~\eqref{eq:L2starestforv-vlambda-eps}, that
	\begin{equation} \label{eq:poinbound-Bepsr}
		\norma*{(v-v_{\lambda_\star-\epsilon})_+ }_{L^\infty(E_r^\varepsilon)} \le C_6 \left( r^{- \frac{n + 2}{2}} + r^{\frac{n + 2}{2}} \, \gamma^{(n^2-4)\alpha} \right) \defi(\kappa),
	\end{equation}
	for some constant~$C_6 \ge 1$ depending only on~$n$,~$C_0$,~$L$,~$f_0$, and~$\| \kappa \|_{L^\infty(\R^n)}$. To obtain the estimate over~$S_\delta^\varepsilon$, we first apply the Alexandrov-Bakelman-Pucci estimate for the Laplacian -- see, e.g.,~\cite[Theorem~9.1]{gt} --, which gives
	\begin{equation} \label{eq:Sepsdeltaest1}
		\begin{aligned}
			& \norma*{(v-v_{\lambda_\star-\epsilon})_+ }_{L^\infty(S_\delta^\varepsilon)} \\
			& \hspace{15pt} \le \sup_{\partial S_\delta^\varepsilon} \, (v-v_{\lambda_\star-\epsilon})_+ \\
			& \hspace{15pt} \quad + C_7 \, \mbox{diam}(S_\delta^\varepsilon) \, \big\| {\left( f_0 \defi(\kappa) \, v^p \, \chi_{\{ v > v_{\lambda_\star - \varepsilon}\}} + \kappa^\ast c_{\lambda_\star - \varepsilon} (v - v_{\lambda_\star - \varepsilon})_+ \right)_+} \big\|_{L^n(S_\delta^\varepsilon)} \\
			& \hspace{15pt} \le C_6 \left( r^{- \frac{n + 2}{2}} + r^{\frac{n + 2}{2}} \, \gamma^{(n^2-4)\alpha} \right) \defi(\kappa) + 2 C_7 \, r \, \Big( {f_0 \defi(\kappa) \, \| v^p \|_{L^n(S_\delta^\varepsilon)}} \\
			& \hspace{15pt} \quad + {\| {\kappa^\ast c_{\lambda_\star - \varepsilon} \chi_{\{ v > v_{\lambda_\star - \varepsilon}\}}} \|_{L^n(S_\delta^\varepsilon)} \, \norma*{(v-v_{\lambda_\star-\epsilon})_+ }_{L^\infty(S_\delta^\varepsilon)}} \Big),
		\end{aligned}	
	\end{equation}
	for some dimensional constant~$C_7 \ge 1$. For the second inequality we have used that the diameter of~$S_\delta^\varepsilon$ is less or equal to~$2 r$ as well as~\eqref{eq:vlambdastar-eps-vge0} and~\eqref{eq:poinbound-Bepsr} to estimate~$(v-v_{\lambda_\star-\epsilon})_+$ on~$\partial S_\delta^\varepsilon$. Through computations similar to~\eqref{eq:vSdeltaepsest}, we get
	\begin{equation*}
		\| v^p \|_{L^n(S_\delta^\varepsilon)} \le 2^{n + 4} \, C_0^p \, r \, \gamma^{(3n-2)\alpha}
	\end{equation*}
	and
	\begin{align*}
		\| {\kappa^\ast c_{\lambda_\star - \varepsilon} \chi_{\{ v > v_{\lambda_\star - \varepsilon}\}}} \|_{L^n(S_\delta^\varepsilon)} & \le L \, \| \kappa \|_{L^\infty(\R^n)} \| v^{p - 1} \|_{L^n(S_\delta^\varepsilon)} \\
		& \le 2^6 L \, \| \kappa \|_{L^\infty(\R^n)} \, C_0^{p - 1} \, r \, \gamma^{4(n-1)\alpha}.
	\end{align*}
	The last inequality implies in particular that
	\begin{equation*}
		2 C_7 \, r \, \| {\kappa^\ast  c_{\lambda_\star - \varepsilon}  \chi_{\{ v > v_{\lambda_\star - \varepsilon}\}}} \|_{L^n(S_\delta^\varepsilon)} \le \frac{1}{2},
	\end{equation*}
	provided
	\begin{equation*}
		\gamma \le \gamma_5 \coloneqq \left(\frac{1}{2^8 L \, C_0^{p - 1} C_7 \, \| \kappa \|_{L^\infty(\R^n)} \, r^2} \right)^{\!\! \frac{1}{4(n - 1)\alpha}} \!.
	\end{equation*}
	By this we can reabsorb the term involving~$\norma*{(v-v_{\lambda_\star-\epsilon})_+ }_{L^\infty(S_\delta^\varepsilon)}$ to the left of~\eqref{eq:Sepsdeltaest1}, obtaining
	\begin{equation*}
		\norma*{(v-v_{\lambda_\star-\epsilon})_+ }_{L^\infty(S_\delta^\varepsilon)} \le C \left( r^{- \frac{n + 2}{2}} + r^{\frac{n + 2}{2}} \, \gamma^{(3n-4)\alpha} \right) \defi(\kappa).
	\end{equation*}
	By combining this with~\eqref{eq:vlambdastar-eps-vge0} and~\eqref{eq:poinbound-Bepsr}, we infer that
	\begin{equation*}
		\norma*{(v-v_{\lambda_\star-\epsilon})_+ }_{L^\infty(\Sigma_{\lambda_\star - \varepsilon})} \le C_8 \left( r^{- \frac{n + 2}{2}} + r^{\frac{n + 2}{2}} \, \gamma^{(3n-4)\alpha} \right) \defi(\kappa),
	\end{equation*}
	for some constant~$C_8 \ge 1$ depending only on~$n$,~$C_0$,~$L$,~$f_0$, and~$\| \kappa \|_{L^\infty(\R^n)}$. Claim~\eqref{eq:Kelvinlastclaim} immediately follows from this, by taking
	\begin{equation} \label{eq:rdef}
			r \coloneqq \max \left\{ \lambda_0 + 1, 4, \frac{4 \sqrt{L \| \kappa \|_{L^\infty(\R^n)}}}{S} \left( \frac{C_0^{2^\ast} \Haus^{n - 1}(\sfera^{n - 1})}{n} \right)^{\!\!\frac{1}{n}}, \left( \frac{2 C_8}{C_4} \right)^{\!\! \frac{2}{n + 2}} \right\}
	\end{equation}
	and
	\begin{equation*}
		\gamma \le \gamma_6 \coloneqq \left( \frac{C_4}{2 C_8} \, \frac{1}{r^{\frac{n + 2}{2}}} \right)^{\!\! \frac{1}{(3n-4)\alpha}} \!.
	\end{equation*}
	Since~\eqref{eq:Kelvinlastclaim} clearly contradicts the definition of~$\lambda_\star$, we conclude that~\eqref{eq:lambdastar>Bdefalpha} is false and therefore that
	\begin{equation*}
		\lambda_\star \le \lambda_1 = \mathcal{B} \defi(\kappa)^\alpha,
	\end{equation*}
	where~$\alpha$ and~$\mathcal{B}$ are defined in~\eqref{eq:hugeconstdefs}, with~$r$ prescribed by~\eqref{eq:rdef}.

	
	\subsection{Almost symmetry in one direction, outside of a small ball.}
	
	So far, we have proven that
	\begin{equation} \label{eq:almsymrecap}
		v(x) - v_\lambda(x) \leq C_4 \defi(\kappa) \quad \mbox{for every } x \in \Sigma_{\lambda} \setminus \{0^\lambda\} \mbox{ and } \lambda \in [\lambda_1, +\infty ).
	\end{equation}
	We shall here deduce from this that
	\begin{equation} \label{eq:almsymclaim}
		v(x', x_n) - v(x', - x_n) \le C_\flat \defi(\kappa)^{\alpha_\flat} \quad \mbox{for every } x \in \R^n \setminus B_{r_\flat}(0),
	\end{equation}
	with~$r_\flat$ as in~\eqref{eq:rflatdef}, for some constants~$C_\flat \ge 1$ and~$\alpha_\flat \in (0, 1)$ depending only on~$n$,~$C_0$,~$L$,~$f_0$, and~$\| \kappa \|_{L^\infty(\R^n)}$, and provided~$\gamma$ is suitably small.
	
	To this end, let~$Q$ be the open cube of center~$0$ and sides parallel to the coordinate axis of length~$2 \ell$, with~$\ell \in (4 \lambda_1, 1]$ to be determined. Let~$x \in \Sigma_0 \setminus Q$, where~$\Sigma_0 = \{ x \in \R^n \,\lvert\, x_n > 0 \}$. Suppose first that~$x \in \Sigma_{4 \lambda_1}$. Using~\eqref{eq:almsymrecap}, we get
	\begin{align*}
		v(x', x_n) - v(x', - x_n) & = v(x) - v_{\lambda_1}(x) + v(x', 2 \lambda_1 - x_n) - v(x', -x_n) \\
		& \le C_4 \defi(\kappa) + 2 \lambda_1 \, \sup_{t \in [0, 2\lambda_1]} \abs*{\partial_n v (x', \cdot - x_n)}.
	\end{align*}
	As~$x \in \Sigma_{4 \lambda_1} \setminus Q$, for every~$t \in [0, 2 \lambda_1]$ we have~$|t - x_n| = x_n - t \ge x_n - 2 \lambda_1 \ge \frac{x_n}{2}$ and therefore~$|(x', t - x_n)| \ge \frac{|x|}{2} \ge \frac{\ell}{2}$, where for the last inequality we took advantage of the fact that~$B_\ell(0) \subseteq Q$. Hence, recalling the gradient bound~\eqref{eq:bound-gradv} we conclude that
	\begin{equation*}
		v(x', x_n) - v(x', - x_n) \le C_4 \defi(\kappa) + 2^{n + 2} C_2 \, \mathcal{B} \, \frac{\defi(\kappa)^\alpha}{\ell^n} \le \left( C_4 + 1 \right) \defi(\kappa)^{\frac{\alpha}{2}},
	\end{equation*}
	provided we take
	\begin{equation*}
		\ell \coloneqq 4 \left( C_2 \, \mathcal{B} \right)^{\frac{1}{n}} \defi(\kappa)^{\frac{\alpha}{2 n}},
	\end{equation*}
	\begin{equation*}
		\gamma \le \gamma_7 \coloneqq \min\left\{\left( \frac{1}{4^n \, C_2 \, \mathcal{B}} \right)^{\!\! \frac{2}{\alpha}},\left(\frac{1}{2 \mathcal{B}}\right)^{\!\frac{2n}{\alpha\left(2n-1\right)}}\right\} \!.
	\end{equation*}
	If, on the other hand,~$x \in \Sigma_0 \setminus \left( Q \cup \Sigma_{4 \lambda_1} \right)$, then~$x_n \in (0, 4 \lambda_1]$ and~$|x'| \ge \ell$. Consequently, by taking advantage of~\eqref{eq:bound-gradv} once again we obtain
	\begin{equation*}
		v(x', x_n) - v(x', - x_n) \le 2 x_n \, \norma*{\partial_n(x', \cdot)}_{L^\infty([- x_n, x_n])} \le 16 \, C_2 \, \frac{\lambda_1}{\ell^{n}} \le \defi(\kappa)^{\frac{\alpha}{2}}.
	\end{equation*}
	Overall, we have established the upper bound claimed in~\eqref{eq:almsymclaim}, with constants~$\alpha_\flat = \frac{\alpha}{2}$ and~$C_\flat \coloneqq \max \left\{ C_4 + 1, 4 \sqrt{n} \left( C_2 \, \mathcal{B} \right)^{\frac{1}{n}} \right\}$. By reproducing Steps~2-4 along the direction~$- e_n$ instead of~$e_n$, one clearly gets an analogous -- negative -- lower bound. As a result,
	\begin{equation*}
		|v(x', x_n) - v(x', - x_n)| \le C_\flat \defi(\kappa)^{\alpha_\flat} \quad \mbox{for every } x \in \R^n \setminus B_{r_\flat}(0),
	\end{equation*}
	with~$r_\flat$ as in~\eqref{eq:rflatdef}. Up to a rotation, this is precisely claim~\eqref{eq:mainclaim-Kelvin}.
	
	
	\subsection{Return to the original function.}
	
	In this step we shall go back to the original function~$u$ and prove that
	\begin{equation*}
		\left| u(x) - u \big( {x - 2 (\omega \cdot x) \omega} \big) \right| \le C_\sharp \defi(\kappa)^{\alpha_\sharp} \quad \mbox{for every } x \in \R^n \mbox{ and } \omega \in \sfera^{n - 1},
	\end{equation*}
	for some constants~$C_\sharp \ge 1$ and~$\alpha_\sharp \in (0, 1)$ depending only on universal quantities, provided~$\gamma$ is suitably small. Clearly, this inequality is equivalent to~\eqref{eq:alm_sym_thm1}.

	As customary, up to a rotation we may assume that~$\omega = e_n$ and therefore limit ourselves to only checking that
	\begin{equation} \label{eq:finalKelvclaim}
		|u(x', x_n) - u(x', - x_n)| \le C_\sharp \defi(\kappa)^{\alpha_\sharp} \quad \mbox{for every } x \in \R^n.
	\end{equation}
	First of all, from~\eqref{eq:mainclaim-Kelvin} we easily deduce that
	\begin{equation} \label{eq:mainclaim-Kelvinreboot}
		|u(x', x_n) - u(x', - x_n)| \le C_\flat \, \frac{\defi(\kappa)^{\alpha_\flat}}{|x|^{n - 2}} \quad \mbox{for every } x \in B_{R_\flat}(0) \setminus \{ 0 \},
	\end{equation}
	with~$R_\flat \coloneqq r_\flat^{-1} = \left( C_\flat \defi(\kappa)^{\frac{\alpha_\flat}{n}} \right)^{\!-1}$. Note that this bound is not particularly meaningful near the origin. Nevertheless, it yields~\eqref{eq:finalKelvclaim} with~$\alpha_\sharp = \alpha_\flat$ and for~$x \in B_{R_\flat}(0) \setminus B_{r_\sharp}(0)$, for any universal radius~$r_\sharp \in (0, 1]$ -- of course, the constant~$C_\sharp$ would then depend on~$r_\sharp$. We now address the validity of~\eqref{eq:finalKelvclaim} for~$x \in B_{r_\sharp}(0)$ and~$x \in \R^n \setminus B_{R_\flat}(0)$.
	
	To obtain the estimate in~$B_{r_\sharp}(0)$, we observe that the function~$w(x) \coloneqq u(x', x_n) - u(x', -x_n)$ is of class~$C^1$ in~$\R^n$,~$w \in W^{1,2}_{\loc}(\R^n)$, and satisfies
	\begin{equation*}
		-\Delta w \le \| \kappa \|_{L^\infty(\R^n)} \, \hat{c} \, w + f_0 \defi(\kappa) u^p \quad \mbox{in } \R^n,
	\end{equation*}
	where
	\begin{equation*}
		\hat{c}(x) \coloneqq \begin{dcases}
			\frac{f(u(x', x_n)) - f(u(x', -x_n))}{u(x', x_n) - u(x', -x_n)} & \quad \mbox{if } u(x', x_n) \ne u(x', -x_n), \\
			0 & \quad \mbox{if } u(x', x_n) = u(x', -x_n).
		\end{dcases}
	\end{equation*}
	Thus, we can apply the Alexandrov-Bakelman-Pucci estimate in the domain~$B_{r_\sharp}(0)$, getting
	\begin{equation*}
		\sup_{B_{r_\sharp}(0)} w \leq \sup_{\partial B_{r_\sharp}(0)} w_+ + C_9 \, r_\sharp^2 \left( L \, C_0^{p-1} \norma{\kappa}_{L^\infty(\R^n)} \sup_{B_{r_\sharp}(0)} w_+ + f_0 \, C_0^p \defi(\kappa) \right) \!,
	\end{equation*}
	for some dimensional constant~$C_9 \ge 1$. By choosing
	\begin{equation*}
		r_\sharp \coloneqq \min \left\{ \sqrt{\frac{1}{2 L \, C_0^{p - 1} C_9 \| \kappa \|_{L^\infty(\R^n)}}}, 1 \right\},
	\end{equation*}
	we can reabsorb the first term inside the round brackets to the left-hand side and infer, using~\eqref{eq:mainclaim-Kelvinreboot}, that
	\begin{equation*}
		\sup_{B_{r_\sharp}(0)} w_+ \le 2 \left( \sup_{\partial B_{r_\sharp}(0)} w_+ + f_0 \, C_0^p \, C_9 \defi(\kappa) \right) \le 2 \left( \frac{C_\flat}{r_\sharp^{n - 2}} + f_0 \, C_0^p \, C_9 \right) \defi(\kappa)^{\alpha_\flat}.
	\end{equation*}
	Since, by interchanging~$u(x)$ and~$u(x', -x_n)$, an analogous bound on~$w_-$ can be obtained, we conclude that~\eqref{eq:finalKelvclaim} holds true for every~$x \in B_{r_\sharp}(0)$.
	
	We are left with establishing~\eqref{eq:finalKelvclaim} for~$x \in \R^n \setminus B_{R_\flat}(0)$. This is indeed an immediate consequence of the decay assumption~\eqref{eq:u-below}, whose application yields
	\begin{equation} \label{eq:uevenfaraway}
	\begin{aligned}
		|u(x', x_n) - u(x', -x_n)| & \le \max \big\{ {u(x', x_n), u(x', -x_n)} \big\} \le \frac{1}{\defi(\kappa)^{\sigma} |x|^\nu} \\
		& \le \frac{1}{R_\flat^\nu \defi(\kappa)^{\sigma}} = C_\flat^\nu \defi(\kappa)^{\frac{\alpha_\flat \nu}{ n} - \sigma},
	\end{aligned}
	\end{equation}
	for every~$x \in \R^n \setminus B_{R_\flat}(0)$. Note that~\eqref{eq:u-below} can be applied to such~$x$'s provided that~$R_\flat \ge R_0$, which holds true if we restrict ourselves to
	\begin{equation*}
		\gamma \leq \gamma_8 \coloneqq \left( \frac{R_0}{C_\flat} \right)^{\frac{n}{\alpha_\flat}}.
	\end{equation*}
	By taking
	\begin{equation*}
		\sigma \le \frac{\alpha_\flat \nu}{2 n},
	\end{equation*}
	we deduce from inequality~\eqref{eq:uevenfaraway} that
	\begin{equation*}
		|u(x', x_n) - u(x', -x_n)| \le C_\flat^\nu \defi(\kappa)^{\frac{\alpha_\flat \nu}{2 n}} \quad \mbox{for every } x \in \R^n \setminus B_{R_\flat}(0).
	\end{equation*}
	As a result, claim~\eqref{eq:finalKelvclaim} and the proof of estimate~\eqref{eq:alm_sym_thm1} is complete.
	

	\subsection{Almost radial symmetry in the $C^2$ sense}
	In this conclusive step, we establish the validity of the~$C^2$-estimate~\eqref{eq:alm_symC2_thm2}. This is a rather straightforward consequence of~\eqref{eq:alm_sym_thm1} along with the Schauder estimates for the Laplacian -- see, e.g., Corollary~2.16 of~\cite{rosotonreal}. Nevertheless, we will include all the details for the benefit of the reader.

	Assume for the remainder of the proof that~$\kappa \in C^{0, \tau}(\R^n)$. First of all, similarly to what we did in~\ref{step:prel-obs} of Theorem~\ref{th:maintheorem-f}, we can suppose that~$\defi(\kappa) \le 1$. Indeed, when~$\defi(\kappa) > 1$ estimate~\eqref{eq:alm_symC2_thm2} follows at once from the fact that $\| u \|_{C^2(\R^n)} \le C \left( 1 + \left[\kappa\right]_{C^{0,\tau}(\R^n)}\right)$, consequence of the Schauder theory applied to equation~\eqref{eq:mainprob-f}. Here and in the remainder of the proof,~$C$ indicates a constant larger than~$1$ depending only on universal quantities and~$\tau$.
	
	Let~$\Theta$ be any rotation with center the origin and write~$u_\Theta(x) \coloneqq u(\Theta x)$ for~$x \in \R^n$. Since~$|\Theta x| = |x|$, by~\eqref{eq:alm_sym_thm1} we have
	\begin{equation}
	\label{eq:u-uT-Linf}
		\norma*{u-u_\Theta}_{L^\infty(\R^n)} \leq C \defi(\kappa)^\vartheta.
	\end{equation}
	We first claim that
	\begin{equation}
	\label{eq:u-uT-C1}
		\norma*{u-u_\Theta}_{C^1(\R^n)} \leq C \defi(\kappa)^\vartheta.
	\end{equation}
	To this aim, observe that~$u - u_\Theta$ solves
	\begin{equation} \label{eq:eqforu-uTheta}
		\Delta (u-u_\Theta) = \kappa_\Theta f(u_\Theta) - \kappa f(u) \quad \mbox{in } \R^n,
	\end{equation}
	with~$\kappa_\Theta(x) \coloneqq \kappa(\Theta x)$ for~$x \in \R^n$. Since, thanks to~\eqref{eq:C0_def},~\eqref{eq:ip-f-sub},~\eqref{eq:ip-f-lip-monot}, and~\eqref{eq:u-uT-Linf},
	\begin{equation} \label{eq:Linftykf-kfest}
	\begin{aligned}
		& \norma*{\kappa_\Theta f(u_\Theta) - \kappa f(u)}_{L^\infty(\R^n)} \\
		& \hspace{20pt} \le \norma*{\kappa_\Theta}_{L^\infty(\R^n)} \norma*{f(u_\Theta) - f(u)}_{L^\infty(\R^n)} + \norma*{f(u)}_{L^\infty(\R^n)} \norma*{\kappa_\Theta - \kappa}_{L^\infty(\R^n)} \\
		& \hspace{20pt} \le L C_0^{p - 1} \norma*{\kappa}_{L^\infty(\R^n)} \norma*{u_\Theta - u}_{L^\infty(\R^n)} + f_0 C_0^p \norma*{\kappa_\Theta - \kappa}_{L^\infty(\R^n)} \le C \defi(\kappa)^\vartheta,
	\end{aligned}
	\end{equation}
	by standard~$C^1$-estimates for the Laplacian -- for instance, a suitably translated version of Proposition~2.18 of~\cite{rosotonreal} -- we easily deduce the validity of claim~\eqref{eq:u-uT-C1}.

	We now claim that, for every~$\varepsilon \in (0, 1)$,
	\begin{equation}
	\label{eq:u-uT-C2}
		\norma*{u-u_\Theta}_{C^2(\R^n)} \leq C_\varepsilon \left( 1 + \left[\kappa\right]_{C^{0,\tau}(\R^n)}\right) \defi(\kappa)^{(1 - \varepsilon) \vartheta},
	\end{equation}
	with~$C_\varepsilon \ge 1$ possibly depending on~$\varepsilon$ as well. We plan to deduce~\eqref{eq:u-uT-C2} from the Schauder theory applied to equation~\eqref{eq:eqforu-uTheta}. In order to do this, we need an estimate for the H\"older norm of its right-hand side. For, let~$\sigma \in (0, \tau)$ and compute
	\begin{align*}
		& [\kappa_\Theta f(u_\Theta) - \kappa f(u)]_{C^{0, \sigma}(\R^n)} \\
		& \hspace{20pt} \le [(\kappa_\Theta - \kappa) f(u_\Theta)]_{C^{0, \sigma}(\R^n)} + [\kappa (f(u_\Theta) - f(u))]_{C^{0, \sigma}(\R^n)} \\
		& \hspace{20pt} \le \norma*{\kappa_\Theta - \kappa}_{L^\infty(\R^n)} [f(u_\Theta)]_{C^{0, \sigma}(\R^n)} + \norma*{f(u_\Theta)}_{L^\infty(\R^n)} [\kappa_\Theta - \kappa]_{C^{0, \sigma}(\R^n)} \\
		& \hspace{20pt} \quad\, + \norma*{\kappa}_{L^\infty(\R^n)} [f(u_\Theta) - f(u)]_{C^{0, \sigma}(\R^n)} + \norma*{f(u_\Theta) - f(u)}_{L^\infty(\R^n)} [\kappa]_{C^{0, \sigma}(\R^n)}.
	\end{align*}
	Using the interpolation inequality
	\begin{equation*}
		[v]_{C^{0, \alpha}(\R^n)} \le 2 \, \norma*{v}_{L^\infty(\R^n)}^{1 - \frac{\alpha}{\beta}} \, [v]_{C^{0, \beta}(\R^n)}^{\frac{\alpha}{\beta}} \quad \mbox{for every } v \in C^{0, \beta}(\R^n) \mbox{ and } 0 < \alpha < \beta \le 1,
	\end{equation*}
	along with~\eqref{eq:C0_def},~\eqref{eq:ip-f-sub},~\eqref{eq:ip-f-lip-monot},~\eqref{eq:C1boundonu-1}, and~\eqref{eq:u-uT-Linf}, the previous estimate yields
	\begin{align*}
		& [\kappa_\Theta f(u_\Theta) - \kappa f(u)]_{C^{0, \sigma}(\R^n)} \\
		& \hspace{20pt} \le C \left( \defi(\kappa) + [\kappa]_{C^{0, \tau}(\R^n)}^{\frac{\sigma}{\tau}} \defi(\kappa)^{1 - \frac{\sigma}{\tau}} + \defi(\kappa)^{\left( 1 - \frac{\sigma}{\tau} \right) \vartheta} + [\kappa]_{C^{0, \tau}(\R^n)}^{\frac{\sigma}{\tau}} \defi(\kappa)^{\vartheta} \right) \\
		& \hspace{20pt} \le C \left( 1 + [\kappa]_{C^{0, \tau}(\R^n)}^{\frac{\sigma}{\tau} } \right) \defi(\kappa)^{\left( 1 - \frac{\sigma}{\tau} \right) \vartheta}.
	\end{align*}
	By virtue of this and the~$L^\infty$-bound~\eqref{eq:Linftykf-kfest}, Schauder~$C^{2, \sigma}$-estimates applied to equation~\eqref{eq:eqforu-uTheta} immediately lead to the claim~\eqref{eq:u-uT-C2}, provided we take~$\sigma = \varepsilon \tau$.
	
	This concludes the proof of the~$C^2$-estimate~\eqref{eq:alm_symC2_thm2} and, with it, that of Theorem~\ref{th:maintheorem-kelvin}.

	\begin{remark}
	\label{rem:monot-f}
	If~$f$ is not assumed to be non-decreasing -- i.e., if assumption~\eqref{eq:ip-f-lip-monot} is replaced by the weaker requirement that~$|f(u_2) - f(u_1)| \le L u_2^{p - 1} (u_2 - u_1)$ for every~$0 \le u_1 < u_2 \le C_0$ --, the proof just displayed works almost identically, apart from a small modification needed in~\ref{step:find-pos} which, however, seems to have a non-negligible impact on the statement of Theorem~\ref{th:maintheorem-kelvin}.

	Indeed, if the monotonicity of~$f$ is dropped, the coefficient~$c_\lambda$ defined in~\eqref{eq:clambdadefkelv} no longer has a sign and, as a result, one needs to account for the presence of the zero-th order term~$\kappa_{\lambda_\star}^\ast (c_{\lambda_\star})_- \, w$ on the left-hand side of~\eqref{eq:perRemark}. This creates some issues when applying the weak Harnack inequality of, say, Theorem~8.18 in~\cite{gt}, since the coefficient of said zero-th order term has size comparable to~$|x|^{-4}$ at a point $x \in \Sigma_{\lambda_\star + 2 \delta} \setminus \Sigma_1$. This size is manageable as long as we only consider balls~$B_\varrho(x)$ of radius~$\varrho \leq C \, |x|^2$. In order to deal with this limitation, one needs to modify point~\ref{3-step_c} of the 3-step procedure displayed on page~\pageref{3-step_c}, for instance by using there a family of balls having constant radius~$\delta$ in order to go from~$x$ to, say,~$x + e_n$, and then linearly expanding balls to go from~$x + e_n$ to~$x + 2 e_n$. This would, however, produce on the right-hand side of~\eqref{eq:iterated_wHi} a factor having exponential dependence on~$\delta$, instead of polynomial -- a fact that would ultimately impact negatively estimates~\eqref{eq:alm_sym_thm1} and~\eqref{eq:alm_symC2_thm2}, making them having logarithmic dependence on the deficit of~$\kappa$, like~\eqref{eq:mainLinftyest} and~\eqref{eq:mainD12est_thm1.1} of Theorem~\ref{th:maintheorem-f}.
	\end{remark}

	
	\section*{Acknowledgments} 
	\noindent The authors have been partially supported by the “Gruppo Nazionale per l'Analisi Matematica, la Probabilità e le loro Applicazioni” (GNAMPA) of the “Istituto Nazionale di Alta Matematica” (INdAM, Italy) and by the Research Project of the Italian Ministry of University and Research (MUR) Prin 2022 “Partial differential equations and related geometric-functional inequalities”, grant number 20229M52AS\_004. The second author has also been supported by the Spanish grants PID2021-123903NB-I00 and RED2022-134784-T funded by MCIN/AEI/10.13039/501100011033 and by ERDF “A way of making Europe”.


\end{document}